\documentclass[11pt]{amsart}
\usepackage{latexsym,amsmath,amssymb,amsfonts,amscd,graphics,appendix,amsxtra}
\usepackage[mathscr]{eucal}
\usepackage[all]{xypic}
\usepackage[normalem]{ulem}
\usepackage{fullpage}

\usepackage{xr-hyper}

\usepackage{hyperref}
\usepackage[usenames,dvipsnames]{xcolor}

\newtheorem{theorem}{Theorem}[section]
\newtheorem{proposition}[theorem]{Proposition}
\newtheorem{lemma}[theorem]{Lemma}

\newtheorem{corollary}[theorem]{Corollary}

\newtheorem{D}[theorem]{Definition}
\newenvironment{definition}{\begin{D} \rm }{\end{D}}
\newtheorem{R}[theorem]{Remark}
\newenvironment{remark}{\begin{R}\rm }{\end{R}}
\newtheorem{E}[theorem]{Example}
\newenvironment{example}{\begin{E}\rm }{\end{E}}

\newcommand{\Dis}{\displaystyle}
\numberwithin{equation}{section}

\def\Zee{\mathbb{Z}}
\def\Q{\mathbb{Q}}

\def\Cee{\mathbb{C}}
\def\Pee{\mathbb{P}}

\def\Ker{\operatorname{Ker}}
\def\Coker{\operatorname{Coker}}
\def\Hom{\operatorname{Hom}}
\def\Ext{\operatorname{Ext}}

\def\Gr{\operatorname{Gr}}

\def\im{\operatorname{Im}}

\def\Spec{\operatorname{Spec}}

\def\Proj{\operatorname{Proj}}
\def\scrO{\mathcal{O}}
\def\spcheck{^{\vee}}
\def\hX{\widehat{X}}
\def\hY{\widehat{Y}}

\title{Deformations of singular Fano and Calabi-Yau varieties} 

\begin{document}
\author[R. Friedman]{Robert Friedman}
\address{Columbia University, Department of Mathematics, New York, NY 10027}
\email{rf@math.columbia.edu}
\author[R. Laza]{Radu Laza}
\address{Stony Brook University, Department of Mathematics, Stony Brook, NY 11794}
\email{radu.laza@stonybrook.edu}

\begin{abstract}
The goal of this paper is to generalize results concerning the deformation theory of Calabi-Yau and Fano threefolds with isolated  hypersurface singularites, due to the first author, Namikawa and  Steenbrink. In particular, under the assumption of  terminal singularities, Namikawa proved smoothability  in the Fano case and also for generalized Calabi-Yau threefolds assuming that a certain topological first order condition is satisfied. In the case of dimension $3$, we extend their results by, among other things,  replacing terminal with canonical. In higher dimensions, we identify a class of singularities to which our method applies. A surprising aspect of our study is the role played by the higher Du Bois and higher rational singularities.  Among other  deformation theoretic results in higher dimensions, we obtain smoothing results for generalized Fano varieties whose singularities are not $1$-rational, and for generalized Calabi-Yau varieties whose singularities are not $1$-rational but are $1$-Du Bois under a topological condition on the links which is similar to the first order obstruction in dimension $3$. 
\end{abstract} 

\thanks{Research of the second author is supported in part by NSF grant DMS-2101640.}    

\bibliographystyle{amsalpha}
\maketitle

 \section*{Introduction}

This paper is the first in a series investigating the deformation theory  of Fano and Calabi-Yau varieties with certain singularities \cite{FL22b}, \cite{FL23}, \cite{FL24}. We begin by describing some of the relevant history. Let $Y$ be a complex analytic surface with at worst rational double points (i.e.\ canonical singularities) such that, if $\omega_Y$ is the dualizing sheaf, then  either $\omega_Y^{-1}$ is ample or $\omega_Y \cong \scrO_Y$ and $H^1(Y; \scrO_Y) = 0$. In other words, if $\hY$ is the minimal resolution of $Y$, then  either $-K_{\hY}$ is nef and big ($\hY$ is a generalized del Pezzo surface) or $\hY$ is a $K3$ surface. Burns-Wahl \cite{BurnsWahl} showed   that, in either case, the deformations of $Y$ are versal for the singularities, and in particular that $Y$ is smoothable.
In dimension $3$, following an idea due to Clemens, the first author  showed the following \cite{F}: Let $Y$ be a compact analytic threefold with only ordinary double points which is the analogue of a Calabi-Yau threefold, in the sense that $\omega_Y\cong \scrO_Y$ and  $h^1(\scrO_Y) = 0$.  Let $Y'\to Y$ be a small resolution, with exceptional curves $C_1, \dots, C_r$, and let $[C_i]$ be the class of $C_i$ in $H^4(Y';\Cee)$. Then $Y'$ is smoothable provided 1) the classes $[C_1], \dots, [C_r]$ generate $H^2(Y';\Omega^2_{Y'})$ and 2) there exists a relation $a_1[C_1]+  \cdots +  a_r[C_r] =0 \in H^4(Y';\Cee)$ with $a_i\neq 0$ for every $i$. Subsequently, independent work of Kawamata \cite{Kawamata}, Ran \cite{Ran}, and Tian \cite{Tian} showed that assumption 1) is not necessary provided that  there exists some resolution $\hY \to Y$ which is K\"ahler. The key  point in removing this assumption is a generalization of the unobstructedness theorem for smooth Calabi-Yau manifolds, in any dimension:

\begin{theorem}[Kawamata, Ran, Tian]\label{thmK} Let $Y$ be a compact analytic variety with only ordinary double points such that $\omega_Y\cong \scrO_Y$ and there exists a resolution $\hY \to Y$ which is K\"ahler. Then the deformations of $Y$ are unobstructed. 
\end{theorem}

\begin{remark}Here, the hypothesis that there is a K\"ahler resolution of $Y$ could be replaced by: there exists a resolution of $Y$ satisfying the $\partial\bar\partial$-lemma. (Compare Remark~\ref{ddbarremark}.)
\end{remark}

This theorem was then generalized by Namikawa \cite{namtop}, in the case of dimension three, to allow for more complicated singularities:

\begin{theorem}[Namikawa]\label{thmN} Let $Y$ be a compact analytic threefold, all of whose singularities are isolated  canonical   local complete intersection (lci) singularities,  such that $\omega_Y\cong \scrO_Y$ and $h^1(\scrO_Y) = 0$, and such that there exists a resolution $\hY \to Y$ which is K\"ahler. Then the deformations of $Y$ are unobstructed.
\end{theorem}

Subsequently, in a series of papers both separate and joint \cite{namtop}, \cite{NS}, \cite{namstrata}, \cite{SteenbrinkDB},  Namikawa and Steenbrink significantly generalized the result of \cite{F}. Among other results, they  proved the following:

\begin{theorem}\label{Thm-NS} Let $Y$ be a compact analytic threefold, all of whose singularities are isolated  canonical  hypersurface singularities,  such $\omega_Y\cong \scrO_Y$ and $h^1(\scrO_Y) = 0$, and such that there exists a resolution $\hY \to Y$ which is K\"ahler.
\begin{enumerate} \item[\rm(i)]  \cite{NS} $Y$ can be deformed to a generalized Calabi-Yau threefold whose only singularities are ordinary double points.  
 \item[\rm(ii)] \cite{namstrata} Suppose that all singular points of $Y$ are in addition terminal singularities. Let $Y' \to Y$ be a small resolution of  $Y$ at the ordinary double points and an isomorphism elsewhere, and, for each $x\in Y$ which is an ordinary double point, let $C_x$ be the corresponding exceptional curve in $Y'$. If there exists a relation $\sum_xa_x[C_x] = 0$ in $H_2(Y')$ with $a_x\in \Cee$ and  $a_x\neq 0$ for every $x$,    then $Y$ is smoothable. 
 \end{enumerate}
\end{theorem}

There is also a related result of Gross \cite[Theorem 3.8]{gross_defcy} in case $Y$ has a crepant resolution:

\begin{theorem}[Gross]  Let $Y$ be a compact analytic threefold, all of whose singularities are isolated  canonical  lci singularities,  such $\omega_Y\cong \scrO_Y$  and such that there exists a resolution $\hY \to Y$ which is K\"ahler and $K_{\hY}\cong \scrO_{\hY}$.
Then  $Y$ can be deformed to a generalized Calabi-Yau threefold whose only singularities are ordinary double points.  More precisely, there exists a small deformation of $Y$ which smooths all of the singular points which are not ordinary double points and which is locally trivial at the ordinary double points.
\end{theorem}

In the case of singular Fano threefolds, the first author in \cite{F} showed that, if $Y$ is a compact analytic threefold such that all of the singularities of $Y$ have a small resolution, $\omega_Y^{-1}$ is ample, and there exists a smooth Cartier divisor $H$ on $Y$ such that $\omega_Y= \scrO_Y(-H)$, then the deformations of $Y$ are unobstructed. Moreover, if all singular points on $Y$ are ordinary double points, then $Y$ is smoothable. Note that, in contrast to the Calabi-Yau case, there is no topological condition on the exceptional curves on the small resolution. This result was generalized by Namikawa \cite{NamikawaF}:

\begin{theorem}[Namikawa] Let $Y$ be a compact analytic Gorenstein threefold such that  $\omega_Y^{-1}$ is ample.  If the  singular points of $Y$ are  isolated canonical local complete intersection  singularities,  then the deformations of $Y$ are unobstructed. If the singularities of $Y$ are  Gorenstein terminal singularities (hence in particular are canonical hypersurface singularities), then $Y$ is smoothable.
\end{theorem}

(See also e.g.\ \cite{Sano} for related results in the $\Q$-Fano threefold case.) 

\medskip

The goal of this paper is to understand and strengthen the above results of Namikawa-Steenbrink and Namikawa in the case of dimension $3$ and to find a natural generalization to higher dimensions. To start, in dimension $3$, in the Fano case, we can replace the hypothesis of terminal singularities with canonical hypersurface  singularities, as follows: 

\begin{theorem}\label{thm-sing-dim3} Let $Y$ be a compact analytic threefold whose singular points are isolated canonical hypersurface  singularities and  such that  $\omega_Y^{-1}$ is ample.  Then $Y$ is smoothable.
\end{theorem}

Still in dimension $3$, in the Calabi-Yau case, we sharpen Theorem \ref{Thm-NS} as following: 

\begin{theorem} Let $Y$ be a compact analytic threefold, all of whose singularities are isolated  canonical  hypersurface singularities,  such that $\omega_Y\cong \scrO_Y$ and $h^1(\scrO_Y) = 0$, and such that there exists a resolution $\hY \to Y$ which satisfies the $\partial\bar\partial$-lemma.
\begin{enumerate} \item[\rm(i)]   There exists a small deformation of $Y$ which smooths all of the singular points which are not ordinary double points and which is locally trivial at the ordinary double points.
 \item[\rm(ii)]   Let $Y' \to Y$ be a small resolution of  $Y$ at the ordinary double points and an isomorphism elsewhere, and, for each $x\in Y$ which is an ordinary double point, let $C_x$ be the corresponding exceptional curve in $Y'$. If there exists a relation $\sum_xa_x[C_x] = 0$ in $H_2(Y')$ with $a_x\in \Cee$ and  $a_x\neq 0$ for every $x$,   then $Y$ is smoothable. 
 \end{enumerate}
\end{theorem}

Before we start our discussion of the   strategy of the proofs of the results above and their generalizations to higher dimensions, let us note that canonical singularities are the natural class of singularities relevant to degenerations of Calabi-Yau and Fano manifolds. Namely, as a consequence of the minimal model program, a one parameter degeneration of Calabi-Yau manifolds can be assumed to be minimal dlt \cite{fujino-ss}. For such a normalized degeneration, the top holomorphic form is preserved $\iff$ the central fiber has canonical singularities \cite{KLSV}, \cite{KLS}. From a  differential geometric perspective, Calabi-Yau varieties with canonical singularities occur as the limits of smooth Calabi-Yau manifolds which are at finite distance with respect to  the Weil-Petersson metric \cite{Tosatti}, \cite{Wang}, \cite{yzhang}. Similarly, it is known that $K$-semistable limits of Fano manifolds have at worst klt singularities \cite{Odaka}, \cite{DonaldsonSun}, and thus are canonical, under the Gorenstein assumption used throughout this paper. Note that under the Gorenstein  assumption, a singularity is canonical if and only if it is  rational, and we will use the two notions interchangeably.  Finally, in order to get a good hold on the local deformation theory, we will usually assume that the singularities are isolated hypersurface singularities (although some of our results hold more generally in the local complete intersection case).

  Let $(X,x)$ be the germ of an isolated  canonical  hypersurface singularity with $\dim X \geq 3$, and let $\pi\colon \hX \to X$ be a good  resolution, i.e.\ such that the exceptional divisor $E =\pi^{-1}(x)$ has simple normal crossings. We will usually replace the germ $(X,x)$ by a good Stein representative. Recall that a \textsl{first order deformation} of $(X,x)$ is a deformation over the dual numbers $\Spec \Cee[\epsilon]$, where $\epsilon^2 =0$, and that these are classified by a finite-dimensional vector space $T^1_{X,x}$, where $T^1_X$ is the sheaf $\mathit{Ext}^1(\Omega^1_X; \scrO_X)$. Likewise, let $Y$ be a compact analytic space with at worst isolated  canonical  hypersurface singularities, and similarly let $\pi \colon \hY \to Y$ be a good resolution. First order deformations of $Y$ are defined similarly and are classified by the finite-dimensional vector space $\mathbb{T}^1_Y=\Ext^1(\Omega^1_X; \scrO_X)$. There is a natural map $\mathbb{T}^1_Y\to \bigoplus_{x\in Y_{\text{\rm{sing}}}}T^1_{Y,x}$ 
which relates the global deformations to the local deformations of the singularities. Given a deformation $\mathcal{X} \to \Delta$ of $(X,x)$ over the unit disk $\Delta$, there is an associated \textsl{Kodaira-Spencer class} $\theta \in  T^1_{X,x}$. In general, the class $\theta$ does not tell us much about how the deformation $\mathcal{X}$ behaves with respect to the singularities. However, if $(X,x)$ is the germ of an isolated hypersurface singularity and $\theta \notin \mathfrak{m}_xT^1_{X,x}$, then $\mathcal{X}$ is smooth in a neighborhood of $x$ and hence the nearby fibers $X_t$ are smooth (Lemma~\ref{defsmooth}).  In this case, we call $\theta$ (or the corresponding deformation over the dual numbers) a \textsl{first order smoothing} of $X$. For a compact $Y$ with isolated hypersurface singularities, a class $\theta \in \mathbb{T}^1_Y$ is a \textsl{first order smoothing} of $Y$ if the image of $\theta$ in $T^1_{Y,x}$ is a first order smoothing for every $x\in Y_{\text{\rm{sing}}}$. 

There are then two main issues:
\begin{enumerate}
\item Analyze the image of $\mathbb{T}^1_Y$ in $\bigoplus_{x\in Y_{\text{\rm{sing}}}}T^1_{Y,x}$ to see if it contains first order smoothings. Of course, the ideal case would be if $\mathbb{T}^1_Y\to \bigoplus_{x\in Y_{\text{\rm{sing}}}}T^1_{Y,x}$ is surjective. However, this is almost never the case. For example, by \cite{F}, there are topological obstructions in the Calabi-Yau case in dimension $3$. We generalize this result to higher dimensions in Theorem~\ref{smoothCY}. In the Fano threefold case, $\mathbb{T}^1_Y\to \bigoplus_{x\in Y_{\text{\rm{sing}}}}T^1_{Y,x}$ is surjective if $Y$ has ordinary double points by \cite{F} but not in general, by examples of Namikawa \cite{NamikawaF}.  For singular Fano varieties $Y$ with  $\dim Y\geq 4$, $\mathbb{T}^1_Y\to \bigoplus_{x\in Y_{\text{\rm{sing}}}}T^1_{Y,x}$ can fail to be surjective even if $Y$ has only ordinary double points (Remark~\ref{Fanoremark}(iv)). There are also  obstructions to the  surjectivity of $\mathbb{T}^1_Y\to \bigoplus_{x\in Y_{\text{\rm{sing}}}}T^1_{Y,x}$ coming from the failure of local deformations of neighborhoods of the exceptional set in $\hY$ to come from global deformations of $\hY$, which we discuss in more detail below. (See also Remark~\ref{Fanoremark}(v) for some higher dimensional examples.)
\item If first order smoothings of $Y$ exist, determine if any of these arise from global smoothings of $Y$. Here, the ideal situation is for all deformations of $Y$ to be unobstructed, and we will only be able to obtain results in this case.
\end{enumerate}

We return to (1) above. In the local setting, the finite dimensional vector space $H^1(\hX; T_{\hX}) = H^0(X; R^1\pi_*T_{\hX})$ roughly measures the first order deformations of the noncompact complex manifold $\hX$, and there is a homomorphism $H^1(\hX; T_{\hX}) \to T^1_{X,x}$ whose image can be viewed as the tangent space to the set of deformations admitting a simultaneous resolution in a certain sense. There is a similar interpretation for $H^1(\hY; T_{\hY})$ and its image in $\mathbb{T}^1_Y$. However, the image of $H^1(\hX; T_{\hX})$ in $ T^1_{X,x}$ depends on the choice of a resolution. In general, one must work with the image of $H^1(\hX; \Omega^{n-1}_{\hX})$, which does not depend on the particular resolution (see Theorem~\ref{0.8} below). If $\pi\colon \hX \to X$ is a \textsl{crepant resolution}, which means that $K_{\hX} \cong \scrO_{\hX}$, or equivalently that $K_{\hX} \cong \pi^*\omega_X$, then the images of $H^1(\hX; T_{\hY})$ and $H^1(\hX; \Omega^{n-1}_{\hX})$ are the same, and thus have a geometric meaning. A similar statement holds if, instead of a good resolution, we use a small resolution $\pi'\colon X'\to X$, where $X'$ is smooth but the dimension of the exceptional set is less than $n-1$. For hypersurface (or more generally local complete intersection) singularities, this last case can only happen if $\dim X=3$ and the exceptional set is a curve. For a partial discussion of the meaning of the image of $H^1(\hX; T_{\hX})$, primarily in  the threefold case, see \cite{FL22b}.

In the global case, where $Y$ is a generalized Calabi-Yau or Fano variety and $\pi\colon \hY \to Y$ is a crepant or small resolution, the corresponding image of $H^1(\hY; T_{\hY})$ in $\mathbb{T}^1_Y$ is notoriously hard to control. For example, in the case of a small resolution of the singularity defined locally by $z_1^2 + z_2^2 + z_3^2 + z_4^{2n} =0$, the exceptional set is a $C\cong \Pee^1$ with normal bundle $\scrO_{\Pee^1}\oplus \scrO_{\Pee^1}(-2)$, and one version of the Clemens conjecture in this case, where $Y$ is a generalized Calabi-Yau threefold, asks whether there is a deformation of a small resolution $Y'$ where $C$ splits up into a union of smooth curves all of which have normal bundle $\scrO_{\Pee^1}(-1)\oplus \scrO_{\Pee^1}(-1)$. Very little is known about this conjecture in general.   In the case of a non-small and non-crepant  resolution, the geometric meaning of the image of $H^1(\hY; \Omega^{n-1}_{\hY})$ in $\mathbb{T}^1_Y$ seems even more mysterious. Instead, our idea is to work modulo this image, both locally and globally.   

To make this idea more precise, consider again the local setting. By a result of Schlessinger, if $U =X -\{x\}$ is a punctured neighborhood of the singularity, then $ T^1_{X,x} = H^1(U; T_U)$. We can also identify $U$ with $\hX -E$ and, by the assumption of canonical singularities,  $T_{\hX}$ with $\Omega^{n-1}_{\hX}(-D)$ where $D$ is an effective divisor supported on $E$. Then there are several natural sheaves on $\hX$ whose restrictions to $U$ are $T_U$. In particular, we can look at the inclusions
$$\Omega^{n-1}_{\hX}(\log E)(-E) \subseteq \Omega^{n-1}_{\hX} \subseteq \Omega^{n-1}_{\hX}(\log E).$$
 By standard results (cf.\ Remark~\ref{biratinvremark}),  $H^q(\hX; \Omega^p_{\hX}(\log E)(-E))$ and $H^q(\hX; \Omega^p_{\hX}(\log E))$ are birational invariants of $X$, i.e.\ do not depend on the particular choice of a resolution.
 
 \medskip

With this said, building on \cite{NS}, we show the following:

\begin{theorem}[Theorem~\ref{maintheorem}(iii)]\label{0.8} Let $X$ be an isolated rational complete intersection singularity. Then the  maps $H^1(\hX; \Omega^{n-1}_{\hX}(\log E)(-E)) \to H^1(U; T_U)$ and $H^1(\hX; \Omega^{n-1}_{\hX}(\log E)) \to H^1(U; T_U)$ are injective. The map $H^1(\hX; \Omega^{n-1}_{\hX})  \to H^1(U; T_U)$ is not in general injective, but its image is the same as the image of $H^1(\hX; \Omega^{n-1}_{\hX}(\log E)(-E))$.
\end{theorem} 

Working modulo the image of $H^1(\hX; \Omega^{n-1}_{\hX})$ is therefore the same as working modulo the image of $H^1(\hX; \Omega^{n-1}_{\hX}(\log E)(-E))$. To describe this quotient in more detail, we have the following:

\begin{theorem}[Theorem~\ref{maintheorem}(v)]\label{0.9} With hypotheses as in Theorem~\ref{0.8}, let $K$ be the quotient of  $H^1(U; T_U)$ by the image of $H^1(\hX; \Omega^{n-1}_{\hX}(\log E)(-E))$. Then $K\cong H^2_E(\hX; \Omega^{n-1}_{\hX}(\log E)(-E)) \cong \Ker \{H^2_E(\hX; \Omega^{n-1}_{\hX}) \to H^2(\hX; \Omega^{n-1}_{\hX}) \}$, and there is an exact sequence
$$0 \to \Gr_F^{n-1}H^n(L) \to K \to H^2_E(\hX; \Omega^{n-1}_{\hX}(\log E)) \to 0,$$
where $L$ is the link of the singularity.
\end{theorem}

Thus, our strategy only has a chance of working in case $K\neq 0$. We make the following definition:

 \begin{definition}\label{def0} With $K$ the vector  space defined above, the rational singularity $(X,x)$ is \textsl{$1$-rational} if $K=0$, and is   \textsl{$1$-irrational} if $K\neq 0$.   The rational singularity $(X,x)$ is \textsl{$1$-Du Bois} if $H^2_E(\hX; \Omega^{n-1}_{\hX}(\log E)) = 0$, and thus is \textsl{not $1$-Du Bois} if $H^2_E(\hX; \Omega^{n-1}_{\hX}(\log E)) \neq 0$.   By Theorem~\ref{0.9}, if $(X,x)$ is not $1$-Du Bois, then $(X,x)$ is  $1$-irrational. Finally, $(X,x)$ is  \textsl{$1$-liminal} if it is $1$-Du Bois but   $1$-irrational, or equivalently if $\Gr_F^{n-1}H^n(L) \neq 0$ and $H^2_E(\hX; \Omega^{n-1}_{\hX}(\log E)) =0$. (N.B. This terminology differs from an earlier use of the term $1$-rational to mean $R^1\pi_*\scrO_{\hX} =0$ for some or equivalently all good resolutions $\pi\colon \hX \to X$.) 
\end{definition} 

The meaning of these conditions in dimension $3$ is given by the following: 

\begin{theorem}[{\cite[Theorem 2.2]{NS}}] If $\dim X = 3$, an  isolated rational hypersurface singularity $(X,x)$ is   $1$-irrational. Furthermore, $(X,x)$ is $1$-liminal $\iff$ $(X,x)$ is an ordinary double point.
\end{theorem} 

However, in higher dimensions,  singularities which are rational but   $1$-irrational form a much more restrictive class of singularities than rational singularities. For example, an ordinary double point of dimension $n\geq 3$ is  $1$-irrational $\iff$ $n=3$. More generally, the cone over a smooth hypersurface of degree $d$ in $\Pee^n$ is rational $\iff$ $d\leq n$,   $1$-irrational $\iff$ $d\geq  \frac12(n+1)$, and $1$-liminal $\iff$ $d=\frac12(n+1)$. 

To give some context to Definition~\ref{def0},  there is a generalization of  of Definition~\ref{def0} defining   $m$-Du Bois,  $m$-rational  and $m$-liminal singularities for all $m\ge 0$, due to \cite{MOPW}, \cite[(1)]{JKSY-duBois}, \cite{MP-Mem} in the $m$-Du Bois case and \cite[\S4]{KL2}, \cite{FL22c} in the $m$-rational case (see also \cite{FL22d} for the $m$-liminal case and further discussion in the isolated complete intersection case). In particular, for $m=0$, an $0$-rational singularity is just a  rational singularity and  a $0$-Du Bois singularity is a Du Bois singularity as defined by Steenbrink \cite[\S3]{Steenbrink}.  In Section~\ref{Section2}, we  show that, in case $m=1$, these definitions agree  with  Definition~\ref{def0}.

\medskip

Given the above definitions, we can state our main results. We begin with the (easier) unobstructedness results:

 \begin{theorem}[Theorem~\ref{Fano1DB}]  Suppose that $Y$ is a  projective variety  with   $1$-Du Bois local complete intersection singularities, not necessarily isolated, such that  $\omega^{-1}_Y$ is ample. Then   all deformations of $Y$ are unobstructed. 
 \end{theorem}

In the Calabi-Yau case,  in \cite[Corollary 1.5]{FL22c}, we   obtained the following generalization of Theorem \ref{thmK}:

\begin{theorem}[Theorem~\ref{unob1}]\label{unob11} Let $Y$ be a canonical Calabi-Yau $n$-fold (Definition~\ref{defnCY})  such that all singularities of $Y$ are $1$-Du Bois lci singularities, not necessarily isolated.   Then the deformations of $Y$ are unobstructed.
\end{theorem}

We note that Gross \cite{Grossobstr} has given examples of canonical Calabi-Yau $n$-folds with non-isolated singularities with obstructed deformations. 

Without more information on the existence of first order smoothings, the above results are not especially useful. In the Fano case, we show the following: 

\begin{theorem}[Corollary~\ref{Fanocorollary}] Let $Y$ be a compact analytic variety of dimension $n\geq 3$ with isolated rational hypersurface singularities, such that $\omega_Y^{-1}$ is ample. Suppose that  the singularities of $Y$ are  $1$-irrational and that there exists a  Cartier divisor $H$ not passing through the singularities of $Y$ such that $\omega_Y =\scrO_Y(-H)$. Under certain cohomological conditions, $Y$ is smoothable. In particular, if $\dim Y =3$,  or if all singularities of $Y$ are $1$-liminal and $H^{n-3}(H; \Omega^1_H) =0$, then $Y$ is smoothable. 
\end{theorem} 

A more precise statement is given in Section~\ref{Section4}. It is likely that both the assumption that $\omega_Y =\scrO_Y(-H)$ for some $H$ not passing through the singularities of $Y$ and the cohomological hypotheses are not  optimal. (Compare Remark~\ref{Fanoremark}(iii).) 

 In the Calabi-Yau case, first order smoothings are constructed via the following theorem:

\begin{theorem}[Theorem~\ref{smoothCY}]\label{smoothCY1} Let $Y$ be a canonical Calabi-Yau $n$-fold with $H^1(Y;\scrO_Y) =0$ such that all singularities of $Y$ are $1$-irrational hypersurface singularities, and are either $1$-Du Bois, hence $1$-liminal, or are not $1$-Du Bois and satisfy a slightly stronger technical condition that we call \textsl{strongly $1$-irrational} (Definition~\ref{defsub}). Write $Z = Y_{\text{\rm{sing}}} = Z'\cup Z''$, where $Z'$ is the set of strongly $1$-irrational points and $Z''$ is the set of $1$-liminal points, and let $Y'$ be a partial resolution of $Y$ at the points of $Z''$. Then there exists a first order smoothing of $Y$ $\iff$ there is a linear relation in $H_{n-1}(Y')$,  analogous to the condition that, for a canonical Calabi-Yau threefold, for every $x\in Z''$,  there exist nonzero $a_x\in \Cee$ such that $\sum_{x\in Z''}a_x[C_x] = 0$ in $H_2(Y')$. 
\end{theorem} 

Putting Theorems~\ref{smoothCY1} and \ref{unob11} together, in case $\dim Y > 3$, we are able to obtain smoothing results in the Calabi-Yau case if the singularities are both $1$-Du Bois and $1$-irrational, or equivalently $1$-liminal.

\subsection*{Summary and Outlook}
We can roughly summarize the above discussion as follows:  In both the Fano and the Calabi-Yau case, the deformations are unobstructed provided that all singularities are local complete intersection $1$-Du Bois singularities (in dimension $3$, for isolated hypersurface singularities, this is just the assumption of ordinary double points, recovering Theorem \ref{thmK}).   Specializing further to isolated hypersurface singularities, in the Fano case, under some mild assumptions, first order smoothings exist as soon as the singularities are $1$-irrational. In dimension $3$, this is equivalent to the assumption that the singularities are  canonical, or equivalently rational. In the Calabi-Yau case, in order for first order smoothings to exist, we need to make an essentially topological assumption about the $1$-liminal singularities (analogous to the condition of \cite{F} for smoothing nodal threefolds). Thus, from the point of view of first order smoothings, $1$-liminal singularities are a very natural generalization of ordinary double points of dimension $3$. We emphasize that  our  main technical results, Theorem \ref{maintheorem} and Theorems~\ref{Fano1DB} and ~\ref{unob1}, require almost completely disjoint conditions on the singularities: $1$-irrational singularities for obtaining first order smoothings, $1$-Du Bois for obtaining lifts to higher order deformations. Only the $1$-liminal case satisfies both conditions.  

We list  some natural open problems:

\begin{enumerate}
\item Can one weaken the assumption of  isolated rational hypersurface singularities to isolated local complete intersection singularities? Compare \cite[\S3]{gross_defcy} in the case of dimension $3$. 
\item Are there  more general unobstructedness results in the Calabi-Yau and Fano cases? For example, is there  a higher dimensional analogue of Theorem \ref{thmN} covering rational  but not necessarily $1$-Du Bois  isolated  local complete intersection singularities, and possibly  even an extension to the log canonical case? 
\item It is surprising, and somewhat disappointing,  that our techniques for finding a first order smoothing do not apply to the simplest singularities, ordinary double points, in case $n= \dim Y > 3$. For $n$ odd and greater than $3$, if $Y$ is a canonical Calabi-Yau $n$-fold  with only ordinary double points, Rollenske-Thomas  \cite{RollenskeThomas} have found necessary conditions on the singular points for there to to exist a first order smoothing. Note  that the case of  hypersurfaces with many ordinary double points of degree  $n+1$ or $n+2$ in $\Pee^{n+1}$, $n\ge 4$, shows that the map $\mathbb{T}^1_Y \to \bigoplus_{x\in Y_{\text{\rm{sing}}}}T^1_{Y,x}$ cannot be surjective in general, so some obstructions must exist. 
If $n$ is odd, say $n = 2k+1$,  ordinary double points are $k$-liminal, i.e.\ both $k$-irrational and $k$-Du Bois. The authors have generalized the methods of \cite{RollenskeThomas} to the case of Calabi-Yau $n$-folds with isolated weighted homogeneous $k$-liminal hypersurface singularities \cite{FL23}. How far are these conditions from being sufficient? 
\item  The above gives some weak necessary conditions in the $k$-liminal case. What happens for $1$-Du Bois singularities in the non-$k$-liminal case, for example the case of ordinary double points when $n$ is even? This case is completely mysterious. 
\end{enumerate}

\subsection*{Structure of the paper} The outline of this paper is as follows: In Section~\ref{Section1}, we review standard facts about isolated singularities and collect some of the main technical preliminaries used throughout the paper. Section~\ref{Section2} is devoted to one of the main technical results, Theorem~\ref{maintheorem}. Some of the statements are contained in work of Steenbrink \cite{SteenbrinkDB}. The key new ingredients are stated in Parts (iii) and (vi) of Theorem~\ref{maintheorem}. We then define   $1$-rational and $1$-Du Bois singularities,  and relate these definitions to the more general definitions of $1$-rational and $1$-Du Bois singularities (Proposition \ref{relatedefs}) given by \cite{MOPW}, \cite{JKSY-duBois}, \cite{FL22c}. Section~\ref{Section3} looks at the case of weighted homogeneous singularities, where we have stronger results (e.g. Theorem \ref{wtdhmg}). In particular, inspired by a result of Wahl \cite{Wahl}, we go beyond a strictly numerical characterization of $1$-rational or $1$-Du Bois singularities to identify the relevant groups inside of $T^1_{X,x}$ in terms of the $\Cee^*$-weights.  Section~\ref{Section4} gives smoothing results for generalized Fano varieties (Theorem~\ref{Fanofirst}  and Corollary~\ref{Fanocorollary}). In Section~\ref{Section5},   we discuss first order smoothings of canonical Calabi-Yau varieties with $1$-irrational singularities  and give criteria for the existence of smoothings whenever there is a corresponding unobstructedness result.

\subsection*{Notations and Conventions} To fix our notation, we emphasize that $(X,x)$ always denotes the germ of an isolated singularity, with $X$ a good representative, i.e.\  a contractible Stein representative. In the local setting, $\pi: \hX\to X$ denotes a good resolution with exceptional divisor $E=\pi^{-1}(x)$.  In particular, note that if $\mathcal F$ is a coherent sheaf on $\hX$, $H^i(\hX;\mathcal F)\cong H^0(X;R^i\pi_*\mathcal F)$, and further we can identify the sheaf $R^i\pi_*\mathcal F$ with its global sections. Thus, we will use $H^i(\hX,\mathcal F)$ and $R^i\pi_*\mathcal F$ interchangeably; of course, in the case of a non-isolated singularity, we would need to consistently use $R^i\pi_*\mathcal F$. We set $U = X-\{x\} = \hX - E$. In the global setting, $Y$ will always denote a compact complex analytic variety with isolated singularities and $\pi: \hY\to Y$  a good resolution with exceptional divisor, which we also denote by $E$. Finally, all singular homology and cohomology groups have  coefficients in $\Cee$.
  
\subsection*{Acknowledgements} It is a pleasure to thank Johan de Jong, J\'anos Koll\'ar, Mircea Musta\c{t}\u{a}, and Mihnea Popa for insightful conversations and correspondence. We would also like to thank the referees for a very careful reading of this paper and for numerous comments which helped to improve it.

\section{Preliminaries}\label{Section1}
\subsection{Hodge theory of singularities}
We begin by collecting  some basic Hodge theory results we shall need. We use  the notational conventions for $X$, $\hX$, $E$ described above. We denote by $\Omega^\bullet_E$ the complex of K\"ahler differentials on $E$ and by $\tau^\bullet_E$ the subcomplex of torsion differentials, i.e.\ those supported on $E_{\text{\rm{sing}}}$. Then $(\Omega^\bullet_E/\tau^\bullet_E, d)$ is a resolution of the constant sheaf $\Cee$ on $E$, and there is a short exact sequence

\begin{equation}\label{1.00}
0 \to \Omega^\bullet_{\hX}(\log E)(-E) \to \Omega^\bullet_{\hX} \to \Omega^\bullet_E/\tau^\bullet_E \to 0.
\end{equation}

 The restriction exact sequence 
\begin{equation}\label{1.01} 0 \to \Omega^\bullet_{\hX}(\log E)(-E) \to \Omega^\bullet_{\hX}(\log E)\to \Omega^\bullet_{\hX}(\log E)|E \to 0,
\end{equation}
gives rise to a long exact sequence
 
\begin{equation}\label{1.1}\dots \to H^q(\hX; \Omega^p_{\hX}(\log E)(-E))  \to  H^q(\hX;\Omega^p_{\hX}(\log E))\to  H^q(E;\Omega^p_{\hX}(\log E)|E)  \to \cdots 
\end{equation}
 
  Let $L$ be the link of the pair $(\hX, E)$ \cite[\S6.2]{PS}, so that $L $ has the homotopy type of  $\hX -E$ in case $E$ is a deformation retract of $\hX$. In this case there is the relative cohomology  exact sequence
\begin{equation}\label{1.11} \cdots \to H^{k-1}(L) \to H^k_E(\hX) \to H^k(E) \to H^k(L) \to \cdots,
\end{equation}
which is an exact sequence  of mixed Hodge structures if all components of $E$ are compact K\"ahler.
Similarly, there are exact sequences
\begin{multline}\label{1.2}
\dots \to H^q(E; \Omega^p_{\hX}(\log E)|E) \to H^q(E; \Omega^p_{\hX}(\log E)/\Omega^p_{\hX}) \to H^{q+1}(E;\Omega^p_E/\tau_E^p) \to \\
\to H^{q+1}(E; \Omega^p_{\hX}(\log E)|E) \to H^{q+1}(E; \Omega^p_{\hX}(\log E)/\Omega^p_{\hX}) \to \dots 
\end{multline}
Using (\ref{1.00}) and (\ref{1.01}), the exact sequence  of the link comes from the exact sequence 
$$\begin{CD}
0@>>>  \Omega^\bullet_{\hX} /\Omega^\bullet_{\hX}(\log E)(-E) @>>> \Omega^\bullet_{\hX}(\log E)/\Omega^\bullet_{\hX}(\log E)(-E) @>>> \Omega^\bullet_{\hX}(\log E)/\Omega^\bullet_{\hX} @>>> 0\\
@. @V{\cong}VV @| @| @.\\
0 @>>> \Omega^\bullet_E/\tau_E^\bullet @>>> \Omega^\bullet_{\hX}(\log E)|E @>>> \Omega^\bullet_{\hX}(\log E)/\Omega^\bullet_{\hX} @>>> 0,
\end{CD}$$
Here $\mathbb{H}^k(E;\Omega^\bullet_{\hX}(\log E)|E) = H^k(L)$ and $\mathbb{H}^k(E;\Omega^\bullet_{\hX}(\log E)/\Omega^\bullet_{\hX}) = H^{k+1}_E(\hX)$ (\cite[Theorem 1.9]{Steenbrink},  \cite[\S5.5, \S6.2]{PS}). The long exact sequences on hypercohomology or cohomology then give the exact sequences (\ref{1.11}) and (\ref{1.2}) above.

 For an isolated singularity $X$, the condition that $X$ is Du Bois (= $0$-Du Bois) is equivalent to the statement that  $R^i\pi_*\scrO_{\hX}(-E) = 0$ for $i>0$. The following is then due to Steenbrink \cite{Steenbrink}:

\begin{lemma}\label{ratlplus} If $(X,x)$ is a rational singularity, then,  for $i> 0$, $H^i(E; \scrO_E) = 0$ and $R^i\pi_*\scrO_{\hX}(-E) = 0$ (i.e.\ $(X,x)$ is a Du Bois singularity). Finally, $H^i_E(\hX; \Omega^n_{\hX}) = 0$ for $i< n$.  
\end{lemma} 
\begin{proof} By \cite[Lemma 2.14]{Steenbrink}, with no assumption on $(X,x)$, the map  $R^i\pi_*\scrO_{\hX} \to H^i(E; \scrO_E)$ is surjective, and hence $H^i(E; \scrO_E) = 0$ for $i> 0$ in case $(X,x)$ is rational. The second statement then follows by applying $R^i\pi_*$ to the exact sequence
$$0 \to \scrO_{\hX}(-E) \to \scrO_{\hX} \to \scrO_E \to 0,$$
and using the fact that $R^0\pi_*\scrO_{\hX} \to  H^0(E; \scrO_E)$ is surjective. The last statement is a consequence of duality.
\end{proof}

Next, we have the vanishing theorem of Guill\'en, Navarro Aznar, Pascual-Gainza, Puerta and Steenbrink:
\begin{theorem}\label{GNAPS}  For $p+q> n$, 
$H^q(\hX; \Omega^p_{\hX}(\log E)(-E))  =0$. Hence, by duality, $H^q_E(\hX; \Omega^p_{\hX}(\log E)) = 0$ for $p+q< n$. \qed
\end{theorem} 

\begin{corollary}\label{Hn-1E} \begin{enumerate} \item[\rm(i)] For all $q\geq 2$,  
$$H^q(\hX; \Omega^{n-1}_{\hX}) \cong H^q(E; \Omega_E^{n-1}/\tau_E^{n-1}) \cong \bigoplus_iH^q(E_i; \Omega_{E_i}^{n-1}).$$
\item[\rm(ii)] If $(X,x)$ is a rational singularity, $H^0(E;\Omega^{n-1}_E/\tau_E^{n-1}) = 0$.
\end{enumerate}
\end{corollary}
\begin{proof}  To see (i), by Theorem~\ref{GNAPS} and the exact sequence (\ref{1.00}), $H^q(\hX; \Omega^{n-1}_{\hX}) \cong H^q(E; \Omega_E^{n-1}/\tau_E^{n-1})$. The Mayer-Vietoris spectral sequence for $\Omega_E^{n-1}/\tau_E^{n-1}$ then implies that $\Omega_E^{n-1}/\tau_E^{n-1} \cong \bigoplus_i\Omega_{E_i}^{n-1}$.  To see (ii), note that  $H^0(E;\Omega^{n-1}_E/\tau_E^{n-1}) \cong  \bigoplus_iH^0(E_i; \Omega^{n-1}_{E_i})$. Also, by the Hodge symmetries, $\dim H^0(E_i; \Omega^{n-1}_{E_i}) = \dim H^{n-1}(E_i; \scrO_{E_i})$. Finally, an examination of the Mayer-Vietoris spectral sequence for $\scrO_E$ which degenerates at $E_2$, and the fact that $H^{n-1}(E;\scrO_E) = 0$, imply that $H^{n-1}(E_i; \scrO_{E_i})=0$ for all $i$.
\end{proof}

Next, we recall some of the basic numerical invariants of an isolated singularity \cite{SteenbrinkDB}:

\begin{definition}\label{defDBlink}  For $q>0$, define the \textsl{Du Bois invariants} 
$$b^{p,q} = \dim H^q(\hX; \Omega^p_{\hX}(\log E)(-E))  = \dim H^{n-q}_E(\hX; \Omega^{n-p}_{\hX}(\log E)).$$
 There are also the  \textsl{link invariants}
$$\ell^{p,q} = \dim  H^q(E;\Omega^p_{\hX}(\log E)|E) = \dim \Gr_F^pH^{p+q}(L).$$
By Serre duality, $\ell^{p,q} =  \ell^{n-p,n-q-1}$. 
\end{definition}

\begin{theorem}[\cite{SteenbrinkDB}]\label{numervan}  For $p+q> n$, $b^{p,q} =0$. If $X$ is an lci singularity, then $b^{p,q} =0$ for all $q> 0$ unless $p+q = n-1$ or $n$. Finally, $\ell^{p,q} = 0$ except in the following cases: $(p,q) = (0,0)$ or $(n, n-1)$, or $p+q = n-1$ or $n$. \qed
\end{theorem}

Another essential tool is the semipurity  theorem of Goresky-MacPherson \cite[(1.12)]{Steenbrink}:

\begin{theorem}\label{semipure} The morphism of mixed Hodge structures  $H^k_E(\hX) \to H^k(E)$ is injective for $k\leq  n$, surjective for $k\geq n$, and hence an isomorphism for $k=n$. \qed
\end{theorem}
 
 The first statement in the following lemma is due  to Namikawa-Steenbrink  \cite[p.\ 407]{NS}, \cite[p.\ 1369]{SteenbrinkDB}:

\begin{lemma}\label{convzero}  The hypercohomology groups $\mathbb{H}^k(\hX; \Omega^\bullet_{\hX}(\log E)(-E))=0$ for all $k$, and hence the spectral sequence with $E_1^{p,q}= H^q(\hX; \Omega^p_{\hX}(\log E)(-E))$ converges to zero.  Dually, the hypercohomology groups $\mathbb{H}^k_E (\hX; \Omega^\bullet_{\hX}(\log E))=0$ for all $k$, and hence the spectral sequence with $E_1^{p,q}= H^q_E(\hX; \Omega^p_{\hX}(\log E))$ converges to zero. 
\end{lemma} 
\begin{proof} We shall just prove the second statement. With $U =\hX -E$, we have the long exact hypercohomology  sequence
 
$$\cdots \to \mathbb{H}^k (\hX;\Omega^\bullet_{\hX}(\log E)) \to \mathbb{H}^k(U;\Omega^\bullet_{\hX}(\log E)|U)\to \mathbb{H}^{k+1}_E (\hX;\Omega^\bullet_{\hX}(\log E)) \to \cdots$$
As $\mathbb{H}^k (\hX;\Omega^\bullet_{\hX}(\log E))\cong H^k(U)$ and 
$\mathbb{H}^k(U;\Omega^\bullet_{\hX}(\log E)|U) = \mathbb{H}^k(U;\Omega^\bullet_{U}) \cong H^k(U)$ 
and the natural map is the identity under these identifications, we see that $\mathbb{H}^{k+1}_E (\hX;\Omega^\bullet_{\hX}(\log E))=0$ for all $k$. As $E_1^{p,q} \implies  \mathbb{H}^{p+q}_E (\hX;\Omega^\bullet_{\hX}(\log E))$, we are done.
\end{proof}

\begin{corollary}\label{1.5} Suppose that $(X,x)$ is a rational or more generally a  Du Bois singularity. Then  $H^1_E(\hX; \Omega_{\hX}^{n-1}(\log E)) = 0$, and  there is an exact sequence
$$0 \to H^2_E(\hX;\Omega^{n-2}_{\hX}(\log E)) \xrightarrow{d} H^2_E(\hX;\Omega^{n-1}_{\hX}(\log E))\to A\to 0,$$
where  $A = \Ker\{d\colon H^3_E(\hX;\Omega^{n-3}_{\hX}(\log E)) \to H^3_E(\hX;\Omega^{n-2}_{\hX}(\log E))\}$.
\end{corollary}
\begin{proof} The Du Bois condition, that $R^i\pi_*\scrO_{\hX}(-E) = 0$ for $i >0$, is equivalent by duality to the statement that $H^i_E(\hX; \Omega^n_{\hX}(\log E)) = 0$ for $i< n$ (cf.\ Lemma~\ref{ratlplus}). The corollary then  follows easily by examining the $E_1$ and $E_2$ pages of the above hypercohomology spectral sequence $E_1^{p,q}= H^q_E(\hX; \Omega^p_{\hX}(\log E)) \implies 0$ and using Theorem~\ref{GNAPS}. 
\end{proof}  

(Compare \cite{SteenbrinkDB} for a related  dual statement.)

\subsection{Deformation theory preliminaries}
We begin with the case where  $(X,x)$ is the germ of an isolated  hypersurface singularity. In this case, $H^0(X;T^1_X)$ is a cyclic $\scrO_{X,x}$-module. In particular,  $u\in H^0(X;T^1_X)$ is a generator for $H^0(X;T^1_X)$ as an $\scrO_{X,x}$-module $\iff$ $u\notin \mathfrak{m}_x\cdot H^0(X;T^1_X)$. The relevance for smoothings of $X$ is the following well-known lemma:

\begin{lemma}\label{defsmooth} Let $\mathcal{X} \to \Delta$ be a one-parameter deformation of $(X,x)$ with Kodaira-Spencer class $u\in H^0(X;T^1_X)$. Then $\mathcal{X}$ is smooth in a neighborhood of $x$, and in particular gives a smoothing of $(X,x)$ $\iff$ $u\notin \mathfrak{m}_x\cdot H^0(X;T^1_X)$.
\end{lemma} 
\begin{proof} Let $R = \scrO_{X,x}$.  By a standard argument, $\mathcal{X}$ is smooth in a neighborhood of $x$ $\iff$ $\Ext^1_R(\Omega^1_{\mathcal{X},x}|X, R) = 0$.  The Kodaira-Spencer class $u$ of a one-parameter deformation $\mathcal{X}$ is by definition  $\partial(1)$ via the exact sequence  
$$\Hom_R(R, R)\cong R \xrightarrow{\partial} \Ext^1_R(\Omega^1_{X,x},R) \to \Ext^1_R(\Omega^1_{\mathcal{X},x}|X,R)\to 0.$$
Thus $\Omega^1_{\mathcal{X},x}$ is locally free in a neighborhood  of $x$ $\iff$ $\Ext^1_R(\Omega^1_{\mathcal{X},x}|X,R) =0$  $\iff$ $\partial(1)$ generates $\Ext^1_R(\Omega^1_{X,x},R) = H^0(X;T^1_X)$.
\end{proof} 

This motivates the following definition: 

\begin{definition}\label{deffirstordersm}  If $(X,x)$ is the germ of an isolated  hypersurface singularity, then an element $u\in H^0(X;T^1_X)$ is a \textsl{first order smoothing} of $X$ if $u\notin \mathfrak{m}_x\cdot H^0(X;T^1_X)$. For a compact $Y$ with isolated hypersurface singularities, a class $u \in \mathbb{T}^1_Y$ is a \textsl{first order smoothing} of $Y$ if the image of $u$ in $T^1_{Y,x}$ is a first order smoothing for every $x\in Y_{\text{\rm{sing}}}$. 
\end{definition}

Turning now to arbitrary isolated singularities, we make the following definitions:

\begin{definition} Let $\pi \colon \hX \to X$ be a resolution of  $X$, where $X$ has an isolated singularity at $x$. Then $\pi$ is \textsl{equivariant} if $R^0\pi_*T_{\hX}\cong T^0_X$. By e.g.\ \cite[Lemma 3.1]{F}, a small resolution is equivariant. Note that, for $q> 0$, $R^q\pi_*T_{\hX}$ is a torsion sheaf supported on $x$.

The resolution $\pi \colon \hX \to X$ is:
\begin{enumerate}
\item a \textsl{good resolution} or a  \textsl{log resolution} if $\pi^{-1}(x) = E$ is a divisor with simple normal crossings.
\item  a \textsl{good equivariant resolution} if $\pi$ is good and equivariant. 
\item  \textsl{crepant} if $K_{\hX} = \pi^*\omega_X$ and hence $K_{\hX}\cong \scrO_{\hX}$, and is a \textsl{good crepant resolution} if $\pi$ is also a good resolution. (Thus, for our purposes, a small resolution is crepant but not a good crepant resolution.) If there exists a crepant resolution, $(X,x)$ is a rational Gorenstein singularity.
\end{enumerate}
If $Y$ is a compact analytic space with isolated singularities, a resolution $\pi \colon \hY \to Y$ is always assumed to be an isomorphism away from the finite set $Y_{\text{sing}}$. Then $\pi \colon \hY \to Y$ is a good resolution if it is so at every point of $Y_{\text{sing}}$, and similarly for good equivariant resolutions, crepant resolutions, and good crepant resolutions. 
\end{definition}

It is well-known  that  good equivariant resolutions always exist:

\begin{lemma}\label{existgoodeq} Let $(X,x)$ be the germ of an isolated singularity. Then there exists a good equivariant resolution of $(X,x)$.
\end{lemma}
\begin{proof} (Sketch.) Choose a resolution algorithm which is functorial with respect to smooth morphisms as in e.g.\ \cite[\S3.4]{Kollarres}, yielding a resolution $\pi\colon\hX \to X$. If $\theta \in H^0(X;T^0_X)$, then $\theta$ has to vanish at $x$ since $x$ is an isolated singular point. Then the vector field $\theta$ can be integrated to a local one parameter group action as in \cite[3.9.1]{Kollarres}, i.e.\  a smooth morphism of germs $(\Delta, 0)\times (X,x) \to (X,x)$ such that the image of $\partial/\partial t$ is $\theta$. By the argument of \cite[3.9.1]{Kollarres}, the local one parameter group action lifts to an action on $\hX$. Differentiating this action gives a lift of $\theta$ to a neighborhood of the exceptional divisor in $\hX$. Then $R^0\pi_*T_{\hX}\to  T^0_X$ is surjective, and it is injective because $R^0\pi_*T_{\hX}$ is torsion free. Hence $R^0\pi_*T_{\hX} \cong T^0_X$. 
\end{proof}

We have the following basic result \cite{KollarMori}, \cite[(11.1)]{Kollar}:

\begin{theorem} A normal  Gorenstein   singularity is canonical if and only if it is rational. \qed 
\end{theorem}

\begin{remark}\label{inclusion}   If $X$ has  canonical singularities and $\pi$ is a good, not necessarily crepant resolution, there exists an effective divisor $D$ whose support is contained in $E$ such that $K_{\hX}\cong \scrO_{\hX}(D)$. It follows that, in general, there is a perfect pairing
$$\Omega^1_{\hX}\otimes \Omega^{n-1}_{\hX} \to \Omega^n_{\hX} \cong \scrO_{\hX}(D),$$
and thus $T_{\hX} \cong \Omega^{n-1}_{\hX}(-D)$. In particular, there is an inclusion $T_{\hX} \to \Omega^{n-1}_{\hX}$.
\end{remark}

We now assume that $(X,x)$ is the germ of an isolated Gorenstein singularity of dimension $n\geq 3$, with $\pi\colon \hX \to X$ a resolution of singularities and $X$   a contractible Stein representative for the germ $(X,x)$. 

\begin{lemma}\label{1.0} The sheaf $T^0_X$ has depth at least $2$. Equivalently, if $U = X-x= \hX -E$ and $i\colon U \to X$ is the inclusion, then the natural map $T^0_X \to i_*i^*T^0_X$ is an isomorphism, $H^0(X;T^0_X) \cong H^0(U; T_U)$, and $H^1_x(X; T^0_X) = 0$.
\end{lemma}
\begin{proof} By assumption,  $\scrO_X$ has depth at least three.  Since $T^0_X$ is the dual of the sheaf $\Omega^1_X$,  dualizing a local presentation $\scrO_X^b \to \scrO_X^a \to \Omega^1_X\to 0$ gives   (locally)  an exact sequence $0\to T^0_X\to \scrO_X^a \to \scrO_X^b$. Hence $T^0_X$ has depth at least $2$.
\end{proof}

\begin{lemma}\label{imageH0} If $\hX$ is a good equivariant resolution of $X$ and $X$ has canonical singularities, $\pi_*T_{\hX} = \pi_*T_{\hX}(D)  = \pi_*\Omega^{n-1}_{\hX}$, and the natural map $H^0(\hX; \Omega^{n-1}_{\hX}) \to H^0(U; T_U)$ is an isomorphism.
\end{lemma}
\begin{proof} The inclusion   $T_{\hX} \to \Omega^{n-1}_{\hX}$ leads to a sequence of inclusions
$$R^0\pi_*T_{\hX} \to R^0\pi_*\Omega^{n-1}_{\hX} \to i_*i^*R^0\pi_*\Omega^{n-1}_{\hX} = i_*i^*T^0_X =T^0_X,$$
since $T^0_X$ has depth at least $2$. As $\pi$ is equivariant, $R^0\pi_*T_{\hX} \to T^0_X$ is an isomorphism, hence all of the above maps are equalities. Since $H^0(\hX; \Omega^{n-1}_{\hX}) = H^0(X; R^0\pi_*\Omega^{n-1}_{\hX}) = H^0(X;T^0_X)$ and $H^0(X;T^0_X) \cong H^0(U; T_U)$, the  map $H^0(\hX; \Omega^{n-1}_{\hX}) \to H^0(U; T_U)$ is an isomorphism.
\end{proof}

Since   $\operatorname{depth}_xX \geq 3$, there is the following result of Schlessinger  \cite[Theorem 2]{Schlessinger}:

\begin{lemma}\label{2.1} There are natural identifications  $H^2_x(X;T^0_X) \cong H^1(U, T_U) \cong H^0(X; T^1_X)$.  
\end{lemma}
\begin{proof} For simplicity, we assume that $X$ is a local complete intersection (the only case we shall need). Let $\mathcal{P}^\bullet = T_{\Cee^N}|X \to N$ be the normal complex, i.e.\ the dual of the  conormal complex $I/I^2 \to \Omega^1_{\Cee^N}|X$, where $I$ is the ideal sheaf of the local complete intersection $X$ in $\Cee^N$. Hence $\mathcal{P}^\bullet$ is a complex of locally free sheaves.  The local hypercohomology groups $\mathbb{H}_x^k(\mathcal{P}^\bullet) =0$ for $k\leq 2$, because the spectral sequence with $E_1^{p,q} = H^q_x(\mathcal{P}^p)$ is $0$ for   $q\leq 2$ and all $p$ since $\operatorname{depth}_xX \geq 3$. Since the cohomology sheaves $\mathcal{H}^k\mathcal{P}^\bullet$ of $\mathcal{P}^\bullet$ are $T^0_X$ in dimension $0$ and $T^1_X$ in dimension $1$, it follows that the only possible nonzero differential in the spectral sequence with $E_2^{p,q}$ page $H^p_x(\mathcal{H}^q\mathcal{P}^\bullet)$, namely $d_2\colon H^0_x(X;T^1_X) \to H^2_x(X;T^0_X)$, is an isomorphism. Thus $H^0(X; T^1_X) = H^0_x(X;T^1_X) \cong H^2_x(X;T^0_X)\cong H^1(U, T_U)$.
\end{proof}

We  have the Leray spectral sequence in local cohomology, with $E_2^{p,q}$ term 
$$ H^p_x(X; R^q\pi_*T_{\hX})  \implies H^{p+q}_E(\hX; T_{\hX}).  $$
For $q> 0$, 
$$H^p_x(X; R^q\pi_*T_{\hX}) = H^p(X; R^q\pi_*T_{\hX}),$$
and in particular both sides are $0$ if $p> 0$. 

From now on, we assume that $\pi \colon \hX \to X$ is an equivariant  resolution of  $X$  and that $(X,x)$ is a rational, hence canonical singularity.   Since $H^1_x(T^0_X)=0$, we have an exact sequence
$$0\to H^1_E(T_{\hX})\to H^0(R^1\pi_*T_{\hX}) \to H^2_x(T^0_X)  \to H^2_E(T_{\hX}) \to H^0(R^2\pi_*T_{\hX}).$$
Thus there is an exact sequence
$$0\to H^1_E(T_{\hX})\to H^0(R^1\pi_*T_{\hX}) \to H^0(X; T^1_X)  \to H^2_E(T_{\hX}) \to H^0(R^2\pi_*T_{\hX}).$$
Tracing through the various identifications, we have the following:

\begin{lemma}\label{localCD} For an equivariant  resolution $\pi \colon \hX \to X$  of  $X$, via the natural isomorphism $H^1(U; T_U)  \cong  H^2_x(X;T^0_X)$, the image of $H^1(\hX; T_{\hX})$ in $H^1(U; T_U)$ is identified with the image of $d_2\colon H^0_x(X; R^1\pi_*T_{\hX}) \to  H^2_x(X;T^0_X)$, where $d_2$ is the differential in the local cohomology Leray spectral sequence.  Likewise, the image of $H^1(\hX; \Omega^{n-1}_{\hX})$ in $H^1(U; \Omega^{n-1}_{\hX}|U)$ is identified with the image of $d_2\colon H^0_x(X; R^1\pi_*\Omega^{n-1}_{\hX}) \to  H^2_x(X;R^0\pi_*\Omega^{n-1}_{\hX})\cong   H^2_x(X;T^0_X)$.
\end{lemma}
\begin{proof} More generally, let $\mathcal{F}$ be any coherent sheaf on $\hX$. The local cohomology sequences for $\mathcal{F}$ on $\hX$ and for $\pi_*\mathcal{F}$ on $X$ are compatible, in the sense that there is a commutative diagram
$$\begin{CD}
H^1(U; \mathcal{F}|U) @>{\cong}>> H^2_x(X;R^0\pi_*\mathcal{F})\\
@| @VVV \\
H^1(U; \mathcal{F}|U) @>>>  H^2_E(\hX; \mathcal{F}).
\end{CD}$$ 
The isomorphism $H^1(U; \mathcal{F}|U) \to  H^2_x(X;R^0\pi_*\mathcal{F})$ then  identifies the kernel of $H^1(U; \mathcal{F}|U) \to  H^2_E(\hX; \mathcal{F})$, namely the image of $H^1(\hX; \mathcal{F})$, with the image of $d_2\colon H^0_x(X; R^1\pi_*\mathcal{F}) \to  H^2_x(X;R^0\pi_*\mathcal{F})$. Applying this remark to $\mathcal{F} = T_{\hX}$ or to $\mathcal{F} =  \Omega^{n-1}_{\hX}$ completes the proof.
\end{proof}
 
The geometric meaning of the map $H^0(R^1\pi_*T_{\hX}) \to H^2_x(T^0_X) \cong H^0(X; T^1_X)$ is as follows. The group  $H^0(X;R^1\pi_*T_{\hX})$ is the tangent space to $\mathbf{Def}_{\hX}$.  Since $X$ has only rational singularities,  there is a natural morphism   $\mathbf{Def}_{\hX} \to \mathbf{Def}_X$ by a theorem of Wahl \cite{Wahl}. By a result of Gross \cite{gross_defcy}, $\mathbf{Def}_{\hX}$ is unobstructed in case $\hX$ is a crepant resolution (in any dimension), and $\mathbf{Def}_X$ is always unobstructed    for a  local complete intersection $(X,x)$. On the level of tangent spaces, there are two ways to see the map  $\mathbf{Def}_{\hX} \to \mathbf{Def}_X$: first, one can check using the rationality assumption that there is an isomorphism $\Ext^1(\pi^*\Omega^1_X, \scrO_{\hX}) \cong  \Ext^1(\Omega^1_X, \scrO_X)$. The   map $\pi^*\Omega^1_X \to \Omega^1_{\hX}$ leads to a map $\Ext^1(\Omega^1_{\hX}, \scrO_{\hX}) \to \Ext^1(\pi^*\Omega^1_X, \scrO_{\hX})$. The map $H^1(\hX; T_{\hX}) \to H^0(X; T^1_X)$ is then   the composition of the natural maps
\begin{gather*}
H^1(\hX; T_{\hX}) =  \Ext^1(\Omega^1_{\hX}, \scrO_{\hX}) \to \Ext^1(\pi^*\Omega^1_X, \scrO_{\hX})\cong \Ext^1(\Omega^1_X, \scrO_X)  \\
 \cong  H^0(X; \mathit{Ext}^1(\Omega^1_X, \scrO_X)) = H^0(X; T^1_X).
\end{gather*}
 Another way to see the map $H^1(\hX; T_{\hX}) \to H^0(X; T^1_X)$ is as the composition
$$H^1(\hX; T_{\hX}) \to H^1(U; T_{\hX}|U)\cong H^0(X; T^1_X)$$
arising from Lemma~\ref{2.1}.

We have seen that, for $\dim X \geq 3$,   $T^0_X$ has depth at least $2$. For a 
local complete intersection $(X,0)\subseteq (\Cee^N,0)$,  we have the exact sequence $0 \to I/I^2 \to \Omega^1_{\Cee^N,0}|X \to \Omega^1_{X,0} \to 0$.  
Hence: 

\begin{lemma}\label{reflexive} If $(X,x)$ is an isolated local complete intersection of dimension $n$, then $\operatorname{depth} \Omega^1_{X,x} \geq n-1$. In particular, if $\dim X \geq 3$, then $\operatorname{depth} \Omega^1_{X,x} \geq 2$ and hence $\Omega^1_{X,x}$ is reflexive. Finally, $\pi_*\Omega^1_{\hX} = \Omega^1_X$.  
\end{lemma}
\begin{proof} We only need to check the last statement. We have a morphism $f\colon\Omega^1_X \to \pi_*\Omega^1_{\hX}$ which is an isomorphism away from $x$. Since $\Omega^1_X$ is reflexive and $\pi_*\Omega^1_{\hX}$ is torsion free, $f$ is an isomorphism. 
\end{proof}

\section{Local deformation theory}\label{Section2}

\subsection{Statement and proof of the main theorem} Let $(X,x)$ be the germ of an isolated canonical (hence rational) Gorenstein singularity  of dimension $n\geq 3$ and let $X$ be a contractible Stein representative for $(X,x)$. Let $\pi \colon \hX \to X$ be  a  good equivariant resolution of singularities of $X$, with normal crossings exceptional divisor $E$. Let $U = X -\{x\} = \hX - E$. Note that $H^0(X; T^1_X) \cong H^1(U; T_U)$. We can further identify $H^1(U; T_U)$ with $H^1(U; \Omega^{n-1}_{\hX}|U)$, $H^1(U; \Omega^{n-1}_{\hX}(\log E)|U)$ or $H^1(U;  \Omega^{n-1}_{\hX}(\log E)(-E)|U)$, and each such identification comes with an associated local cohomology exact sequence.  Our  first goal in this section is to prove the following:

\begin{theorem}\label{maintheorem} Suppose that either $\dim X = 3$ or $X$ is a local complete intersection.
\begin{enumerate} 
\item[\rm(i)] There is an exact sequence (arising from the long exact local cohomology sequence)
$$0 \to H^1(\hX; \Omega^{n-1}_{\hX}(\log E)) \to H^1(U; T_U) \to H^2_E( \hX; \Omega^{n-1}_{\hX}(\log E)) \to 0.$$
\item[\rm(ii)] There is an exact sequence
$$0 \to H^1(\hX; \Omega^{n-1}_{\hX}(\log E)(-E)) \to H^1(\hX; \Omega^{n-1}_{\hX}(\log E)) \to H^1(\hX; \Omega^{n-1}_{\hX}(\log E)|E)\to 0.$$
Thus in particular $\dim H^0(X; T^1_X) = b^{1, n-2} + b^{n-1, 1} + \ell^{n-1, 1}$ \cite[Theorem 4]{SteenbrinkDB}.
\item[\rm(iii)] The map $H^1(\hX; \Omega^{n-1}_{\hX}(\log E)(-E)) \to H^1(\hX; \Omega^{n-1}_{\hX})$ is injective and    $H^1(\hX; \Omega^{n-1}_{\hX})$  and $H^1(\hX; \Omega^{n-1}_{\hX}(\log E)(-E))$ have the same image in $H^1(U; T_U)$.  More precisely, there is a canonical isomorphism
$$H^1(\hX; \Omega^{n-1}_{\hX}) \cong H^1(\hX; \Omega^{n-1}_{\hX}(\log E)(-E)) \oplus H^1_E(\hX; \Omega^{n-1}_{\hX})$$
such that the map $H^1(\hX; \Omega^{n-1}_{\hX}) \to H^1(U; T_U)$ is the natural map on  $H^1(\hX; \Omega^{n-1}_{\hX}(\log E)(-E))$ and is $0$ on the $H^1_E(\hX; \Omega^{n-1}_{\hX})$ summand. 
Moreover $H^1_E(\hX; \Omega^{n-1}_{\hX}(\log E)(-E)) =0$, or equivalently $H^{n-1}(\hX; \Omega^1_{\hX}(\log E)) =0$.
\item[\rm(iv)] \cite{SteenbrinkDB} Let $A = \Ker\{d\colon H^3_E(\hX;\Omega^{n-3}_{\hX}(\log E)) \to H^3_E(\hX;\Omega^{n-2}_{\hX}(\log E))\}$. Then there is an exact sequence
$$0 \to H^2_E(\hX;\Omega^{n-2}_{\hX}(\log E)) \xrightarrow{d} H^2_E(\hX;\Omega^{n-1}_{\hX}(\log E))\to A\to 0.$$
Hence,  with   $b^{p,q} $ as in Definition~\ref{defDBlink} and $a =\dim A$, 
$$\dim H^0(T^1_X) = b^{n-1, 1} + b^{2, n-2} + \ell^{n-1, 1} + a.$$
Finally, if $X$ is a local complete intersection and $\dim X =3$, then the singularity $(X,x)$ has a good $\Cee^*$ action $\iff$ $A= 0$, or equivalently  $a=0$.
\item[\rm(v)] Let $K = \Ker\{H^2_E(\hX; \Omega^{n-1}_{\hX}) \to H^2(\hX; \Omega^{n-1}_{\hX})\}$. Then $K\cong H^2_E(\hX; \Omega^{n-1}_{\hX}(\log E)(-E))$, and  there are exact sequences
\begin{gather*}
0\to H^1(\hX; \Omega^{n-1}_{\hX}(\log E)(-E)) \to H^1(U; T_U) \to K \to 0;\\
0 \to \Gr^{n-1}_FH^n (L) \to K \to H^2_E( \hX; \Omega^{n-1}_{\hX}(\log E)) \to 0.
\end{gather*}
\item[\rm(vi)] There is a  natural map $H^2_E(\hX; \Omega^{n-1}_{\hX}) \to H^{n+1}_E(\hX)$, to be defined in the course of the proof. Let $K' =\Ker\{H^2_E(\hX; \Omega^{n-1}_{\hX}) \to H^{n+1}_E(\hX)\}$. Then there is an inclusion 
$$K'\oplus \Gr^{n-1}_FH^n (L) \hookrightarrow K.$$ 
If $K'\oplus \Gr^{n-1}_FH^n (L) \hookrightarrow K$ is an isomorphism, then  $K'\cong H^2_E( \hX; \Omega^{n-1}_{\hX}(\log E))$. In general, the  image of $H^2_E(\hX; \Omega^{n-1}_{\hX}) \to H^{n+1}_E(\hX)$ contains the subspace $H^1(\hX; \Omega^{n-1}_{\hX}(\log E) /\Omega^{n-1}_{\hX}) = \Gr_F^{n-1}H^{n+1}_E(\hX)$. If equality holds, then $K'\oplus \Gr^{n-1}_FH^n (L) \hookrightarrow K$ is an isomorphism and hence the second exact sequence in (v) splits: 
$$K \cong \Gr^{n-1}_FH^n (L) \oplus H^2_E( \hX; \Omega^{n-1}_{\hX}(\log E))\cong \Gr^{n-1}_FH^n (L) \oplus K'.$$
 Finally, if $\dim X =3$, then 
$\Gr_F^2H^4_E(\hX) = H^4_E(\hX)$ and   
$$K \cong \Gr^2_FH^3 (L) \oplus H^2_E( \hX; \Omega^2_{\hX}(\log E))= H^3(L)\oplus H^2_E( \hX; \Omega^2_{\hX}(\log E))\cong H^3(L) \oplus K'.$$
\end{enumerate} 
\end{theorem} 

\begin{remark} The map $H^1(\hX; \Omega^{n-1}_{\hX}(\log E)(-E)) \to H^1(U; T_U)$ already figures in \cite[Theorem 1.1]{NS} and the group $K$  in (v) above is $\im\tau$ in that notation. 
\end{remark} 

\begin{proof}

\noindent Proof of (i): The local cohomology exact sequence and the identification $H^1(U; T_U) \cong H^1(U; \Omega^{n-1}_{\hX}(\log E)|U)$ give an exact sequence
$$H^1_E(\Omega_{\hX}^{n-1}(\log E)) \to H^1( \Omega_{\hX}^{n-1}(\log E))\to H^1(U; T_U) \to H^2_E( \Omega_{\hX}^{n-1}(\log E)) \to H^2(\Omega^{n-1}_{\hX}(\log E)).$$
But $H^1_E(\hX; \Omega_{\hX}^{n-1}(\log E)) = 0$ by Corollary~\ref{1.5}. We claim that, under the assumptions of (i),  $H^2(\hX; \Omega_{\hX}^{n-1}(\log E)) = 0$, giving the exact sequence in (i). To see this, consider the long exact cohomology sequence (\ref{1.1}) for $p=n-1$. 
By Theorem~\ref{GNAPS}, $H^2(\hX; \Omega^{n-1}_{\hX}(\log E)(-E)) = H^3(\hX; \Omega^{n-1}_{\hX}(\log E)(-E)) =0$. Thus  
$$H^2(\hX; \Omega_{\hX}^{n-1}(\log E)) \cong H^2(\Omega^{n-1}_{\hX}(\log E)|E)= \Gr^{n-1}_FH^{n+1}(L).$$   If   $X$ is a local complete intersection, then $H^{n+1}(L) =0$,  and therefore  $H^2(\hX; \Omega_{\hX}^{n-1}(\log E)) = 0$. Likewise, if $\dim X =3$, then we have the link exact sequence 
$$H^4_E(\hX) \to H^4(E) \to H^4(L) \to H^5_E(\hX).$$
By semipurity, $H^4_E(\hX) \to H^4(E)$ is surjective. By duality, $H^5_E(\hX)$ is dual to $H_1(E)$. A somewhat involved direct argument (see  \cite[Proposition 2.1(iii)]{FL22b}) shows that $H^1(E) =0$ and thus $H_1(E) =0$. Hence $H^4(L) = 0$.

\smallskip
\noindent  Proof of (ii): 
 Using the exact sequence (\ref{1.1}) for $p=n-1$ gives
\begin{gather*}
H^0(\Omega^{n-1}_{\hX}(\log E)|E) \to  H^1(\hX; \Omega^{n-1}_{\hX}(\log E)(-E)) \to H^1(\hX; \Omega^{n-1}_{\hX}(\log E)) \\
\to H^1(\hX; \Omega^{n-1}_{\hX}(\log E)|E) 
\to H^2(\hX; \Omega^{n-1}_{\hX}(\log E)(-E)),
\end{gather*}
 with $H^2(\hX; \Omega^{n-1}_{\hX}(\log E)(-E)) =0$.
By \cite[Lemma 2]{SteenbrinkDB}, $H^0( \Omega^{n-1}_{\hX}(\log E)|E) = 0$, giving the exact sequence in (ii). (Note that we did not need the hypothesis $\dim X = 3$ or $X$ is a local complete intersection.) A direct argument that $H^0(\Omega^{n-1}_{\hX}(\log E)|E) = 0$ goes as follows:   by  semipurity,  there is an exact  sequence of mixed Hodge structures
$$0 \to    H^{n-1}_E(\hX) \to H^{n-1}(E) \to H^{n-1}(L) \to 0.$$  By the strictness of morphisms with respect to the Hodge filtration, the map $H^0(\Omega^{n-1}_E/\tau_E^{n-1})\to H^0(\Omega^{n-1}_{\hX}(\log E)|E)$ is surjective. By Corollary~\ref{Hn-1E},  $H^0(\Omega^{n-1}_E/\tau_E^{n-1}) =0$. Thus $H^0(\Omega^{n-1}_{\hX}(\log E)|E) = 0$ as well.

\smallskip
\noindent  Proof of (iii): There is a commutative diagram
$$\begin{CD}
@. @.   0 @. @. @.\\
@. @. @VVV @. @.\\
@. 0 @. H^1_E(\hX; \Omega^{n-1}_{\hX}) @. @.\\
@. @VVV @VVV @. @.\\
0 @>>> H^1(\hX; \Omega^{n-1}_{\hX}(\log E)(-E)) @>>> H^1(\hX; \Omega^{n-1}_{\hX}) @>>> H^1(E; \Omega^{n-1}_E/\tau^{n-1}_E) @>>>  0
\\
@. @VVV @VVV @. @.\\
@. H^1(U; T_U) @>{=}>> H^1(U; T_U). @. @. 
\end{CD}$$
The map $H^1(\hX; \Omega^{n-1}_{\hX}(\log E)(-E)) \to H^1(U; T_U)$ is injective since it is the composition of the maps
$$H^1(\hX; \Omega^{n-1}_{\hX}(\log E)(-E)) \to H^1(\hX; \Omega^{n-1}_{\hX}(\log E) )\to H^1(U; T_U),$$
and each of these maps is injective by (ii) and (i). Since the composition $$H^1(\hX; \Omega^{n-1}_{\hX}(\log E)(-E)) \to  H^1(\hX; \Omega^{n-1}_{\hX}) \to H^1(U; T_U)$$ is injective, the map $H^1(\hX; \Omega^{n-1}_{\hX}(\log E)(-E)) \to H^1(\hX; \Omega^{n-1}_{\hX})$ is injective as well.  The map $H^1_E(\hX; \Omega^{n-1}_{\hX}) \to H^1(\hX; \Omega^{n-1}_{\hX})$ is also injective since its kernel is the cokernel of the map 
$$H^0(\hX; \Omega^{n-1}_{\hX}) \to H^0(U; T_U),$$ which is an isomorphism by Lemma~\ref{imageH0}. Hence the kernel of the map $H^1(\hX; \Omega^{n-1}_{\hX}) \to H^1(U; T_U)$ is isomorphic to $H^1_E(\hX; \Omega^{n-1}_{\hX})$.   If we can show that the composition $H^1_E(\hX; \Omega^{n-1}_{\hX}) \to H^1(\hX; \Omega^{n-1}_{\hX}) \to  H^1(E; \Omega^{n-1}_E/\tau^{n-1}_E)$ is an isomorphism, then there is a splitting 
$$H^1(\hX; \Omega^{n-1}_{\hX}) \cong H^1(\hX; \Omega^{n-1}_{\hX}(\log E)(-E)) \oplus H^1_E(\hX; \Omega^{n-1}_{\hX}),$$
and in particular $H^1(\hX; \Omega^{n-1}_{\hX})$ and  $H^1(\hX; \Omega^{n-1}_{\hX}(\log E)(-E))$ have the same image in $H^1(U; T_U)$.

To see that $H^1_E(\hX; \Omega^{n-1}_{\hX}) \to   H^1(E; \Omega^{n-1}_E/\tau^{n-1}_E)$ is an isomorphism, we  first show that $$H^1_E(\hX; \Omega^{n-1}_{\hX}) \cong  H^0_E(\hX; \Omega^{n-1}_{\hX}(\log E)/\Omega^{n-1}_{\hX}).$$
 Consider the exact sequence
$$
0 \to \Omega^{n-1}_{\hX}  \to \Omega^{n-1}_{\hX}(\log E) \to \Omega^{n-1}_{\hX}(\log E)/\Omega^{n-1}_{\hX} 
   \to 0.$$
By taking the associated long exact local cohomology sequence, we get
$$0 = H^0_E( \Omega^{n-1}_{\hX}(\log E)) \to H^0_E(\Omega^{n-1}_{\hX}(\log E)/\Omega^{n-1}_{\hX}) \to H^1_E( \Omega^{n-1}_{\hX})  \to H^1_E( \Omega^{n-1}_{\hX}(\log E)).$$
By Corollary~\ref{1.5}, $H^1_E(\hX;  \Omega^{n-1}_{\hX}(\log E)) =0$, so that 
  $H^0_E(\hX; \Omega^{n-1}_{\hX}(\log E)/\Omega^{n-1}_{\hX}) \cong H^1_E( \hX; \Omega^{n-1}_{\hX}) $. By semipurity, there is an  isomorphism of mixed Hodge  structures $H^n_E(\hX) \cong H^n(E)$, and both are in fact pure Hodge structures of weight $n$. By the strictness of morphisms, we get an isomorphism $\Gr_F^{n-1} H^n_E(\hX)  \cong \Gr_F^{n-1} H^n(E)$. But $\Gr_F^{n-1} H^n_E(\hX) = H^0_E(\hX; \Omega^{n-1}_{\hX}(\log E)/\Omega^{n-1}_{\hX}) \cong H^1_E( \hX; \Omega^{n-1}_{\hX})$ and $\Gr_F^{n-1}  H^n(E)  =  H^1(E; \Omega^{n-1}_E/\tau^{n-1}_E)$. We now have two maps $H^1_E(\hX; \Omega^{n-1}_{\hX})   \to  H^1(E; \Omega^{n-1}_E/\tau^{n-1}_E)$: the first is the map given by the composition
$$H^1_E(\hX; \Omega^{n-1}_{\hX})    \to  H^1(\hX; \Omega^{n-1}_{\hX}) \to H^1(E; \Omega^{n-1}_E/\tau^{n-1}_E )$$ and the second is the composition of isomorphisms given by
$$H^1_E(\hX; \Omega^{n-1}_{\hX}) \cong H^0_E(\hX; \Omega^{n-1}_{\hX}(\log E)/\Omega^{n-1}_{\hX} )\to  H^1(E; \Omega^{n-1}_E/\tau^{n-1}_E ),$$
which factors as
$$ H^0_E(\hX; \Omega^{n-1}_{\hX}(\log E)/\Omega^{n-1}_{\hX} )\xrightarrow{\partial}  H^1(\hX; \Omega^{n-1}_{\hX}) \to  H^1(\hX; \Omega^{n-1}_{\hX}/ \Omega^{n-1}_{\hX}(\log E)(-E)) = H^1(E; \Omega^{n-1}_E/\tau^{n-1}_E ).$$
Thus the two maps are equal, and in particular the first map is an isomorphism as well.
This completes the proof of the second statement in (iii). Finally, we have the exact sequence
$$H^0(\Omega^{n-1}_{\hX}(\log E)(-E)) \to H^0(U; T_U) \to H^1_E( \Omega^{n-1}_{\hX}(\log E)(-E))\to H^1(\Omega^{n-1}_{\hX}(\log E)(-E)) \to H^1(U; T_U).$$
Since $H^1(\hX; \Omega^{n-1}_{\hX}(\log E)(-E)) \to H^1(U; T_U)$ is injective, to prove that $H^1_E(\hX; \Omega^{n-1}_{\hX}(\log E)(-E)) =0$, it suffices to prove that $H^0(\hX; \Omega^{n-1}_{\hX}(\log E)(-E)) \to H^0(U; T_U)$ is surjective. We have seen that $H^0(\hX; \Omega^{n-1}_{\hX}) \to H^0(U; T_U)$ is surjective, in fact an isomorphism, so it suffices to show that 
$$H^0(\hX; \Omega^{n-1}_{\hX}(\log E)(-E)) \to H^0(\hX; \Omega^{n-1}_{\hX})$$
is surjective. But the cokernel of this map is contained in $H^0(\Omega^{n-1}_E/\tau^{n-1}_E)$, which is $0$ by Corollary~\ref{Hn-1E}.
(Note that, in the proof of (iii),  we did not need the hypothesis $\dim X = 3$ or $X$ is a local complete intersection. The dual form of the last statement, that  $H^{n-1}(\hX; \Omega^1_{\hX}(\log E)) =0$ for an isolated rational singularity, is proved in \cite{MOP}. See also below.)

\smallskip
\noindent  Proof of (iv): The exact sequence was established in Corollary~\ref{1.5}. (Note that we did not need the hypothesis $\dim X = 3$ or $X$ is a local complete intersection for this part.) In particular, 
$$  \dim H^2_E(\hX;\Omega^{n-1}_{\hX}(\log E)) = \dim H^2_E(\hX;\Omega^{n-2}_{\hX}(\log E)) + \dim A.$$
By definition, $\dim A = a$. By duality $\dim H^q_E(\hX; \Omega^p_{\hX}(\log E)) = \dim H^{n-q}(\hX; \Omega^{n-p}_{\hX}(\log E)(-E))= b^{n-p, n-q}$. Thus $ \dim H^2_E(\hX;\Omega^{n-1}_{\hX}(\log E)) = b^{1, n-2} = b^{2, n-2} + a$. From the exact sequences in (i) and (ii), 
$$\dim H^0(T^1_X) = \dim H^0(\hX; \Omega^{n-1}_{\hX}(\log E)) + b^{1,n-2} = b^{n-1,1} + \ell^{n-1, 1}+ b^{2, n-2} + a.$$

 Let $\mu$ be the Milnor number of $(X,x)$ and $\tau$ the Tyurina number. Then, by \cite{SteenbrinkDB}, for an isolated rational (and hence Du Bois) local complete intersection singularity, $\mu -\tau = \alpha$, where $\alpha = \dim \Coker \{d\colon H^0(\hX; \Omega_{\hX}^{n-1}(\log E)(-E)) \to H^0(\hX; \Omega_{\hX}^n)\}$ (compare Proposition 3, Lemma 2, and the last line of Lemma 3 in \cite{SteenbrinkDB}).  Set $B= \Coker \{d\colon H^{n-3}(\hX; \Omega_{\hX}^2(\log E)(-E)) \to H^{n-3}(\hX; \Omega_{\hX}^3)\}$, so that $\alpha = \dim B$ in case $n=3$. By Lemma~\ref{ratlplus}, $H^i(\hX;\scrO_{\hX}(-E)) =0$.     By looking at the $E_2$ page of the spectral sequence with $E_1^{p,q}=H^q(\hX; \Omega^p_{\hX}(\log E) (-E))$, which converges to $0$ by Lemma~\ref{convzero}, there is an exact sequence (dual to that in Corollary~\ref{1.5})
 $$0 \to B \to H^{n-2}(\hX; \Omega^1_{\hX}(\log E) (-E)) \to H^{n-2}(\hX; \Omega^2_{\hX}(\log E) (-E)) \to 0.$$
Hence, in case $n=3$,   $\alpha = b^{1,1} - b^{2,1} = a$, by Corollary~\ref{1.5} and duality (and in fact $B\cong A\spcheck $).   Thus, for $n=3$,  $A = 0$ $\iff$ $a =0$ $\iff$ $\alpha =0$ $\iff$ $\mu =\tau$.  By a result due to K. Saito \cite{KSaito} for hypersurface singularities and Vosegaard \cite{Vosegaard} for general complete intersections, $\mu =\tau$ $\iff$ $(X,x)$ has a good $\Cee^*$ action, i.e.\ is quasihomogeneous.
 
 \smallskip
\noindent  Proof of (v): From the local cohomology sequence
$$ H^1(\hX; \Omega^{n-1}_{\hX}) \to H^1(U; T_U) \to H^2_E(\hX; \Omega^{n-1}_{\hX}) \to H^2(\hX; \Omega^{n-1}_{\hX})$$
and (iii), we get an exact sequence 
$$0\to H^1(\hX; \Omega^{n-1}_{\hX}(\log E)(-E)) \to H^1(U; T_U) \to K \to 0.$$
Thus $K$ is isomorphic to the cokernel of $H^1(\hX; \Omega^{n-1}_{\hX}(\log E)(-E)) \to H^1(U; T_U)$. But this cokernel  is $H^2_E(\hX; \Omega^{n-1}_{\hX}(\log E)(-E))$ as $H^2(\hX; \Omega^{n-1}_{\hX}(\log E)(-E))=0$. 
Now comparing this with the exact sequence in (i), there is a commutative diagram with exact rows
$$\begin{CD}
0@>>> H^1(\hX; \Omega^{n-1}_{\hX}(\log E)(-E)) @>>>  H^1(U; T_U) @>>>  K @>>> 0\\
@. @VVV @| @VVV @. \\ 
0 @>>> H^1(\hX; \Omega^{n-1}_{\hX}(\log E)) @>>> H^1(U; T_U) @>>> H^2_E( \hX; \Omega^{n-1}_{\hX}(\log E)) @>>> 0
\end{CD}$$
Here, $K \to  H^2_E( \hX; \Omega^{n-1}_{\hX}(\log E))$ is surjective because $H^1(U; T_U) \to H^2_E( \hX; \Omega^{n-1}_{\hX}(\log E))$ is surjective. Since $K \cong H^2_E(\hX; \Omega^{n-1}_{\hX}(\log E)(-E))$, we have an exact sequence
$$H^1_E(\hX; \Omega^{n-1}_{\hX}(\log E)) \to H^1(\hX; \Omega^{n-1}_{\hX}(\log E)|E) \to K \to H^2_E( \hX; \Omega^{n-1}_{\hX}(\log E))\to 0.$$
Since  $H^1_E(\hX; \Omega^{n-1}_{\hX}(\log E))=0$ by Corollary~\ref{1.5}, we get the    exact sequence   
$$0 \to H^1(\hX; \Omega^{n-1}_{\hX}(\log E)|E)) \to K \to H^2_E( \hX; \Omega^{n-1}_{\hX}(\log E)) \to 0,$$
which is  the second exact sequence in (v). 

 \smallskip
\noindent  Proof of (vi): We begin by defining the map $H^2_E(\hX; \Omega^{n-1}_{\hX}) \to H^{n+1}_E(\hX)$. By Lemma~\ref{ratlplus},  $H^1_E(\hX; \Omega^n_{\hX}) = H^2_E(\hX; \Omega^n_{\hX}) =0$. Thus, looking at the spectral sequence with $E_1^{p,q}$ term $H^q_E(\hX; \Omega^p_{\hX})$, which converges to $\mathbb{H}^{p+q}_E(\Omega^\bullet_{\hX}) = H^{p+q}_E(\hX)$, gives the map $H^2_E(\hX; \Omega^{n-1}_{\hX}) \to H^{n+1}_E(\hX)$.  As noted in \S1, $\mathbb{H}^k(E;\Omega^\bullet_{\hX}(\log E)/\Omega^\bullet_{\hX}) = H^{k+1}_E(\hX)$. Moreover, the associated spectral sequence degenerates at $E_1$ and computes the Hodge filtration on $H^{k+1}_E(\hX)$. Taking $k=n$,  we have $\Gr_F^{n+1}H^{n+1}_E(\hX) = 0$ for dimension reasons, and $\Gr_F^nH^{n+1}_E(\hX) = H^0(\hX; \Omega^n_{\hX}(\log E)/\Omega^n_{\hX}) = H^0(E; \omega_E)=0$, as $H^0(E; \omega_E)$ is Serre dual to $H^{n-1}(E; \scrO_E)=0$. Thus $H^1(\hX; \Omega^{n-1}_{\hX}(\log E)/\Omega^{n-1}_{\hX}) = \Gr_F^{n-1}H^{n+1}_E(\hX)$ includes into $H^{n+1}_E(\hX)$. 

From the exact sequence  
$$0 \to \Omega^\bullet_{\hX} \to \Omega^\bullet_{\hX}(\log E) \to \Omega^\bullet_{\hX}(\log E)/\Omega^\bullet_{\hX} \to 0$$
there is a long exact hypercohomology sequence
$$\cdots \to \mathbb{H}^{k-1}_E(\Omega^\bullet_{\hX}(\log E)/\Omega^\bullet_{\hX}) \to \mathbb{H}^k_E(\Omega^\bullet_{\hX}) \to \mathbb{H}^k_E(\Omega^\bullet_{\hX}(\log E)) \to \mathbb{H}^k_E(\Omega^\bullet_{\hX}(\log E)/\Omega^\bullet_{\hX}) \to \mathbb{H}^{k+1}_E(\Omega^\bullet_{\hX}) \to \cdots$$
Since $\mathbb{H}^k_E(\Omega^\bullet_{\hX}(\log E)) =0$ for all $k$, there is an isomorphism 
$$\mathbb{H}^{k-1}_E(\Omega^\bullet_{\hX}(\log E)/\Omega^\bullet_{\hX}) = \mathbb{H}^{k-1} (\Omega^\bullet_{\hX}(\log E)/\Omega^\bullet_{\hX}) \cong \mathbb{H}^k_E(\Omega^\bullet_{\hX})\cong H^k_E(\hX),$$
compatible with the corresponding spectral sequences. In particular, looking at the $E_1$ pages, there is a map 
 
 $$H^1(\hX; \Omega^{n-1}_{\hX}(\log E)/\Omega^{n-1}_{\hX}) \to  H^2_E(\hX; \Omega^{n-1}_{\hX})$$ 
for which the diagram 
$$\begin{CD}
H^1(\hX; \Omega^{n-1}_{\hX}(\log E)/\Omega^{n-1}_{\hX}) @>>> H^2_E(\hX; \Omega^{n-1}_{\hX}) \\
@VVV @VV{\beta}V\\
H^{n+1}_E(\hX) @>{=}>> H^{n+1}_E(\hX)
\end{CD}$$ 
is commutative. Thus the image of $\beta\colon H^2_E(\hX; \Omega^{n-1}_{\hX}) \to H^{n+1}_E(\hX)$  contains $\Gr_F^{n-1}H^{n+1}_E(\hX)$. In fact, the proof shows that the inclusion 
$$\Gr_F^{n-1}H^{n+1}_E(\hX) = H^1_E(\hX; \Omega^{n-1}_{\hX}(\log E)/\Omega^{n-1}_{\hX}) \to \im \beta$$ is given by the coboundary $\partial$ followed by the surjection $ H^2_E(\hX; \Omega^{n-1}_{\hX}) \to \im  \beta$. 

\begin{lemma} As before, let $K = \Ker \{H^2_E(\hX; \Omega^{n-1}_{\hX}) \to H^2(\hX; \Omega^{n-1}_{\hX})\}$, and let 
$$K'  =\Ker \{H^2_E(\hX; \Omega^{n-1}_{\hX}) \to  H^{n+1}_E(\hX )\}.$$
\begin{enumerate} 
\item[\rm(i)] The natural map $H^2(\hX; \Omega^{n-1}_{\hX}) \to H^{n+1}(\hX)$ is injective. 
\item[\rm(ii)] $K'$ is contained in $K$.
\end{enumerate}
\end{lemma}
\begin{proof} (i) By Corollary~\ref{Hn-1E}, $H^2(\hX; \Omega^{n-1}_{\hX}) \cong H^2(E; \Omega^{n-1}_E/\tau^{n-1}_E)$.  Since $\dim E = n-1$,
$$H^0(E; \Omega^{n+1}_E/\tau^{n+1}_E)= H^1(E; \Omega^n_E/\tau^n_E)=0.$$  As the Hodge spectral sequence $H^q(E; \Omega^p_E/\tau^p_E)\implies H^{p+q}(E)$ degenerates at $E_1$, there is an inclusion $H^2(E; \Omega^{n-1}_E/\tau^{n-1}_E) \hookrightarrow H^{n+1}(E)$. From the commutative diagram
$$\begin{CD}
H^2(\hX; \Omega^{n-1}_{\hX}) @>{\cong}>> H^2(E; \Omega^{n-1}_E/\tau^{n-1}_E) \\
@VVV @VVV\\
H^{n+1}(\hX) @>{\cong}>> H^{n+1}(E)
\end{CD}$$
and the fact that the right hand vertical map is injective, the left hand vertical map is injective as well.

\smallskip
\noindent (ii) There is a commutative diagram
$$\begin{CD}
H^2_E(\hX; \Omega^{n-1}_{\hX}) @>>> H^2(\hX; \Omega^{n-1}_{\hX}) \\
@VVV @VVV\\
H^{n+1}_E(\hX) @> >> H^{n+1}(\hX)
\end{CD}$$
where the right hand vertical map is injective  by (i). Hence $K'$, which by definition is the kernel of the left hand vertical map, is contained in the kernel of $H^2_E(\hX; \Omega^{n-1}_{\hX}) \to H^2(\hX; \Omega^{n-1}_{\hX})$, which by definition is $K$.
\end{proof}

Returning to the proof of (vi) of Theorem~\ref{maintheorem}, we have the two subspaces $K'$ and $\Gr_F^{n-1}H^n(L) = H^1(\Omega^{n-1}_{\hX}(\log E)|E)$ of $K$, where the second inclusion is given by the second exact sequence in (v).  We first show that $K'\cap \Gr_F^{n-1}H^n(L) =  0$. This follows from the commutative diagram 
$$\begin{CD}
H^1(\Omega^{n-1}_{\hX}(\log E)|E) @>>> K \\
@VVV @VVV\\
H^n(L) @>>> H^{n+1}_E(\hX)
\end{CD}$$
where the existence of the left hand vertical map follows because $\Gr_F^nH^n(L) = H^0(\Omega^n_{\hX}(\log E)|E)= H^0(E; \omega_E) = 0$. As noted in \S1, $\mathbb{H}^k(\Omega^\cdot_{\hX}(\log E)|E) \cong  H^k(L)$, and the associated spectral sequence degenerates at $E_1$.  Thus the left vertical arrow is the inclusion of $\Gr_F^{n-1}H^n(L)$ in $H^n(L)$, and in particular it  is injective. By Theorem~\ref{semipure}, $H^n_E(\hX) \to H^n(E)$ is surjective, so that  $H^n(L) \to  H^{n+1}_E(\hX)$ is injective by the link exact sequence (\ref{1.11}). Hence, if  $\xi \in K'\cap \Gr_F^{n-1}H^n(L)$, then $\xi$ maps to $0$ in $H^{n+1}_E(\hX)$ by the definition of $K'$, and thus $\xi =0$. In particular, the map $K'\oplus \Gr^{n-1}_FH^n (L) \to K$ is an inclusion.

Next suppose that the image of $H^2_E(\hX; \Omega^{n-1}_{\hX})$ in $H^{n+1}_E(\hX)$ is equal to $\Gr_F^{n-1}H^{n+1}_E(\hX)$. Then the coboundary map 
$$H^1(\hX; \Omega^{n-1}_{\hX}(\log E)/\Omega^{n-1}_{\hX}) = H^1_E(\hX; \Omega^{n-1}_{\hX}(\log E)/\Omega^{n-1}_{\hX}) \to H^2_E(\hX; \Omega^{n-1}_{\hX})$$ splits the surjection $H^2_E(\hX; \Omega^{n-1}_{\hX}) \to \Gr_F^{n-1}H^{n+1}_E(\hX)$. Thus  $H^2_E(\hX; \Omega^{n-1}_{\hX}) = K'\oplus  \Gr_F^{n-1}H^{n+1}_E(\hX)$, where by hypothesis $K' =\Ker \{H^2_E(\hX; \Omega^{n-1}_{\hX}) \to \Gr_F^{n-1}H^{n+1}_E(\hX )\}$. Hence   $K = K'\oplus K''$, where $K''$ is the kernel of 
$$\Gr_F^{n-1}H^{n+1}_E(\hX) \to H^2(\hX; \Omega^{n-1}_{\hX}) \cong H^2(E; \Omega^{n-1}_E/\tau^{n-1}_E) = \Gr_F^{n-1}H^{n+1}(E).$$
Thus $K''$ is the kernel of $\Gr_F^{n-1}H^{n+1}_E(\hX) \to \Gr_F^{n-1}H^{n+1}(E)$
which by semipurity is exactly $\Gr_F^{n-1}H^n(L) = H^1(\hX;\Omega^{n-1}_{\hX}(\log E)|E)$. If $K \cong \Gr_F^{n-1}H^n(L) \oplus K'$, then the surjection $K \to H^2_E( \hX; \Omega^{n-1}_{\hX}(\log E))$, whose kernel is $\Gr_F^{n-1}H^n(L)$,   identifies $K'$ with $H^2_E( \hX; \Omega^{n-1}_{\hX}(\log E))$. 

Finally, the statement that $\Gr_F^2H^4_E(\hX)= H^4_E(\hX)$ when $\dim X =3$ follows from the fact that
$$\Gr_F^3H^4_E(\hX)= H^0(\hX;\Omega^3_{\hX}(\log E)/ \Omega^3_{\hX}) = H^0(E; \omega_E) =0,$$
as $H^2(E; \scrO_E) =0$, and
$$\Gr_F^1H^4_E(\hX)= H^2(\hX;\Omega^1_{\hX}(\log E)/ \Omega^1_{\hX}) \cong \bigoplus_i H^2(E_i; \scrO_{E_i}) =0,$$
as $H^2(E; \scrO_E) =0$  and $H^2(E; \scrO_E) = \bigoplus_i H^2(E_i; \scrO_{E_i})$ by the Mayer-Vietoris spectral sequence.
\end{proof}

\begin{remark} As we shall see in Section~\ref{Section3}, the condition that $K'\oplus \Gr^{n-1}_FH^n (L) \cong K$ is satisfied in case $X$ is  a rational weighted homogeneous hypersurface singularity. It is natural to ask if $K'\oplus \Gr^{n-1}_FH^n (L) \cong K$ for an arbitrary rational local complete intersection singularity.
\end{remark}

\subsection{First order smoothings  and $1$-rationality}  Theorem~\ref{maintheorem} highlights  the relevance of the groups $K= \Ker \{H^2_E(\hX; \Omega^{n-1}_{\hX}) \to H^2(\hX; \Omega^{n-1}_{\hX})\} \cong H^2_E(\hX; \Omega^{n-1}_{\hX}(\log E)(-E))$, $H^2_E(\hX; \Omega^{n-1}_{\hX}(\log E)$, and   $\Gr_F^{n-1}H^n(L)$ to the deformation theory of $X$.   We now specialize to the case of a hypersurface singularity.  

\begin{lemma}\label{com1} Suppose that $(X,x)$ is moreover a hypersurface singularity and is not a smooth point, i.e.\ that $T^1_X\neq 0$. The following are equivalent:
\begin{enumerate}
\item[\rm(i)] $K\neq 0$, i.e.\ $H^2_E(\hX; \Omega^{n-1}_{\hX}(\log E)(-E))\neq 0$.   
\item[\rm(ii)] Either $\Gr^{n-1}_FH^n (L)=H^1(\hX; \Omega^{n-1}_{\hX}(\log E)|E)  \neq 0$ or $H^2_E( \hX; \Omega^{n-1}_{\hX}(\log E)) \neq 0$. 
\item[\rm(iii)] $ H^1(\hX; \Omega^{n-1}_{\hX}(\log E)(-E)) \neq H^0(X;T^1_X)$, or equivalently $\im H^1(\hX; \Omega^{n-1}_{\hX}) \neq H^0(X;T^1_X)$.
\item[\rm(iv)] $H^1(\hX; \Omega^{n-1}_{\hX}(\log E)(-E)) \subseteq \mathfrak{m}_x\cdot H^0(X;T^1_X)$,  or equivalently $\im H^1(\hX; \Omega^{n-1}_{\hX})\subseteq \mathfrak{m}_x\cdot H^0(X;T^1_X)$.
\end{enumerate}
Moreover,  the following are equivalent: 
\begin{enumerate}
\item[\rm(i)] $H^2_E(\hX; \Omega^{n-1}_{\hX}(\log E))\neq 0$.   
\item[\rm(ii)]  $ H^1(\hX; \Omega^{n-1}_{\hX}(\log E)) \neq H^0(X;T^1_X)$.
\item[\rm(iii)]  $H^1(\hX; \Omega^{n-1}_{\hX}(\log E)) \subseteq \mathfrak{m}_x\cdot H^0(X;T^1_X)$.
  \end{enumerate}
\end{lemma}
\begin{proof} We begin with the first set of equivalences:

\smallskip
\noindent
(i) $\iff$ (ii): This is clear from the exact sequence (Theorem~\ref{maintheorem}(v)) 
$$0 \to \Gr^{n-1}_FH^n (L) \to H^2_E(\hX; \Omega^{n-1}_{\hX}(\log E)(-E)) \to H^2_E( \hX; \Omega^{n-1}_{\hX}(\log E)) \to 0.$$

\smallskip
\noindent (ii) $\iff$ (iii):  This is clear from the exact sequence (Theorem~\ref{maintheorem}(v)) 
$$0\to H^1(\hX; \Omega^{n-1}_{\hX}(\log E)(-E)) \to H^0(X;T^1_X) \to H^2_E(\hX; \Omega^{n-1}_{\hX}(\log E)(-E)) \to 0.$$

\smallskip
\noindent (iii) $\iff$ (iv): 
 (iv) $\implies$ (iii) is obvious, using Theorem~\ref{maintheorem}(iii) for the second part. Conversely, suppose that $H^1(\hX; \Omega^{n-1}_{\hX}(\log E)(-E))$ (or equivalently the image of $H^1(\hX; \Omega^{n-1}_{\hX})$) is not contained in $\mathfrak{m}_x\cdot H^0(X;T^1_X)$.  Then  $H^1(\hX; \Omega^{n-1}_{\hX}(\log E)(-E))$ contains a generator of the cyclic $\scrO_{X,x}$-module $H^0(X;T^1_X)$. As $H^1(\hX; \Omega^{n-1}_{\hX}(\log E)(-E))$ is   an  $\scrO_{X,x}$-submodule of $ H^0(X;T^1_X)$, it  follows that $ H^1(\hX; \Omega^{n-1}_{\hX}(\log E)(-E)) = H^0(X;T^1_X)$.
 
 The proof for the second set of equivalences is similar, using Theorem~\ref{maintheorem}(i) as a point of departure.
\end{proof}

Given the above lemma and Theorem~\ref{maintheorem}, we make a definition which captures when certain terms in the various exact sequences in Theorem~\ref{maintheorem} vanish. We will relate this definition to more standard definitions in Proposition~\ref{relatedefs} and Remark~\ref{altDBdef} below.

\begin{definition}\label{defsub} Let $(X,x)$ be an isolated rational lci singularity. 
\begin{enumerate}
\item[\rm(i)]   $(X,x)$ is \textsl{$1$-rational} if $K= 0$, where as before  
$$K  = \Ker \{H^2_E(\hX; \Omega^{n-1}_{\hX}) \to H^2(\hX; \Omega^{n-1}_{\hX})\}\cong H^2_E(\hX; \Omega^{n-1}_{\hX}(\log E)(-E)).$$ By duality, $(X,x)$ is $1$-rational $\iff$ $H^{n-2}(\hX; \Omega^1_{\hX}(\log E)) = 0$. A singularity which is not $1$-rational is \textsl{$1$-irrational}. Thus, $(X,x)$ is $1$-irrational $\iff$ $K\neq 0$. 

\item[\rm(ii)] $(X,x)$ is \textsl{$1$-Du Bois} if $H^2_E(\hX; \Omega^{n-1}_{\hX}(\log E)) = 0$. By duality, $(X,x)$ is $1$-Du Bois $\iff$ $H^{n-2}(\hX; \Omega^1_{\hX}(\log E)(-E)) = 0$ $\iff$ $b^{1, n-2} =0$. By Theorem~\ref{maintheorem}(v), if $(X,x)$ is $1$-rational, then $(X,x)$ is  $1$-Du Bois.
  
 \item[\rm(iii)]  $(X,x)$ is \textsl{$1$-liminal} if it is $1$-Du Bois but   $1$-irrational, or equivalently if $H^2_E(\hX; \Omega^{n-1}_{\hX}(\log E)) =0$ but $\Gr_F^{n-1}H^n(L) \neq 0$. Note that in this case $\Gr_F^{n-1}H^n(L) \cong K$.

\item[\rm(iv)]  $(X,x)$ is \textsl{strongly $1$-irrational} if it is not $1$-Du Bois, hence is $1$-irrational,  and the map $K'\oplus \Gr_F^{n-1}H^n(L) \to K$ defined in Theorem~\ref{maintheorem}(vi) is an isomorphism.  Note that, if   $(X,x)$ is  strongly $1$-irrational, then $K'\cong H^2_E(\hX; \Omega^{n-1}_{\hX}(\log E))$ by Theorem~\ref{maintheorem}(vi), and $H^2_E(\hX; \Omega^{n-1}_{\hX}(\log E)) \neq 0$ since $(X,x)$ is not $1$-Du Bois. Thus, if  $(X,x)$ is  strongly $1$-irrational, then $K' \neq 0$. 
\end{enumerate}
\end{definition}

We can rephrase part of Lemma~\ref{com1} as follows:

\begin{lemma}\label{com2} Let $X$ be an isolated rational hypersurface singularity, not a smooth point. 
\begin{enumerate}
\item[\rm(i)] $X$ is   $1$-irrational $\iff$ $H^1(\hX; \Omega^{n-1}_{\hX}(\log E)(-E)) \subseteq \mathfrak{m}_x\cdot H^0(X;T^1_X)$.
\item[\rm(ii)] $X$ is not $1$-Du Bois $\iff$ $H^1(\hX; \Omega^{n-1}_{\hX}(\log E)) \subseteq \mathfrak{m}_x\cdot H^0(X;T^1_X)$. \qed
\end{enumerate}
\end{lemma}

 In terms of the deformation theory of \S\ref{Section1}, we have the following:

\begin{corollary}\label{meaning} Let $X$ be an isolated rational hypersurface singularity. 
\begin{enumerate}
\item[\rm(i)] If $X$ is   $1$-irrational, then any subspace $T$ of $H^0(X;T^1_X)$ mapping onto $K$ contains a first order smoothing.
\item[\rm(ii)] If $X$ is strongly $1$-irrational, then any subspace $T$ of $H^0(X;T^1_X)$ mapping onto $K'$ contains a first order smoothing.  
\end{enumerate}
\end{corollary}
\begin{proof} (i)  Let $u\in T$ map onto the image of a generator $u_0$ of $H^0(X;T^1_X)$. Then $u$ and $u_0$ differ by an element of  $H^1(\hX; \Omega^{n-1}_{\hX}(\log E)(-E))$ and hence by an element of $\mathfrak{m}_x\cdot H^0(X;T^1_X)$. Thus   $u \notin \mathfrak{m}_x\cdot H^0(X;T^1_X)$. By definition, $u$ is a first order smoothing.

\smallskip
\noindent (ii) By Theorem~\ref{maintheorem}(vi), if $X$ is strongly $1$-irrational, then $K' \cong H^2_E(\hX; \Omega^{n-1}_{\hX}(\log E))$ and there is an exact sequence
$$0 \to H^1(\hX; \Omega^{n-1}_{\hX}(\log E))  \to H^0(X;T^1_X) \to K'\to 0.$$
The argument then proceeds as in (i).
\end{proof} 

We shall in fact use a slight variant of this idea in \S\ref{Section4} and \S\ref{Section5}.

\subsection{Examples and further comments} We begin with a series of explicit examples of $1$-irrational and $1$-liminal singularities.

\begin{example}   For $n\geq 3$, let $f(z) = \sum_{i=1}^{n+1}z_i^d$, or more generally the polynomial defining the affine cone over a smooth hypersurface in $\Pee^n$ of degree $d$, and let $X$ be the corresponding germ of $V(f)$ at $0$. As we shall show in the next section (Corollary~\ref{wtscor}),  $X$ is rational $\iff$ $d< n+1$, $X$ is   $1$-irrational $\iff$ $d\geq \frac12(n+1)$, $X$ is strongly $1$-irrational $\iff$ $d> \frac12(n+1)$, and $X$ is $1$-liminal  if $d= \frac12(n+1)$. In particular, an ordinary double point is $1$-rational if $n> 3$, and our methods do not apply to such singularities. 

 Again by  Corollary~\ref{wtscor}, if $n$ is odd and $n\ge 3$, say $n=2k-1$ with $k\ge 2$, then the affine cone over a smooth hypersurface of degree $k$ in $\Pee^n$ is $1$-liminal. Likewise, in case $n$ is even and $n\geq 4$, say $n =2k$,  an example of a $1$-liminal  singularity is given by the weighted homogeneous polynomial $$f(z) = z_1^k + \cdots + z_{n-1}^k + z_n^{2k} + z_{n+1}^{2k}.$$
\end{example}

 The following shows that our definitions of $1$-rational and $1$-Du Bois agree with the definitions of \cite[\S4]{KL2} and \cite{FL22c}  for the $1$-rational case, and \cite{MOPW}, \cite[(1)]{JKSY-duBois}, \cite{MP-Mem} for the  $1$-Du Bois case. (See \cite[Theorem 5.2(iii) and Theorem 5.3(iii)]{FL22d} for a generalization.)
 
 \begin{proposition}\label{relatedefs} \begin{enumerate} \item[\rm(i)] An isolated rational lci singularity $X$ is $1$-rational  $\iff$ $H^q(\hX; \scrO_{\hX}) = H^q(\hX; \Omega^1_{\hX}(\log E)) = 0$ for all $q>0$. 
 
 \item[\rm(ii)] An isolated rational lci singularity $X$ is $1$-Du Bois in the sense of Definition~\ref{defsub} $\iff$ $H^q(\hX; \scrO_{\hX}(-E)) = H^q(\hX; \Omega^1_{\hX}(\log E)(-E)) = 0$ for all $q>0$.
 \end{enumerate}
 \end{proposition}
 \begin{proof} Proof of (i):  The $\impliedby$ implication is clear. Conversely, if $X$ is $1$-rational, then by definition $H^{n-2}(\hX; \Omega^1_{\hX}(\log E)) = 0$. By Theorem~\ref{maintheorem}(iii), since $X$ is rational, $H^1_E(\hX; \Omega^{n-1}_{\hX}(\log E)(-E)) =0$, and hence by duality $H^{n-1}(\hX; \Omega^1_{\hX}(\log E)) =0$. Then it follows from the exact sequence (\ref{1.1}) and Theorem~\ref{numervan} that $H^q(\hX; \Omega^1_{\hX}(\log E)) =0$ for $0< q < n-2$. Thus $H^q(\hX; \scrO_{\hX}) = H^q(\hX; \Omega^1_{\hX}(\log E)) = 0$ for all $q>0$. 
 
\smallskip
\noindent Proof of (ii):  Again, the $\impliedby$ implication is clear. Conversely, if $X$ is $1$-Du Bois, then it is rational. By Corollary~\ref{1.5}, $H^1_E(\hX; \Omega^{n-1}_{\hX}(\log E)) =0$, and hence $H^{n-1}(\hX; \Omega^1_{\hX}(\log E)(-E)) = 0$. Theorem~\ref{numervan} implies that that $H^q(\hX; \Omega^1_{\hX}(\log E)(-E)) =0$ for $0< q < n-2$. Thus $ H^q(\hX; \Omega^1_{\hX}(\log E)(-E)) = 0$ for all $q>0$.
 \end{proof}
 
 \begin{remark}\label{altDBdef} For an algebraic variety $X$, let $\underline{\Omega}_X^\bullet$ denote the \textsl{filtered de Rham} or \textsl{Deligne-Du Bois} complex of $X$, which is an object of the filtered derived category (see e.g.\ \cite[\S7.3]{PS}). Then \cite{MOPW}, \cite[(1)]{JKSY-duBois}, \cite{MP-Mem} have defined $X$ to have a \textsl{$1$-Du Bois singularity} if the natural maps $\scrO_X \to \underline{\Omega}_X^0$ and $\Omega^1_X \to \underline{\Omega}_X^1$ are isomorphisms. By \cite[Remark 3.20]{FL22c}, if  $X$ has an isolated   lci singularity and $\dim X \ge 3$, then $X$ is rational and $1$-Du Bois in the sense of Definition~\ref{defsub} $\iff$ $X$ is $1$-Du Bois in the above sense.   
  \end{remark} 
 
 The following is due to Namikawa-Steenbrink \cite[Theorem 2.2]{NS}. A proof in the above language may be found in \cite[Corollary 6.12]{FL22d}:
 
 \begin{proposition} Let $X$ be an isolated rational hypersurface singularity    with $\dim X = 3$. Then $X$  is $1$-irrational. Moreover,   $X$ is $1$-liminal $\iff$  $X$ is $1$-Du Bois  $\iff$ $X$ is an ordinary double point. \qed
 \end{proposition}
 
 The last sentence in Theorem~\ref{maintheorem}(vi) then implies: 
 
  \begin{corollary}\label{gooddim3cor}  Let $X$ be an isolated rational hypersurface singularity    with $\dim X = 3$. Then either  $X$ is $1$-liminal, in which case  $X$ is an ordinary double point, or $X$ is strongly $1$-irrational. \qed
 \end{corollary}

 An ordinary double point satisfies $\dim T^1_{X,x} = 1$. As we shall see, the following is  an analogue for $1$-liminal singularities in higher dimensions (\cite[Corollary 6.14]{FL22d}, where the dual statement is proved):
 
 \begin{proposition}\label{limdimone} Let $X$ be an isolated  $1$-liminal  hypersurface singularity   with $\dim X \geq  3$. Then $\ell^{n-1,  1} = 1$, i.e.\ $\dim \Gr^{n-1}_FH^n (L) = 1$, and the map  
$$\Gr_F^{n-1}H^n(L) \to K = H^2_E (\hX;\Omega_{\hX}^{n-1}(\log E)(-E)) $$ is an isomorphism. \qed
 \end{proposition}
  
\section{The case of a weighted homogeneous hypersurface singularity}\label{Section3}

Our goal in this section is to make the condition  of $1$-irrationality explicit for a weighted homogeneous hypersurface singularity, and to sharpen some of the results of Section~\ref{Section2}. Let  $X\subseteq (\Cee^{n+1}, 0)$ be the germ of an isolated, not necessarily  rational  hypersurface  singularity of dimension $n\geq 3$ with a good $\Cee^*$-action. We may as well assume that $X$ is affine, say $X \subseteq \Cee^{n+1}$, with an isolated singularity at $0\in X$, and in particular that $X$ is not smooth.  Let $\Cee^*$ act  on  $\Cee^{n+1}$ with integer weights $a_1, \dots, a_{n+1}> 0$ and $\gcd(a_1, \dots, a_{n+1}) = 1$, and suppose that $X =V(f)$, where $f$ is weighted homogeneous of degree $d>1$. Let $\hX \to X$ be a good equivariant resolution of $X$, which we may assume is equivariant for the $\Cee^*$-action, with exceptional set $E$ and $U = \hX -E = X -\{0\}$. We set   $N = \sum_{i=1}^{n+1}a_i - d$, so that the canonical bundle of the hypersurface $E^\#$ defined by $f$ in the corresponding weighted projective space is $\scrO_{E^\#}(-N)$. 

The $\Cee^*$-action on $X$ induces a  $\Cee^*$-action on $H^0(X;T^1_X) \cong H^1(U; T_U)$. Thus $H^0(X;T^1_X) =\bigoplus_{k\in \Zee}H^0(X;T^1_X)(k)$.  
For a multi-index $\alpha = (\alpha_1, \dots, \alpha_{n+1})$,   let $z^\alpha= z_1^{\alpha_1}\cdots z_{n+1}^{\alpha_{n+1}}$.  Fix an index set $A$ of multi-indices $\alpha$ such that $\{z^\alpha: \alpha \in A\}$ maps to a basis of $H^0(X;T^1_X) \cong \Cee\{z_1, \dots, z_{n+1}\}/J(f)$. Define $w_i=a_i/d\in \Q$, so that $\sum_iw_i = N/d + 1$. Set
$$\ell(\alpha) = \sum_{i=1}^{n+1}(\alpha_i+1)w_i.$$
For future reference, we note the following (easily checked):

\begin{lemma}\label{Ceestar1} \begin{enumerate} \item[\rm(i)] The $\Cee^*$-weight of $z^\alpha$, viewed as an element of $H^0(X;T^1_X)$, is
$$\sum_ia_i\alpha_i - d = d\left(\ell(\alpha)-\sum_iw_i -1\right) = d(\ell(\alpha) -2) -N.$$
\item[\rm(ii)] The $\Cee^*$-weight of $z^\alpha$ is $-N$ $\iff$ $\ell(\alpha) = 2$. \qed
\end{enumerate}
\end{lemma}

\begin{remark}\label{Cstaridentify}
There is a $\Cee^*$-equivariant isomorphism
$$H^1(U; T_U\otimes \Omega^n_U) \to H^1(U; \Omega^{n-1}_U).$$
Here, the line bundle $\Omega^n_U$ is trivialized by the section $\omega$ defined by $dz_1\wedge\cdots\wedge dz_{n+1}/f$, of $\Cee^*$ weight $N = \sum_ia_i - d$. Thus, under the image of the isomorphisms
$$H^1(U; \Omega^{n-1}_U)  = H^1(U; T_U)\cdot\omega \cong H^1(U; T_U),$$ the image in $H^1(U; T_U)$ of the  weight space $H^1(U; \Omega^{n-1}_U)(a)$ is the weight space $H^1(U; T_U)(a-N)$.
\end{remark}

Our aim in this section, inspired by a result of Wahl \cite[Corollary 3.8]{Wahl},  is to prove:

\begin{theorem}\label{wtdhmg} The image of the  map 
$$H^1(\hX; \Omega^{n-1}_{\hX}(\log E)) \to H^1(U; \Omega^{n-1}_{\hX}(\log E)|U) \cong H^1(U; T_U)\cong H^0(X;T^1_X)$$ is the space $\bigoplus_{k\geq -N}H^0(X;T^1_X)(k)$. 
\end{theorem}

\begin{remark} Neither the hypothesis that $X$ is rational nor that it is a hypersurface singularity play an essential role.
\end{remark}

As for the subspace $H^0(X;T^1_X)(-N)$ of weight exactly $-N$, we have the following:

\begin{theorem}\label{wtdlink} $H^0(X;T^1_X)(-N) = \Gr_F^{n-1}H^n(L)$.   
\end{theorem}

\begin{corollary}\label{wtdcor} If $(X,x)$ is a rational isolated weighted homogeneous hypersurface singularity with $\dim X \geq 3$, then
\begin{align*}
  \bigoplus_{k> -N}H^0(X;T^1_X)(k)&= H^1(\hX; \Omega^{n-1}_{\hX}(\log E)(-E));\\
    \bigoplus_{k\leq -N}H^0(X;T^1_X)(k)&= K;\\
  \bigoplus_{k< -N}H^0(X;T^1_X)(k)&=  H^2_E(\hX; \Omega^{n-1}_{\hX}(\log E)),
\end{align*} 
and  finally
$$K \cong K' \oplus\Gr^{n-1}_FH^n (L),$$
where $K' =\Ker \{H^2_E(\hX; \Omega^{n-1}_{\hX}) \to H^{n+1}_E(\hX)\}$ is as defined in Theorem~\ref{maintheorem}(vi).
\end{corollary}
\begin{proof} By Theorem~\ref{maintheorem}(ii), we have an exact sequence
$$0 \to H^1(\hX; \Omega^{n-1}_{\hX}(\log E)(-E)) \to H^1(\hX; \Omega^{n-1}_{\hX}(\log E)) \to \Gr^{n-1}_FH^n (L)\to 0.$$
By Theorem~\ref{wtdhmg}, $H^1(\hX; \Omega^{n-1}_{\hX}(\log E)) = \bigoplus_{k\geq -N}H^0(X;T^1_X)(k)$. By Theorem~\ref{wtdlink},  $\Gr^{n-1}_FH^n (L) = H^0(X;T^1_X)(-N)$. Hence $H^1(\hX; \Omega^{n-1}_{\hX}(\log E)(-E)) = \bigoplus_{k> -N}H^0(X;T^1_X)(k)$. Using the exact sequence 
$$0 \to H^1(\hX; \Omega^{n-1}_{\hX}(\log E)(-E) )\to H^0(X; T^1_X) \to K \to 0$$
of Theorem~\ref{maintheorem}(v), it follows that $K \cong \bigoplus_{k\leq -N}H^0(X;T^1_X)(k)$. Likewise $H^2_E(\hX; \Omega^{n-1}_{\hX}(\log E)) \cong   \bigoplus_{k< -N}H^0(X;T^1_X)(k)$ as follows   from the exact sequence
$$0 \to \Gr_F^{n-1}H^n(L) \to K \to H^2_E(\hX; \Omega^{n-1}_{\hX}(\log E)) \to 0.$$
Finally, since the $\Cee^*$-action on $H^{n+1}_E(\hX)$ is trivial, and hence becomes of weight $-N$ after the   shift of Remark~\ref{Cstaridentify}, $\bigoplus_{k< -N}H^0(X;T^1_X)(k)  \subseteq K'$. Thus  $K' \oplus\Gr^{n-1}_FH^n (L)$ contains and is contained in $\bigoplus_{k\leq -N}H^0(X;T^1_X)(k) =K$, and hence  $K' \oplus\Gr^{n-1}_FH^n (L)=K$.    
\end{proof}

\begin{corollary}\label{wtscor}  $X$ is $1$-irrational $\iff$  $\sum_iw_i \leq 2$, $X$ is not $1$-Du Bois $\iff$ $X$ is strongly $1$-irrational  $\iff$  $\sum_iw_i < 2$ and $X$ is $1$-liminal $\iff$  $\sum_iw_i = 2$.
\end{corollary}
\begin{proof} Since the smallest weight of $T^1_{X,x}$ is $-d$, corresponding to $1\in \Cee[z_1, \dots, z_{n+1}]/J(f)$, Corollary~\ref{wtdcor} implies  that $X$ is $1$-irrational $\iff$ $-d \leq -N$ $\iff$ $N/d \leq 1$ $\iff$  $d \geq \frac12\sum_ia_i $ $\iff$ $\sum_iw_i \leq 2$, and similarly $X$ is not $1$-Du Bois $\iff$ $d >\frac12\sum_ia_i $  $\iff$ $\sum_iw_i < 2$ and $X$ is $1$-liminal $\iff$ $d =\frac12\sum_ia_i $  $\iff$ $\sum_iw_i = 2$. Finally, if $X$ is not $1$-Du Bois, then it is   strongly $1$-irrational by Corollary~\ref{wtdcor} and Definition~\ref{defsub}(iv). 
\end{proof}

\begin{remark}\label{rem4.6} (i) Let $\widetilde\alpha_X$ be the minimal exponent as defined by Saito  \cite{Saito-b}. By a result of Saito \cite[(2.5.1)]{SaitoV}, in the weighted homogeneous case $\widetilde\alpha_X = \sum_iw_i$; this follows directly in  case $X$ is an   isolated weighted homogeneous hypersurface singularity, as then $\widetilde\alpha_X = \ell(0) = \sum_iw_i$. Thus we recover  the numerical conditions that  $X$ is $1$-irrational $\iff$ $\widetilde\alpha_X \leq 2$, $X$ is not $1$-Du Bois $\iff$ $\widetilde\alpha_X <2$ \cite[Thm. 1]{JKSY-duBois}, \cite[Thm. 1.1.]{MOPW}, and $X$ is $1$-liminal $\iff$ $\widetilde\alpha_X = 2$. See \cite[\S6]{FL22d} for more details.

\smallskip
\noindent (ii) For the germ of an isolated hypersurface singularity defined by $f=0$ in $(\Cee^{n+1},0)$, we have the Milnor algebra $Q_f =\Omega^{n+1}_{\Cee^{n+1}, 0}/df\wedge \Omega^n_{\Cee^{n+1}, 0}$. It carries a decreasing filtration $V^bQ_f$, indexed by rational numbers $b\in \Q$.
In the weighted homogeneous case, $Q_f = T^1_{X,x}$ and $T^1_{X,x}(k) = Q_f(k) \cong \Gr_V^bQ_f$ by e.g.\  \cite{ScherkSteenbrink}, where 
$$b = \ell(\alpha) - 1 = \frac{k+N}{d} +1.$$ 
Thus Corollary~\ref{wtdcor} identifies $V^1Q_f$ with the image of the map $H^1(\hX; \Omega^{n-1}_{\hX}(\log E)) \to T^1_{X,x}$ and  $V^{>1}Q_f$ with the image of the map $H^1(\hX; \Omega^{n-1}_{\hX}(\log E)(-E)) \to T^1_{X,x}$.   

In the general, not necessarily weighted homogeneous case, $T^1_{X,x} = Q_f/fQ_f$ and $fQ_f \subseteq V^{>1}Q_f$ in case $X$ is rational, since then $Q_f = V^{>0}Q_f$. It is then natural to ask if  the image of the map $H^1(\hX; \Omega^{n-1}_{\hX}(\log E)) \to T^1_{X,x}$ is $V^1T^1_{X,x}$ and if the image of the map $H^1(\hX; \Omega^{n-1}_{\hX}(\log E)(-E)) \to T^1_{X,x}$ is equal to $V^{>1}T^1_{X,x}$ in the general case.    
\end{remark}

\begin{corollary} With $X$ as above, suppose that $\pi\colon \hX \to X$ is a crepant resolution. Then the image of  $H^1(\hX; T_{\hX})$ in $H^0(X; T^1_X)$ is  $\bigoplus_{k> -N}H^0(X;T^1_X)(k)$.
\end{corollary}
\begin{proof} In this case, this image is the same as that of $H^1(\hX; \Omega^{n-1}_{\hX})$ and hence, by Theorem~\ref{maintheorem}(iii),  of $H^1(\hX; \Omega^{n-1}_{\hX}(\log E)(-E))$.
\end{proof}

\begin{proof}[Proof of Theorem~\ref{wtdhmg}]
Note that $X$ is an affine algebraic variety, and we may assume that $\hX$ is a scheme. For the remainder of the proof of Theorem~\ref{wtdhmg},  as opposed to the tacit  conventions in the rest of this paper, all sheaves,  cohomology and higher direct images will be taken with respect to the \textbf{Zariski topology} and all stacks will be algebraic.

We begin by recalling the weighted blowup $X^\#$ of $X$. Let $R = \Cee[z_1, \dots, z_{n+1}]/(f)$ be the homogeneous coordinate ring of $X$, with the induced grading, and hence $X =\Spec R$. Note that $X$ is normal since $X$ has an isolated singularity.  Let $E^\# = \Proj R$. Then $E^\#$ is a hypersurface  in the weighted projective space $WP^n= \Proj \Cee[z_1, \dots, z_{n+1}]$. As usual, there are the sheaves $\scrO_{E^\#}(k)$, $k\in \Zee$ and homomorphisms $\scrO_{E^\#}(k) \otimes \scrO_{E^\#}(\ell) \to \scrO_{E^\#}(k+\ell)$. Thus $\bigoplus_{k\in \Zee}\scrO_{E^\#}(k)$ and $\bigoplus_{k\geq 0} \scrO_{E^\#}(k)$ are sheaves of graded $\scrO_{E^\#}$-algebras. Define 
$$X^\# = \mathbf{Spec}_{E^\#}\bigoplus_{k\geq 0} \scrO_{E^\#}(k),$$
and let $\rho\colon X^\# \to E^\#$ be the natural morphism. It is easy to check that   $X^\#$ is normal and that  all fibers of $\rho$ have dimension one, although $\rho$ is not in general flat.

We then have the following facts  (cf.\ EGA II \cite[8.6.2, 8.2.7]{EGAII}):

\begin{proposition}\begin{enumerate} \item[\rm(i)] $E^\# = \mathbf{Proj}_{E^\#}\bigoplus_{k\geq 0} \scrO_{E^\#}(k)$.
\item[\rm(ii)] $\mathbf{Spec}_{E^\#}\bigoplus_{k\in \Zee} \scrO_{E^\#}(k) = X-\{0\} = X^\# - E^\#=U$. In particular, $X^\# - E^\#$ is smooth. 
\item[\rm(iii)] $\rho_*\scrO_{X^\#} = \bigoplus_{k\geq 0} \scrO_{E^\#}(k)$.
\item[\rm(iv)] If we also denote  by $\rho$ the morphism $U=X^\#-E^\# \to E^\#$, 
 $\rho_*\scrO_{X^\#-E^\#} = \rho_*\scrO_U = \bigoplus_{k\in \Zee} \scrO_{E^\#}(k)$.
\qed 
\end{enumerate}
\end{proposition}

Next, we outline the steps in the proof of Theorem~\ref{wtdhmg}. To explain the argument below, we first recall the basic setup of Wahl \cite{Wahl}: if $D$ is a smooth scheme or complex manifold and $\mathbb{V} = \mathbb{V}(L)$ is the total space of a line bundle on $D$, with $D$ embedded in $\mathbb{V}$ as the zero section, then $\Omega^1_{\mathbb{V}}(\log D)$ is the pullback of the vector bundle $\Omega^1_{\mathbb{V}}(\log D)|D$ on $D$. There is also a more precise description of the vector bundle $\Omega^1_{\mathbb{V}}(\log D)|D$:  Poincar\'e residue $\Omega^1_{\mathbb{V}}(\log D) \to \scrO_D$ defines an extension
 $$0\to \Omega^1_D \to \Omega^1_{\mathbb{V}}(\log D)|D \to \scrO_D \to 0,$$
 whose extension class is $c_1(L)$. As we shall show,   $X^\#$ is an orbifold and $E^\#$ is an orbifold smooth  divisor in $X^\#$. We would like to replicate the argument above for $X^\#$  and $E^\#$: morally, $X^\# = \mathbb{V}(\scrO_{E^\#}(1))$ with $E^\#$ embedded as the zero section, and hence $\Omega^{n-1}_{X^\#}(\log E^\#) = \rho^*W$ for some vector bundle $W$ on $X^\#$ (although we don't need to have a concrete description of $W$). If this were indeed the case, then, for the affine morphism $\rho\colon X^\# \to E^\#$,  
 $$H^1(X^\#; \Omega^{n-1}_{X^\#}(\log E^\#)) = H^1(E^\#; \rho_*\rho^*W) = \bigoplus_{k\geq 0} H^1(E^\#; W\otimes\scrO_{E^\#}(k)) = \bigoplus_{k\geq 0} H^1(X^\#; \Omega^{n-1}_{X^\#}(\log E^\#))(k),$$
 for the grading on $H^1(X^\#; \Omega^{n-1}_{X^\#}(\log E^\#))$ coming from the $\Cee^*$-action. Similarly, using instead  the affine morphism $\rho\colon U \to E^\#$, we have
\begin{align*}H^1(U; \Omega^{n-1}_{X^\#} |U) &= H^1(U; \Omega^{n-1}_{X^\#}(\log E^\#)|U) = H^1(U; \rho_*\rho^*W) = \bigoplus_{k\in \Zee} H^1(E^\#; W\otimes\scrO_{E^\#}(k)) \\
&= \bigoplus_{k\in \Zee} H^1(U; \Omega^{n-1}_{X^\#} |U)(k).
\end{align*}
Thus, the image of $H^1(X^\#; \Omega^{n-1}_{X^\#}(\log E^\#))$ in $H^1(U; \Omega^{n-1}_{X^\#} |U) = H^1(U; \Omega^{n-1}_{X^\#}(\log E^\#)|U)$ is the nonnegative weight space in $H^1(U; \Omega^{n-1}_{X^\#}(\log E^\#)|U)$.    As in Remark~\ref{Cstaridentify}, under the image of the isomorphisms
$$H^1(U; \Omega^{n-1}_U)  \to H^1(U; T_U)\cdot\omega \to H^1(U; T_U),$$ the image in $H^1(U; T_U)$ of the nonnegative weight space $\bigoplus_{k\geq 0}H^1(U; \Omega^{n-1}_U)$ is the space 
$$\bigoplus_{k\geq -N}H^1(U; T_U) \cong \bigoplus_{k\geq -N}H^0(X;T^1_X)(k),$$
as claimed. Finally, we would like to identify  $H^1(\hX; \Omega^{n-1}_{\hX}(\log E))$ with $H^1(X^\#; \Omega^{n-1}_{X^\#}(\log E^\#))$. In case  $E^\# =E$ is a smooth hypersurface in $\Pee^n$, this argument then  gives a proof of the following special case of \cite[Corollary 3.8]{Wahl}:

\begin{proposition} Suppose that $X$ is the cone over a smooth hypersurface $E$ in $\Pee^n$ of degree $d$. Then the image of  $H^1(\hX; T_{\hX}(-\log E))$ in $H^0(X; T^1_X)$ is  $\bigoplus_{k\geq 0}H^0(X;T^1_X)(k)$.  
\end{proposition}
\begin{proof} In this case, we can take $\hX =X^\#$, the ordinary blowup of $X$ at $0$, with exceptional divisor $E= E^\#$.   Then $T_{\hX}(-\log E) =\Omega^{n-1}_{\hX}(\log E)\otimes K_{\hX}^{-1} \otimes \scrO_{\hX}(-E) = \Omega^{n-1}_{\hX}(\log E)\otimes \rho^*K_E^{-1}$. Since all of the weights $a_i$ are $1$, $N =  n+1 -d$ and     $K_E = \scrO_E(n+1-d) =  \scrO_E(-N)$. Thus, after shifting by $N$,  the image of $H^1(\hX; T_{\hX}(-\log E))$ is the weight $\geq 0$ subspace.
\end{proof}

 Because of the orbifold singularities, it is not possible to carry out this argument directly for a general weighted homogeneous singularity. Instead, we work with stacks. Then the outline of the proof is as follows: 
 
 \smallskip
 \noindent \textbf{Step I:} Replacing $X^\#$ and $E^\#$ by the corresponding quotient stacks $\underline{X}^\#$ and $\underline{E}^\#$, we show that there is a corresponding vector bundle $\Omega^{n-1}_{\underline{X}^\#}(\log \underline{E}^\#)$ on $\underline{X}^\#$ with   $\Omega^{n-1}_{\underline{X}^\#}(\log \underline{E}^\#) = \rho^*\underline{W}$, where $\rho\colon \underline{X}^\#\to \underline{E}^\#$ is the corresponding morphism and $\underline{W}$ is a vector bundle on the stack $\underline{E}^\#$.
 
 \smallskip
 \noindent \textbf{Step II:} We show that there is an open substack of $\underline{X}^\#$ isomorphic to (the representable stack) $U$, identifying  $\Omega^{n-1}_{\underline{X}^\#}(\log \underline{E}^\#)|U$ with $\Omega^{n-1}_U$, and that, as in Wahl's argument but using the cohomology of the corresponding sheaves on the appropriate stacks, 
 $$H^1(\underline{X}^\#; \Omega^{n-1}_{\underline{X}^\#}(\log \underline{E}^\#))= H^1(\underline{E}^\#;\rho_*\rho^*\underline{W})= \bigoplus_{k\geq 0}H^1(\underline{E}^\# ; \underline{W} \otimes \scrO_{\underline{E}^\#}(k))$$
  is the nonnegative weight part of 
  $$H^1(U; \Omega^{n-1}_U) = \bigoplus_{k\in \Zee}H^1(\underline{E}^\#; \underline{W} \otimes\scrO_{\underline{E}^\#}(k)).$$
 
 \smallskip
 \noindent \textbf{Step III:} We identify $H^1(\underline{X}^\#; \Omega^{n-1}_{\underline{X}^\#}(\log \underline{E}^\#))$ with $H^1(X^\#; \Omega^{n-1}_{X^\#}(\log E^\#))$, where $\Omega^{n-1}_{X^\#}(\log E^\#)$ is the sheaf on the scheme $X^\#$ defined by Steenbrink in \cite[Definition 1.17]{Steenvanishing}. 
  
 \smallskip
 \noindent \textbf{Step IV:} We identify $H^1(X^\#; \Omega^{n-1}_{X^\#}(\log E^\#))$ with $H^1(\hX; \Omega^{n-1}_{\hX}(\log E))$.
 
 \smallskip

\noindent \textbf{Steps I and II:} For clarity, first consider the case of the weighted projective space $WP^n$: it  is the coarse moduli space of the quotient stack $[\Cee^{n+1}-0/\Cee^*]$ for the $\Cee^*$-action $\lambda\cdot (z_1, \dots,   z_{n+1}) = (\lambda^{a_1}z_1, \dots , \lambda^{a_{n+1}}z_{n+1})$. Note that the $\Cee^*$-action extends to a morphism $F\colon \Cee\times (\Cee^{n+1}-0) \to \Cee^{n+1}$ defined by $(t, z)\mapsto (t^{a_1}z_1, \dots , t^{a_{n+1}}z_{n+1})$.  Working instead with the scheme $\Cee\times (\Cee^{n+1}-0)$ and the $\Cee^*$-action
 $$\lambda \cdot (t,  z_1, \dots,   z_{n+1})= (\lambda^{-1}t, \lambda^{a_1}z_1, \dots , \lambda^{a_{n+1}}z_{n+1})$$ defines another quotient stack $[\Cee\times (\Cee^{n+1}-0)/\Cee^*]$.  Here, $\Cee\times (\Cee^{n+1}-0)$ is the normalization of the closure in  $\Cee^{n+1} \times (\Cee^{n+1}-0)$ of the ``incidence correspondence" $\mathcal{I} \subseteq (\Cee^{n+1}-0)\times (\Cee^{n+1}-0)$ defined by 
 $$\mathcal{I} = \{(z, z')\in (\Cee^{n+1}-0)\times (\Cee^{n+1}-0): \text{ there exists $\lambda \in \Cee^*$ such that $z = \lambda\cdot z'$} \}.$$  
 The geometric quotient $\Cee\times (\Cee^{n+1}-0)/\Cee^*$ is the weighted blowup $\widetilde{\Cee^{n+1}}$ of $\Cee^{n+1}$ at the origin: the morphism $F\colon \Cee\times (\Cee^{n+1}-0) \to \Cee^{n+1}$ is $\Cee^*$-equivariant, for the given action on $\Cee\times (\Cee^{n+1}-0)$ and the trivial action on $\Cee^{n+1}$, and, for $t\neq 0$, $F(t, z) = F(t', z')$ $\iff$ $(t, z)$ and $(t', z')$ are in the same $\Cee^*$-orbit.   Thus, there is a morphism $\rho\colon \widetilde{\Cee^{n+1}} \to WP^n$, arising from the $\Cee^*$-equivariant morphism $\pi_2\colon \Cee\times (\Cee^{n+1}-0) \to \Cee^{n+1} -0$,  as well as an induced morphism $\widetilde{\Cee^{n+1}} \to \Cee^{n+1}$ which is an isomorphism except over $0$, and whose fiber over $0$ is $WP^n$.   In terms of stacks, identifying a representable stack with the corresponding scheme, we have a morphism  
 $$\rho \colon [\Cee\times (\Cee^{n+1}-0)/\Cee^*] \to [(\Cee^{n+1}-0)/\Cee^*],$$
  as well as a morphism $[\Cee\times (\Cee^{n+1}-0)/\Cee^*] \to   \Cee^{n+1}$, inducing an isomorphism from the open substack $O= [\Cee^*\times (\Cee^{n+1}-0)/\Cee^*]$  to $\Cee^{n+1}-0$.  Likewise, the $\Cee^*$-equivariant inclusion $0\times (\Cee^{n+1}-0) \to \Cee\times (\Cee^{n+1}-0)$ defines a closed embedding 
  $$[(\Cee^{n+1}-0)/\Cee^*] \to [\Cee\times (\Cee^{n+1}-0)/\Cee^*]$$ which is a section of the morphism $[\Cee\times (\Cee^{n+1}-0)/\Cee^*] \to [(\Cee^{n+1}-0)/\Cee^*]$.
 
 The stacks $[\Cee^{n+1}-0/\Cee^*]$ and $[\Cee\times (\Cee^{n+1}-0)/\Cee^*]$ can also be covered by orbifold charts. As usual, we have the morphism $\Cee^{n+1}-0\to  [\Cee^{n+1}-0/\Cee^*]$ defined by the trivial bundle $(\Cee^{n+1}-0)\times \Cee^*$ and the $\Cee^*$-equivariant morphism $(\Cee^{n+1}-0)\times \Cee^* \to \Cee^{n+1}-0$ defined by $(z,s) \mapsto s\cdot z$. Let $A_i = \Spec \Cee[z_1, \dots, \hat{z}_i, \dots, z_{n+1}]\cong \Cee^n$, which we view   as the subset of $\Cee^{n+1}-0$ where  $z_i =1$. Note that, while $A_i$ is not a $\Cee^*$-invariant subset of $\Cee^{n+1}-0$, the group of $a_i^{\text{th}}$ roots of unity $\boldsymbol{\mu}_{a_i}$ acts on $A_i$ via 
 $$\zeta \cdot (z_1, \dots, \hat{z}_i, \dots, z_{n+1}) = (\zeta^{-a_1}z_1, \dots, \hat{z}_i, \dots, \zeta^{-a_{n+1}}z_{n+1}).$$
Moreover, abbreviating $(z_1, \dots, \hat{z}_i, \dots, z_{n+1})$ by $z$, the \'etale morphism $F_i\colon A_i\times \Cee^* \to \Cee^{n+1}-0$ defined by
 $$F_i(z,s)=( s^{a_1}z_1, \dots , s^{a_i}, \dots,  s^{a_{n+1}}z_{n+1})$$ 
 satisfies $F_i(\zeta^{-1}\cdot z,s) = F_i(z, \zeta s)$ and thus  induces a $\Cee^*$-equivariant morphism $A_i\times^{\boldsymbol{\mu}_{a_i}}\Cee^* \to \Cee^{n+1}-0$, where, $A_i\times^{\boldsymbol{\mu}_{a_i}}\Cee^*$ denotes the quotient of $A_i \times \Cee^*$ by the equivalence relation
 $$(\zeta^{-1}z, s) \sim (z , \zeta s).$$
  Thus there is a morphism of quotient stacks $[A_i/\boldsymbol{\mu}_{a_i}] \to [(\Cee^{n+1}-0)/\Cee^*]$, which is an open embedding. Since $F_i(z, \lambda s) = \lambda\cdot F_i(z,s)$, the composite morphism $A_i\to [A_i/\boldsymbol{\mu}_{a_i}] \to [(\Cee^{n+1}-0)/\Cee^*]$ corresponds to the free $\Cee^*$-action on the trivial $\Cee^*$-bundle $A_i\times \Cee^*$ given by  
 $$\lambda \cdot (z_1, \dots, \hat{z}_i, \dots, z_{n+1}, s) = (z_1, \dots, \hat{z}_i, \dots, z_{n+1}, \lambda s),$$
 and $A_i$ is identified with the geometric quotient $(A_i\times \Cee^*)/\Cee^*$. 
 In particular, the image of $A_i/\boldsymbol{\mu}_{a_i}$ in $WP^n$ is the open subset corresponding to $z_i\neq 0$. A similar construction works for blowups: let $B_i = \Cee \times A_i$, with $\boldsymbol{\mu}_{a_i}$-action
 $$\zeta \cdot (t, z_1, \dots, \hat{z}_i, \dots, z_{n+1}) = (\zeta t, \zeta^{-a_1}z_1, \dots, \hat{z}_i, \dots, \zeta^{-a_{n+1}}z_{n+1}),$$
 and define $G_i\colon B_i\times \Cee^* \to \Cee\times (\Cee^{n+1}-0)$ by
 $$G_i(t,z,s) = (s^{-1}t,  s^{a_1}z_1, \dots , s^{a_i}, \dots,  s^{a_{n+1}}z_{n+1}).$$ 
 Note that $G_i(\zeta^{-1}\cdot (t,z), s) = G_i(t, z, \zeta s)$ and $G_i(t, z , \lambda s) = \lambda \cdot G_i(t,z,s)$. 
 Then as before $G_i$ is \'etale and induces a $\Cee^*$-equivariant morphism $B_i\times^{\boldsymbol{\mu}_{a_i}}\Cee^* \to \Cee \times (\Cee^{n+1}-0)$ and hence a morphism of quotient stacks $[B_i/\boldsymbol{\mu}_{a_i}] \to [\Cee \times (\Cee^{n+1}-0)/\Cee^*]$.  
 As before, $[B_i/\boldsymbol{\mu}_{a_i}] \to [\Cee\times(\Cee^{n+1}-0)/\Cee^*]$ is an open embedding, the composite morphism $B_i \to [B_i/\boldsymbol{\mu}_{a_i}] \to [\Cee\times(\Cee^{n+1}-0)/\Cee^*]$ corresponds to the trivial $\Cee^*$-bundle   $B_i\times \Cee^*$,   and $B_i$ is identified with the geometric quotient $(B_i\times \Cee^*)/\Cee^*$. Moreover, the $\boldsymbol{\mu}_{a_i}$-action on $B_i$ is compatible both with the projection $B_i \to A_i$ and the inclusion $0\times A_i \to B_i$. 
More generally, the two constructions are compatible in the sense that there is a commutative diagram
 $$\begin{CD}
 B_i @>>> [B_i/\boldsymbol{\mu}_{a_i}] @>>> [\Cee\times (\Cee^{n+1}-0)/\Cee^*]\\
 @VVV @VVV @VVV\\
 A_i @>>> [A_i/\boldsymbol{\mu}_{a_i}] @>>> [ \Cee^{n+1}-0/\Cee^*]
 \end{CD}$$
 This discussion restricts to the affine subvariety $X\subseteq \Cee^{n+1}$, with $U = X-0 \subseteq \Cee^{n+1}-0$. By assumption, $U$ is smooth.  We have the quotient stacks $[U/\Cee^*] = \underline{E}^\#$ and $[\Cee\times U/\Cee^*] = \underline{X}^\#$. As before, we have morphisms $\underline{X}^\# \to \underline{E}^\#$ and $\underline{X}^\# \to X$. The open substack $U= [\Cee^*\times U/\Cee^*]\subseteq \underline{X}^\#$ is isomorphic to $U$ via $\pi_2$. Likewise, the $\Cee^*$-equivariant inclusion $0\times U \to \Cee\times U$ defines a closed embedding $\underline{E}^\# \to \underline{X}^\#$ which is a section of the morphism $\underline{X}^\# \to \underline{E}^\#$. Finally, we also have orbifold charts:  Let 
$$U_i = \{(z_1, \dots, \hat{z}_i, \dots, z_{n+1})\in A_i: f(z_1, \dots, 1, \dots, z_{n+1}) = 0\}=X\cap A_i,$$
where we identify $A_i$ with the corresponding subset of $\Cee^{n+1}-0$, and let $V_i = \Cee\times U_i$. There are \'etale morphisms $U_i\times \Cee^* \to U$ and $V_i \times \Cee^* \to \Cee\times U$, as well as compatible morphisms
$$\begin{CD}
 V_i @>>> [V_i/\boldsymbol{\mu}_{a_i}] @>>> \underline{X}^\#\\
 @VVV @VVV @VVV\\
 U_i @>>> [U_i/\boldsymbol{\mu}_{a_i}] @>>> \underline{E}^\#
 \end{CD}$$ 
 On the level of coarse moduli spaces, $X^\#$ is covered by open subsets $V_i/\boldsymbol{\mu}_{a_i}$ and $E^\#$ is covered by open subsets $U_i/\boldsymbol{\mu}_{a_i}$.  Since $U_i$ and $V_i$ are smooth and $U_i$ is the smooth divisor $0\times U_i \subseteq \Cee\times U_i$, $X^\#$ is an orbifold and $E^\#$ is an orbifold smooth divisor.
 
 Turning now to sheaves, recall that a coherent sheaf $\mathcal{F}$ on $\Cee^{n+1}-0$ together with an action of $\Cee^*$ on $\mathcal{F}$ lifting the action of $\Cee^*$ on $\Cee^{n+1}-0$ defines a sheaf $\underline{\mathcal{F}}$ on 
 $[ \Cee^{n+1}-0/\Cee^*]$, and similarly for any of the other spaces on which $\Cee^*$ acts. For example, the sheaf $\scrO_{\Cee^{n+1}-0}$ together with the action given by the character $\chi_r$ of $\Cee^*$, where $\chi_r(\lambda) =\lambda^r$, defines the analogue of the sheaf $\scrO_{WP^n}(r)$. Recall that  $U =[\Cee^*\times U/\Cee^*]\subseteq \underline{X}^\# =[\Cee \times U/\Cee^*]$ and that there are  morphisms  (both of which we shall denote by $\rho$) $U \to \underline{E}^\#$ and $\underline{X}^\# \to \underline{E}^\#$. Then  $\Cee^*\times U =\Spec  A(U)[x, x^{-1}]$ and  $\Cee\times U =\Spec  A(U)[x]$, where $\Cee^*$ acts on $x$, the dual coordinate to $t\in \Cee$, by $\lambda \cdot x = \lambda x$. Hence $\rho_*\scrO_U = \bigoplus_{k\in \Zee} \scrO_{\underline{E}^\#}(k)$ and 
 $\rho_*\scrO_{\underline{X}^\#} = \bigoplus_{k\geq 0} \scrO_{\underline{E}^\#}(k)$.
 
 We now define the stacky equivalent of the sheaf $\Omega^p_{X^\#}(\log E^\#)$. It  suffices to consider the case $p=1$. 
 As before, we begin with the case of the orbifold smooth divisor $WP^n \subseteq \widetilde{\Cee^{n+1}}$. On $\Cee\times (\Cee^{n+1}-0)$, with coordinates $t, z_1, \dots, z_{n+1}$, we have the Euler vector field 
  $\Dis \tilde\xi = -t\frac{\partial}{\partial t} + \sum_{i=1}^{n+1}a_iz_i\frac{\partial}{\partial z_i}$, 
 as well as the analogous vector field $\Dis \xi =  \sum_{i=1}^{n+1}a_iz_i\frac{\partial}{\partial z_i}$  on  $\Cee^{n+1}-0$.
   Let $\Omega  = \Omega^1_{\Cee\times (\Cee^{n+1}-0)} \subseteq   \Omega(\log) = \Omega^1_{\Cee\times (\Cee^{n+1}-0)}(\log (0\times (\Cee^{n+1}-0)))$, which is free with basis $\Dis \frac{dt}{t}, dz_1, \dots, dz_{n+1}$. Let $\Omega_0\subseteq \Omega$ be the annihilator of $\tilde\xi$: $\Omega_0 = \{\varphi\in \Omega: \varphi(\tilde\xi) = 0\}$, and define $\Omega(\log)_0$ similarly. Thus, in terms of local sections,

\begin{align*}
hdt + \sum_ig_idz_i\in \Omega_0  &\iff th = \sum_ia_iz_ig_i; \\
h\frac{dt}{t} + \sum_ig_idz_i\in \Omega(\log)_0 &\iff h = \sum_ia_iz_ig_i.
\end{align*}

It is easy to check that $\Omega_0$ is locally free: let $D_i$ be the  open subset of $\Cee^{n+1}$ where $z_i\neq 0$. Then a basis of sections of $\Omega_0$ over $\Cee\times D_i$ is given by
$$dt + t(a_iz_i)^{-1}dz_i \quad \text{ and} \quad  dz_j - \frac{a_jz_j}{a_iz_i}dz_i, j\neq i.$$
   Similarly,  $\Omega(\log)_0$ is free with basis $\Dis a_iz_i\frac{dt}{t}+ dz_i$, $1\leq i\leq n+1$. Over the open set $\Cee^*\times (\Cee^{n+1}-0)$,  $\Omega(\log)_0|\Cee^*\times (\Cee^{n+1}-0)$ is the pullback of $\Omega^1_{\Cee^{n+1}-0}$  via the morphism $(t, z)\mapsto (t^{a_1}z_1, \dots , t^{a_{n+1}}z_{n+1})$, with basis $\Dis t^{a_i}\left(a_iz_i\frac{dt}{t}+ dz_i\right)$, $1\leq i\leq n+1$ and $\Cee^*$ acts trivially on this basis. Another useful basis for $\Omega(\log)_0$ over $\Cee\times D_i$ is given by
\begin{equation}\label{basiseqn}  dz_j - \frac{a_jz_j}{a_iz_i}dz_i, j\neq i   \quad \text{ and} \quad \frac{dt}{t} + (a_iz_i)^{-1}dz_i .
\end{equation}
   Here $\Cee^*$ acts on $\Dis dz_j - \frac{a_jz_j}{a_iz_i}dz_i$ with weight $a_j$ and on   $\Dis\frac{dt}{t} + (a_iz_i)^{-1}dz_i$ with weight $0$. 
   
 \begin{lemma} The vector bundle $\Omega(\log)_0$ is isomorphic to the pullback via $\pi_2^*$ of the vector bundle $\Omega(\log)_0|0\times (\Cee^{n+1}-0)$, viewed as a vector bundle on $\Cee^{n+1}-0$, and the isomorphism is compatible with the $\Cee^*$-actions on both bundles. 
 \end{lemma}
 \begin{proof}  By Poincar\'e residue, after   identifying $\Cee^{n+1}-0$ with the divisor $0\times (\Cee^{n+1}-0)$,  there is an exact sequence
   $$0 \to \Omega  \to \Omega(\log) \to \scrO_{\Cee^{n+1}-0} \to 0.$$
  Restricting the above  sequence to $0\times (\Cee^{n+1}-0)$ gives an exact sequence
$$0 \to J/J^2 \to \Omega|0\times (\Cee^{n+1}-0)  \to \Omega(\log)|0\times (\Cee^{n+1}-0)  \to \scrO_{\Cee^{n+1}-0} \to 0,$$
where $J/J^2$ is the conormal bundle to the smooth divisor $0\times (\Cee^{n+1}-0) \subseteq \Cee \times (\Cee^{n+1}-0) $. 

The induced map $\Omega(\log)_0 \to \scrO_{\Cee^{n+1}-0}$ is  surjective: over the open subset $\Cee\times D_i$ where as before $D_i\subseteq \Cee^{n+1}$ is the open set defined by $z_i\neq 0$, 
the Poincar\'e residue of $\Dis \frac{dt}{t} + (a_iz_i)^{-1}dz_i$ is $1$. The bundle $\Omega_0|0\times (\Cee^{n+1}-0)$ is locally free with  basis $ dt$ and $\Dis  dz_j - \frac{a_jz_j}{a_iz_i}dz_i, j\neq i$ and the image of $J/J^2$ is spanned by $dt$. Thus 
$$\Omega_0|0\times (\Cee^{n+1}-0)\Big/\im J/J^2 = (\Omega^1_{\Cee^{n+1}-0 })_0,$$
the annihilator of $\xi$ in $\Omega^1_{\Cee^{n+1}-0 }$.  In other words, still under the identification $0\times (\Cee^{n+1}-0)\cong \Cee^{n+1}-0$, there is an exact sequence
$$0 \to  (\Omega^1_{\Cee^{n+1}-0 })_0 \to \Omega(\log)_0|0\times (\Cee^{n+1}-0) \to \scrO_{\Cee^{n+1}-0} \to 0.$$
It follows that,  over   $\Cee\times D_i$, $\Omega(\log)_0|0\times (\Cee^{n+1}-0)$ has basis $\Dis dz_j - \frac{a_jz_j}{a_iz_i}dz_i, j\neq i$ and $\Dis\frac{dt}{t} + (a_iz_i)^{-1}dz_i$ and $\Cee^*$ acts on $\Dis dz_j - \frac{a_jz_j}{a_iz_i}dz_i$ with weight $a_j$ and on  $\Dis\frac{dt}{t} + (a_iz_i)^{-1}dz_i$ with weight $0$. Pulling back via $\pi_2$, this basis  lifts to the basis $\Dis dz_j - \frac{a_jz_j}{a_iz_i}dz_i, j\neq i$ and $\Dis\frac{dt}{t} + (a_iz_i)^{-1}dz_i$ of   (\ref{basiseqn}) and thus defines an isomorphism over $\pi_2^{-1}(D_i)$ from $\pi_2^*\Omega(\log)_0|0\times (\Cee^{n+1}-0)$ to  $\Omega(\log)_0$  which is compatible with the $\Cee^*$-actions. Over $D_i\cap D_j$ resp.\ $\pi_2^{-1}(D_i\cap D_j)$,  the transition functions for the two bases are clearly compatible since the bases are defined by the same formulas. This defines the $\Cee^*$-equivariant isomorphism of the lemma. 
\end{proof}  

A very similar argument handles the case of the divisor $0\times U \subseteq \Cee\times U$. Note that, as $f$ is weighted homogeneous, there are Euler vector fields $\xi$ and   $\tilde\xi$ on $U$ and $\Cee\times U$, respectively:  First, we have the normal bundle sequence
$$0 \to T_U \to T_{\Cee^{n+1}-0}|U \to N_{U/\Cee^{n+1}-0} \to 0,$$
where $N_{U/\Cee^{n+1}-0}$ is the normal bundle to $U$, the dual to the conormal bundle $I_U/I_U^2\cong \scrO_U\cdot [f]$ and  $[f]$ is the class   $f\bmod f^2$. A derivation $\theta$ defines a homomorphism $I_U/I_U^2$ to $\scrO_U$ by: $\theta (h[f]) = h\theta (f)$. Since $\xi(f) = d\cdot f$,  the section $\xi|U$ maps to $0$ in $N_{U/\Cee^{n+1}-0}$ and hence defines a section of $T_U$, also denoted $\xi$. Likewise, since $f$ does not depend on $t$, we can define the vector field $\tilde\xi$ as a global section of $T_{\Cee\times U}$ and hence the annihilator of $\tilde\xi$ in $\Omega^1_{\Cee\times U}(\log (0\times U))$; denote it by $\Omega^1_{\Cee\times U}(\log (0\times U))_0$.  From the conormal sequence, there is an exact sequence
$$0 \to \pi_2^*(I_U/I_U^2) \to \Omega(\log)_0|U \to \Omega^1_{\Cee\times U}(\log (0\times U))_0 \to 0.$$
Since $\Omega(\log)_0|U$ and the map $\pi_2^*(I_U/I_U^2) \to \Omega(\log)_0|U$ are both pulled back from the restriction  to  $0\times U$, it follows that $\Omega^1_{\Cee\times U}(\log (0\times U))_0$ is the pullback of the corresponding vector bundle on $U$, equivariantly with respect to the $\Cee^*$-actions. 

The upshot is that $\Omega^1_{\Cee\times U}(\log (0\times U))_0$ defines a vector bundle  $\Omega^1_{\underline{X}^\#}(\log \underline{E}^\#)$ on the stack $\underline{X}^\#$, and it is the pullback of a vector bundle $\underline{W}_0$ on $\underline{E}^\#$ via the morphism   $\rho \colon  \underline{X}^\# \to \underline{E}^\#$. By taking the $(n-1)^{\text{st}}$ exterior power, we similarly have $\Omega^{n-1}_{\underline{X}^\#}(\log \underline{E}^\#)= \rho^*\left(\bigwedge ^{n-1}\underline{W}_0\right)=\rho^*\underline{W}$, where $\underline{W} = \bigwedge ^{n-1}\underline{W}_0$.  We can then take the   cohomology of $\Omega^{n-1}_{\underline{X}^\#}(\log \underline{E}^\#)$ and have: 
$$H^1(\underline{X}^\#; \Omega^{n-1}_{\underline{X}^\#}(\log \underline{E}^\#))= H^1(\underline{X}^\#; \rho^*\underline{W})= \bigoplus_{k\geq 0} H^1(\underline{E}^\#; \underline{W}\otimes \scrO_{\underline{E}^\#}(k)),$$
with the grading corresponding to the $\Cee^*$-action on $H^1(\underline{X}^\#; \Omega^{n-1}_{\underline{X}^\#}(\log \underline{E}^\#))$. Here, the subspace $H^1(\underline{X}^\#; \Omega^{n-1}_{\underline{X}^\#}(\log \underline{E}^\#))(k)$  where $\Cee^*$ acts with weight $k$, is equal to $H^1(\underline{E}^\#; \underline{W}\otimes \scrO_{\underline{E}^\#}(k))$. Similarly, restricting to the open substack $[\Cee^*\times U/\Cee^*] = U$ and the morphism $\rho\colon U \to \underline{E}^\#$, we have
$$H^1(U; \Omega^{n-1}_{\underline{X}^\#}(\log \underline{E}^\#)|U)= H^1(U; \rho^*\underline{W})= \bigoplus_{k\in \Zee} H^1(\underline{E}^\#; \underline{W}\otimes \scrO_{\underline{E}^\#}(k)),$$
again  with the grading corresponding to the $\Cee^*$-action. 
On the other hand, via the morphism $(t, z)\mapsto (t^{a_1}z_1, \dots , t^{a_{n+1}}z_{n+1})$,   $\Omega(\log)_0|\Cee^*\times (\Cee^{n+1}-0)$ is the pullback of $\Omega^1_{\Cee^{n+1}-0}$.   Similarly the restriction of $\Omega^1_{\underline{X}^\#}(\log \underline{E}^\#)$ to the open substack $U$ is $\Omega^1_U$.  Thus also $H^1(U; \Omega^{n-1}_{\underline{X}^\#}(\log \underline{E}^\#)|U)= H^1(U;\Omega^{n-1}_U)$. This completes the proof of Steps I and II.

\smallskip

\noindent \textbf{Step III:} The next step is to compare the cohomology group  $H^1(\underline{X}^\#; \Omega^{n-1}_{\underline{X}^\#}(\log \underline{E}^\#))$ with  the group $H^1(X^\#; \Omega^{n-1}_{X^\#}(\log E^\#))$, where $\Omega^{n-1}_{X^\#}(\log E^\#)$ is the sheaf defined by Steenbrink \cite{Steenvanishing} on the coarse moduli space $X^\#$:

\begin{lemma} For all $p,q$, $H^q(\underline{X}^\#; \Omega^p_{\underline{X}^\#}(\log \underline{E}^\#))=H^q(X^\#; \Omega^p_{X^\#}(\log E^\#))$.
\end{lemma}
\begin{proof} We have an open cover of $\underline{X}^\#$ by the stacks $[V_i/\boldsymbol{\mu}_{a_i}]$. We use these to compute \v{C}ech cohomology. Our goal is to show that the group of sections of $\Omega^p_{\underline{X}^\#}(\log \underline{E}^\#)$ over $[V_i/\boldsymbol{\mu}_{a_i}]$ is $(\Omega^p_{V_i}(\log U_i))^{\boldsymbol{\mu}_{a_i}}$, which by definition is $\Omega^p_{X^\#}(\log E^\#))(V_i/\boldsymbol{\mu}_{a_i})$. It suffices to show that the pullback of $\Omega^p_{\underline{X}^\#}(\log \underline{E}^\#)$ to $V_i$ is $\Omega^p_{V_i}(\log U_i)$; here we use the fact that we are in characteristic zero. We shall just work out the argument for the corresponding case of $A_i$ and $B_i =\Cee\times A_i$ and for $p=1$. Recall that we have an \'etale, $\Cee^*$-equivariant  morphism $A_i\times \Cee^* \to \Cee^{n+1}-0$ and hence a corresponding morphism $G_i \colon B_i\times \Cee^* = \Cee \times A_i\times \Cee^* \to \Cee \times (\Cee^{n+1}-0)$, defined by
$$G_i(t, z_1, \dots, \hat{z}_i, \dots, z_{n+1}, s) = (s^{-1}t,  s^{a_1}z_1, \dots , s^{a_i}, \dots, s^{a_{n+1}}z_{n+1}).$$ 
Since $G_i$ is \'etale,  $G_i^*\Omega(\log) = \Omega^1_{B_i\times \Cee^*}(\log (A_i\times \Cee^*))$, viewing $A_i\times \Cee^*$ as the divisor $0\times A_i\times \Cee^*\subseteq \Cee \times A_i\times \Cee^*$. 
Via $G_i^*$, the basis $\Dis \frac{dt}{t} + (a_iz_i)^{-1}dz_i$, $\Dis    dz_j - \frac{a_jz_j}{a_iz_i}dz_i, j\neq i$ of $\Omega(\log)_0$ pulls back to $\Dis \frac{dt}{t} $, $\Dis    s^{a_j}dz_j , j\neq i$, and hence $G_i^*\Omega(\log)_0$ is the pullback of $\Omega^1_{B_i}(\log A_i)$ (or equivalently, is the annihilator of the Euler vector field $\Dis \tilde\xi_i =  s\frac{\partial}{\partial s}$). Then the sheaf of $\Cee^*$-invariant sections of $G_i^*\Omega(\log)_0$ over $B_i\times \Cee^*$ is  $\Omega^1_{B_i}(\log A_i)$, and similarly for the exterior powers. The analogous statement for   $U_i$ and $V_i$ then implies that  the pullback of $\Omega^p_{\underline{X}^\#}(\log \underline{E}^\#)$ to $V_i$ is $\Omega^p_{V_i}(\log U_i)$.
\end{proof}

 \noindent \textbf{Step IV:} Lastly, we compare $H^1(X^\#; \Omega^{n-1}_{X^\#}(\log E^\#))$ with $H^1(\hX; \Omega^{n-1}_{\hX}(\log E))$:

\begin{lemma}\label{reswp} If $\hX$ is a good equivariant resolution of $X$ dominating $X^\#$, then, for all $p,q$,
$$H^q(\hX; \Omega^p_{\hX}(\log E)) \cong H^q(X^\#; \Omega^p_{X^\#}(\log E^\#)).$$
\end{lemma}
\begin{proof} By construction, $X^\#$ and hence $\hX$ dominate the affine variety $X$ and are isomorphic to it away from $0$. Let $\overline{X}$ be some projective completion of $X$ such that $\overline{X}-\{0\}$ is smooth and the complement $F$ of $X$ in $\overline{X}$ is a divisor with normal crossings. Then $\overline{X}$ defines projective completions $\overline{X^\#}$ and $\overline{\hX}$ of $X^\#$ and $\hX$ respectively, and we can identify $\overline{X^\#} -X^\#$ and $\overline{\hX} - \hX$ with $F$ as well. Let $j^\#\colon X^\# \to \overline{X^\#}$ and $j\colon \hX \to \overline{\hX}$ be the inclusions. Then $j^\#$ and $j$ are affine morphisms. Hence $R^ij_*\Omega^p_{\hX}(\log E) =0$ for $i >0$. By the Leray spectral sequence, $H^q(\hX; \Omega^p_{\hX}(\log E)) = H^q(\overline{\hX}; j_*\Omega^p_{\hX}(\log E))$, and likewise $H^q(X^\#; \Omega^p_{X^\#}(\log E^\#)) = H^q(\overline{X^\#}; j^\#_*\Omega^p_{X^\#}(\log E^\#) )$. Finally let
$$Q = j_*\Omega^p_{\hX}(\log E)/\Omega^p_{\overline{\hX}}(\log (E+F)) = j^\#_*\Omega^p_{X^\#}(\log E^\#)/\Omega^p_{\overline{X^\#}}(\log (E^\#+F)).$$
Then the birational morphism $\hX \to X^\#$ induces a commutative diagram
$$\begin{CD}
@>>> H^q(\overline{X^\#}; \Omega^p_{\overline{X^\#}}(\log (E^\#+F))) @>>> H^q(\overline{X^\#};j^\#_*\Omega^p_{X^\#}\log (E^\#)) @>>> H^q(\overline{X^\#};Q) @>>> \\
@. @VVV @VVV @| @.
\\
@>>> H^q(\overline{\hX}; \Omega^p_{\overline{\hX}}(\log (E+F))) @>>> H^q(\overline{\hX};j_*\Omega^p_{\hX}\log (E^\#)) @>>> H^q(\overline{\hX};Q) @>>>
\end{CD}$$
By \cite[1.12 and 1.19]{Steenvanishing} (see especially the comments after the proof of 1.19), the homomorphism
$$H^q(\overline{X^\#}; \Omega^p_{\overline{X^\#}}(\log (E^\#+F))) \to H^q(\overline{\hX}; \Omega^p_{\overline{\hX}}(\log (E+F)))$$
is an isomorphism for all $p, q$.   Hence the homomorphism 
$$H^q(\overline{X^\#};j^\#_*\Omega^p_{X^\#}(\log E^\#)) \to H^q(\overline{\hX};j_*\Omega^p_{\hX}(\log E))$$
is also an isomorphism, so that $H^q(X^\#; \Omega^p_{X^\#}(\log E^\#)) \to H^q(\hX; \Omega^p_{\hX}(\log E))$ is an isomorphism as well.
\end{proof}
\renewcommand{\qedsymbol}{}
\end{proof}

\begin{remark}\label{biratinvremark} A similar argument shows that, for a general isolated singularity $X$ (not necessarily a hypersurface singularity or one with a $\Cee^*$-action), and for $\pi \colon \hX \to E$ a good resolution, the groups $H^q(\hX; \Omega^p_{\hX}(\log E))$ and $H^q(\hX; \Omega^p_{\hX}(\log E)(-E))$ are independent of the choice of $\hX$.
\end{remark}

\begin{proof}[Proof of Theorem~\ref{wtdlink}] As in Remark~\ref{rem4.6}, let $Q_f$ denote the Milnor algebra. Since $X$ is weighted homogeneous, $H^0(X;T^1_X) = Q_f$, and  in particular  $H^0(X;T^1_X)(-N) = Q_f(-N)$. In terms of the $V$-filtration on $Q_f$,  since weight $-N$ corresponds to $\ell(\alpha) = 2$ and thus $\ell(\alpha) -1 =1$, it follows from   Remark~\ref{rem4.6}(i) that $Q_f(-N) = \Gr_V^1Q_f$. By a result of Scherk-Steenbrink \cite{ScherkSteenbrink}, $\Gr_V^1Q_f \cong \Gr_F^{n-1}H^n(M)_1$, where $M$ is the Milnor fiber of $X$ and the subscript denotes the subspace of $H^n(M)$ where the semisimple part of the monodromy $T$ acts trivially.  Since  the monodromy is semisimple in the weighted homogeneous case, $T$ acts trivially on $H^n(M)_1$. Then  $H^n(M)_1 = W_{n+1}H^n(M)_1 \cong H^{n-1}(E^\#)_0(-1) \cong H^n(L)$ as mixed Hodge structures, where $H^{n-1}(E^\#)_0$ denotes the primitive cohomology (cf.\ Steenbrink \cite[bottom of p.\ 216]{SteenCompositio}) and $H^{n-1}(E^\#)_0(-1)$ denotes the Tate twist (not a $\Cee^*$-weight space). Thus $\Gr_F^{n-1}H^n(M)_1 \cong \Gr_F^{n-1}H^n(L)$. In particular, $\dim H^0(X;T^1_X)(-N) = \dim Q_f(-N) = \dim \Gr_F^{n-1}H^n(L)$.

Since $\Gr_F^{n-1}H^n(L)$ is a subquotient of $H^n(L)$, the $\Cee^*$-action on $\Gr_F^{n-1}H^n(L)$ is trivial, i.e.\ of weight $0$. Viewing $\Gr_F^{n-1}H^n(L)$ instead as a subquotient of $H^0(X;T^1_X)$ shifts the weight by $-N$. Thus
$\Gr_F^{n-1}H^n(L) \subseteq H^0(X;T^1_X)(-N)$. As both spaces have the same dimension, they are equal.  
\end{proof}

\section{Generalized Fano varieties}\label{Section4}
 
 For the moment, let $Y$ be an arbitrary compact complex analytic variety of dimension $n$ with isolated Gorenstein singularities,  and let $Z = Y_{\text{sing}}$. Let $\pi\colon \hY \to Y$ be a good equivariant resolution with exceptional set $E=\bigcup_{i=1}^rE_i$, and let $V = \hY -E$. We let $X$ be the union of good representatives for germs of $Y$ at the singular points and denote by $\pi \colon \hX \to X$ the induced resolution of $X$, with $U = X - Z = \hX - E$.  We can write
$$H^0(Y; T^1_Y) =\bigoplus_{x\in Z}T^1_{Y,x}$$ 
and likewise
$$K =\bigoplus_{x\in Z}K_x = \bigoplus_{x\in Z}T^1_{Y,x}/\im H^0(Y;(R^1\pi_*\Omega^{n-1}_{\hY})_x).$$

 There is  a  global version of Lemma~\ref{2.1}  (cf.\ \cite[Proposition 3.6]{F}):

\begin{lemma}\label{2.1global} Under the assumption of isolated lci singularities  of $\dim\ge 3$, $H^1_Z(Y; T^0_Y) =0$, $H^2_Z(Y; T^0_Y) \cong H^0(Y; T^1_Y)$, $\mathbb{T}^1_Y \cong H^1(V; T_V)$, and there is an isomorphism of exact sequences identifying the local cohomology  sequence for $T^0_Y$  and  the Ext spectral sequence relating $\Ext^{p+q}(\Omega^1_Y, \scrO_Y)$ to $H^p(Y; \mathit{Ext}^q (\Omega^1_Y, \scrO_Y))$ as follows:
$$\begin{CD}
0 @>>> H^1(Y; T^0_Y) @>>> \mathbb{T}^1_Y @>>> H^0(Y; T^1_Y) @>>> H^2(Y; T^0_Y) \\
@. @VV{=}V @V{\cong}VV @V{\cong}VV @VV{=}V \\
 0=H^1_Z(Y; T^0_Y) @>>> H^1(Y; T^0_Y) @>>> H^1(V; T^0_Y|V) @>>> H^2_Z(Y; T^0_Y) @>>> H^2(Y; T^0_Y).  
\end{CD}$$
\end{lemma}
    \begin{proof} Let $\mathcal{I}^\bullet$ be an injective resolution of   $\scrO_Y$ and let $\mathcal{J}^\bullet = \mathit{Hom}(\Omega^1_Y, \mathcal{I}^\bullet)$. Then  
$\mathbb{H}^i(Y;\mathcal{J}^\bullet) = \Ext^i(\Omega^1_Y, \scrO_Y) = \mathbb{T}^i_Y$ and  $\mathbb{H}^i(V;\mathcal{J}^\bullet|V)\cong H^i(V; T_V)$. The local cohomology exact sequence gives an exact sequence
$$\begin{CD}
\cdots @>>> \mathbb{H}^i_Z(Y;\mathcal{J}^\bullet) @>>> \mathbb{H}^i(Y;\mathcal{J}^\bullet) @>>> \mathbb{H}^i(V;\mathcal{J}^\bullet|V) @>>> \cdots\\
@. @| @| @| @. \\
\cdots @>>> \mathbb{H}^i_Z(Y;\mathcal{J}^\bullet) @>>> \mathbb{T}^i_Y  @>>> H^i(V; T_V) @>>> \cdots
\end{CD}$$
(cf.\  \cite[Expos\'e VI]{SGA2}, where $\mathbb{H}^i_Z(Y;\mathcal{J}^\bullet)$ is denoted $\Ext^i_Z(\Omega^1_Y, \scrO_Y)$). 
There is a spectral sequence with $E_2^{p,q} = H^p_Z(Y;\mathit{Ext}^q(\Omega^1_Y, \scrO_Y))$ converging to $\mathbb{H}^i_Z(Y;\mathcal{J}^\bullet)$, with $Ext^i(\Omega^1_Y, \scrO_Y) = T^i_Y$. Then $H^i_Z(T^0_Y)  = 0$ for $i=0,1$ by Lemma~\ref{1.0}, $H^i_Z(T^1_Y) = 0$ for $i>0$,  $T^2_Y =0$, and $d_2\colon H^0_Z(T^1_Y) \to H^2_Z(T^0_Y)$ is easily checked by excision to be the isomorphism of Lemma~\ref{2.1}. Hence $\mathbb{H}^1_Z(Y;\mathcal{J}^\bullet) = \mathbb{H}^2_Z(Y;\mathcal{J}^\bullet) =0$. Comparing the induced maps on the $E_2$ pages of the hypercohomology spectral sequences shows the commutativity of the diagram in the statement of Lemma~\ref{2.1global}. 
\end{proof} 

\begin{remark} It is easy to see that we can weaken the hypothesis of Lemma~\ref{2.1} and Lemma~\ref{2.1global} from   isolated lci  singularities to isolated  singularities of  depth $\geq 3$. 
\end{remark}

\begin{lemma}\label{bigCD} We have  the following commutative diagram with exact rows: 
\begin{equation}\label{CD4.3}  
\begin{CD}
{} @. H^1(T^0_Y) @>{=}>>  H^1(T^0_Y) @. @. @. \\
@. @VVV @VVV @. @. \\
H^1_E(\hY; T_{\hY})@>>> H^1(\hY; T_{\hY}) @>>> \mathbb{T}^1_Y @>>>H^2_E(\hY; T_{\hY}) @>>> H^2(\hY; T_{\hY})\\
@| @VVV @VVV @|  @VVV\\
H^1_E(\hY; T_{\hY})@>>> H^0_Z(R^1\pi_*T_{\hY}) @>>> H^0(T^1_Y) @>>>H^2_E(\hY; T_{\hY}) @>>> H^0(R^2\pi_*T_{\hY})\\
@VVV @| @VVV @VVV @|\\
H^1(\hY; T_{\hY})@>>> H^0(R^1\pi_* T_{\hY}) @>>> H^2(T^0_Y) @>>>H^2(\hY; T_{\hY}) @>>> H^0(R^2\pi_*T_{\hY})
\end{CD}
\end{equation}
\end{lemma}

Here the second row  is the local cohomology exact sequence, using Lemma~\ref{2.1global}.     
The bottom two rows come  from the Leray spectral sequences in local or ordinary cohomology, using $H^0_Z(R^1\pi_*T_{\hY}) = H^0(R^1\pi_*T_{\hY})$, $H^i_Z(R^1\pi_*T_{\hY}) = 0$ for $i>0$, $T^0_Y \cong R^0\pi_*T_{\hY}$ and $H^2_Z(Y; T^0_Y) \cong H^1(U; T^0_Y|U)  \cong  H^0(Y; T^1_Y)$.
Note  that the middle column is exact. The map $H^1(\hY; T_{\hY})  \to \mathbb{T}^1_Y$ is the map on tangent spaces of a corresponding morphism on functors $\mathbf{Def}_{\hY} \to \mathbf{Def}_Y$, which is the global analogue of the morphism of functors $\mathbf{Def}_{\hX} \to \mathbf{Def}_X$ discussed in Section 1.

In case $\pi\colon \hY \to Y$ is not necessarily crepant, $K_{\hY} =\pi^*\omega_Y\otimes \scrO_{\hY}(D)$, where $D$ is an effective divisor supported in $E$. 
Thus 
$$T_{\hY} \cong \Omega^{n-1}_{\hY}\otimes K_{\hY}^{-1} = \Omega^{n-1}_{\hY}\otimes \pi^*\omega_Y^{-1}\otimes\scrO_{\hY}(-D),$$
and hence 
$$\Omega^{n-1}_{\hY}\otimes \pi^*\omega_Y^{-1}  \cong T_{\hY}\otimes \scrO_{\hY}(D).$$
Then, by Lemma~\ref{imageH0},  
$$(R^0\pi_*\Omega^{n-1}_{\hY})\otimes \omega_Y^{-1} = R^0\pi_*(\Omega^{n-1}_{\hY}\otimes \pi^*\omega_Y^{-1}) =   R^0\pi_*(T_{\hY}\otimes \scrO_{\hY}(D)) = R^0\pi_*T_{\hY},$$
and we can replace $T_{\hY}$ by $\Omega^{n-1}_{\hY}\otimes \pi^*\omega_Y^{-1}$ throughout the  diagram (\ref{CD4.3}). Note that, for $i> 0$,
$$R^i\pi_*(\Omega^{n-1}_{\hY}\otimes \pi^*\omega_Y^{-1}) = (R^i\pi_*\Omega^{n-1}_{\hY})\otimes \omega_Y^{-1} \cong R^i\pi_*\Omega^{n-1}_{\hY},$$
after choosing a local trivialization of $\omega_Y$ near each $x\in Z$. 

In this case, we have the following:

\begin{lemma}\label{Fano1} The diagram
$$\begin{CD}
H^0(Y;R^1\pi_*(\Omega^{n-1}_{\hY}\otimes \pi^*\omega_Y^{-1})) \cong H^0(Y;R^1\pi_*\Omega^{n-1}_{\hY}) @>{d_2}>> H^2(Y;R^0\pi_*(\Omega^{n-1}_{\hY}\otimes \pi^*\omega_Y^{-1}))\\
@VVV @VV{\cong}V\\
H^0(Y;T^1_Y) @>>> H^2(Y;T^0_Y) 
\end{CD}$$
is commutative. The image of  the map $H^0(Y;R^1\pi_*(\Omega^{n-1}_{\hY}\otimes \pi^*\omega_Y^{-1})) \to H^0(Y;T^1_Y)$ is the same as the image of the natural map $H^0(Y;R^1\pi_*\Omega^{n-1}_{\hY}) \to H^0(Y;T^1_Y)$  given in Lemma~\ref{localCD}.  
\end{lemma}
\begin{proof} Comparing the Leray spectral sequences in local and ordinary cohomology, there is a commutative diagram
$$\begin{CD}
H^0_Z(Y; R^1\pi_*(\Omega^{n-1}_{\hY}\otimes \pi^*\omega_Y^{-1})) @>{=}>> H^0(Y; R^1\pi_*(\Omega^{n-1}_{\hY}\otimes \pi^*\omega_Y^{-1}))\\
@V{d_2}VV @VV{d_2}V \\
H^2_Z(Y; R^0\pi_*(\Omega^{n-1}_{\hY}\otimes \pi^*\omega_Y^{-1})) \cong H^2_Z(Y;T^0_Y) @>>> H^2(Y;R^0\pi_*(\Omega^{n-1}_{\hY}\otimes \pi^*\omega_Y^{-1})) \cong  H^2(Y;T^0_Y).
\end{CD}$$
By Lemma~\ref{2.1global}, $H^2_Z(Y;T^0_Y) \cong H^0(Y; T^1_Y)$ and the map $ H^2_Z(Y;T^0_Y) \to H^2(Y;T^0_Y)$ is identified with the map $H^0(Y; T^1_Y) \to H^2(Y;T^0_Y)$. Thus the diagram in the statement of Lemma~\ref{Fano1} is commutative. The   last statement follows from the local case (Lemma~\ref{localCD}).
\end{proof}

We turn now to the case of Fano varieties, and begin with a very general and easy unobstructedness theorem in the special  case of $1$-Du Bois singularities, which however will only be of limited use. Here, in case the singularities of $Y$ are not isolated, we use the definition of $1$-Du Bois described in Remark~\ref{altDBdef}. 
 
 \begin{theorem}\label{Fano1DB}  Suppose that $Y$ is a  projective variety   with   $1$-Du Bois local complete intersection singularities, not necessarily isolated, such that  $\omega^{-1}_Y$ is ample. Then $\mathbb{T}^i_Y =0$ for all $i\ge 2$ and in particular all deformations of $Y$ are unobstructed. Moreover, if the singularities of $Y$ are also isolated, then $H^i(Y; T^0_Y) = 0$ for $i\geq 3$. 
 \end{theorem}
 \begin{proof} By e.g.\ \cite[Proposition 2.4.8]{Sernesi}, since $Y$ is reduced with  only local intersection singularities,  we have $\mathbb{T}^i_Y =\Ext^i(\Omega^1_Y, \scrO_Y)$, which is Serre dual to $H^{n-i}(Y; \Omega^1_Y\otimes \omega_Y)$. By the $1$-Du Bois assumption,
 $$H^{n-i}(Y; \Omega^1_Y\otimes \omega_Y) =\mathbb{H}^{n-i}(Y; \underline\Omega^1_Y\otimes \omega_Y),$$ 
 where $\underline\Omega^1_Y$ is first graded piece of the filtered  de Rham complex (Remark~\ref{altDBdef}). 
If $i\ge 2$ and hence $n-i+1< n$, then $\mathbb{H}^{n-i}(Y; \underline\Omega^1_Y\otimes \omega_Y) =0$ by the generalization of the Akizuki-Nakano vanishing theorem due to Guill\'en, Navarro Aznar, Pascual-Gainza and Puerta \cite[V Theorem (7.10)]{GNPP} as $\omega_Y$ is the dual of an ample line bundle and $Y$ has lci singularities. Thus $\mathbb{T}^i_Y = 0$ for all $i\ge 2$. 
 
 The final statement then follows from the Ext spectral sequence, since $T^k_Y =Ext^k(\Omega^1_Y, \scrO_Y)=0$ for all $k \ge 2$ and $H^k(Y; T^1_Y) =0$  for $k > 0$ since the singularities of $Y$ are isolated. Thus, for $i \ge 3$, $H^i(Y; T^0_Y) = \mathbb{T}^i_Y =0$. 
 \end{proof}

To get actual smoothing results, even to first order, we begin with a somewhat \emph{ad hoc} definition.
 
 \begin{definition} A \textsl{generalized Fano variety $Y$} is a projective variety $Y$ with isolated rational Gorenstein singularities such that  $\omega^{-1}_Y$ is ample.  
  \end{definition}

 \begin{lemma}\label{Fano0}   Let $Y$ be a generalized Fano variety of dimension $n$. 
 \begin{enumerate}
 \item[\rm(i)] For $i> 0$, $H^i(Y;\scrO_Y) = H^i(\hY; \scrO_{\hY}) = 0$.  More generally, $H^i(Y;\omega_Y^{-k}) =0$ for $i> 0$ and $k\ge 0$. 
  \item[\rm(ii)]    $H^0(\hY; \Omega^1_{\hY}) =H^0(\hY; \Omega^1_{\hY}(\log E)) =0$.
   \item[\rm(iii)] If $H$ is an effective Cartier divisor on $Y$ such that $\omega_Y = \scrO_Y(-H)$, then $H^i(H;\scrO_H(-kH)) =0$ for all $i< n-1$ and all $k> 0$, as well as for $0< i< n-1$ and $k=0$.  
   \end{enumerate}
 \end{lemma}
 \begin{proof} (i)  By the Kawamata-Viehweg vanishing theorem, $H^i(\hY;   \pi^*\omega_Y) = 0$ for $i< n$ since   $\pi^*\omega_Y^{-1}$   is a nef  and big line bundle. Thus $H^i(Y; \omega_Y) =0$ for $i< n$ by the assumption of rational singularities. By duality, $H^i(Y;\scrO_Y) = 0$ for $i> 0$, and hence $H^i(\hY; \scrO_{\hY}) = 0$, again because the singularities of $Y$ are rational. The vanishing of $H^i(Y;\omega_Y^{-k})$ is similar. 
 
 \smallskip
 \noindent (ii) By the Hodge symmetries on the Moishezon manifold $\hY$, $h^0(\hY; \Omega^1_{\hY}) = h^1(\hY; \scrO_{\hY}) = 0$. Taking Poincar\'e residue, we have an exact sequence
 $$0 = H^0(\hY; \Omega^1_{\hY}) \to H^0(\hY; \Omega^1_{\hY}(\log E))\to \bigoplus_iH^0(E_i; \scrO_{E_i}) \to H^1(\hY; \Omega^1_{\hY}).$$
 Next we claim that the classes $[E_i] \in H^2(\hY)$ are linearly independent. This follows from:
 
 \begin{lemma} Let $\hY$ be a smooth projective variety of dimension $n \ge 2$, and let $\pi\colon \hY \to Y$ be a morphism from $\hY$ to an analytic space $Y$ such that
 \begin{enumerate} 
  \item[\rm(i)] There exists a divisor $E =\bigcup _iE_i$, in $\hY$ , where $E_i$ are the irreducible components of $E$,  with $\pi(E) = Z$ a finite number of points; 
 \item[\rm(ii)]  $\pi|\hY -E \to Y-Z$ is an isomorphism. 
 \end{enumerate}  
 Then the classes $[E_i] \in H^2(\hY)$ are linearly independent. 
 \end{lemma}
 \begin{proof} The lemma is well-known if $n=2$ since the intersection matrix $(E_i\cdot E_j)$ is negative. Assume that the lemma has been proved for $\hY$ of dimension $n-1$, and let $\widehat{A} \subseteq \hY$ be a  general very ample divisor. Setting $A = \pi(\widehat{A})$ and $E_i' = E_i\cap \widehat{A}$, the morphism $\pi|\widehat{A}\colon \widehat{A} \to A$ and the divisor $E' = \bigcup _iE_i'$ satisfy the assumptions of the lemma. Hence the classes $[E_i'] \in H^2(\widehat{A})$ are linearly independent. Since $[E_i'] $ is the image of  $[E_i] \in H^2(\hY)$, the classes  $[E_i]$  must be  linearly independent in $H^2(\hY)$ as well. 
 \end{proof}

 Returning to the proof of (ii), the lemma implies that  $\bigoplus_iH^0(E_i; \scrO_{E_i}) \to H^1(\hY; \Omega^1_{\hY})$ is injective. Thus $H^0(\hY; \Omega^1_{\hY}(\log E)) =0$. 
 
  \smallskip
 \noindent (iii) We have the exact sequence
  $$0 \to \scrO_Y(-(k+1)H) \to \scrO_Y(-kH) \to  \scrO_H(-kH) \to 0.$$
  By  the Kawamata-Viehweg vanishing theorem as in (i) and the fact that $Y$ has rational singularities, we have  $H^i(Y; \scrO_Y(-kH)) = H^i(\hY; \pi^*\omega_Y^{\otimes k}) = 0$ for $i< n$ and $k> 0$. Thus $H^i(H;\scrO_H(-kH)) =0$ for   $i< n-1$ and   $k> 0$.  The case  $0< i< n-1$ and $k=0$ is similar.
 \end{proof}

  We can now state the main theorem of this section as follows.

\begin{theorem}\label{Fanofirst} Let $Y$ be a generalized Fano variety of dimension $n\geq 3$ such that   the singularities of $Y$ are local complete intersections.   
\begin{enumerate}
\item[\rm(i)] Suppose that the map $H^2(\hY; \Omega^{n-1}_{\hY}\otimes \pi^*\omega_Y^{-1})\to
 H^0(Y;R^2\pi_*(\Omega^{n-1}_{\hY}\otimes \pi^*\omega_Y^{-1}))$ is injective. Then  $\mathbb{T}^2_Y =0$, and  the induced homomorphism  
$$\mathbb{T}^1_Y \to H^0(Y; T^1_Y)/\im H^0(\hY; R^1\pi_*\Omega^{n-1}_{\hY}\otimes \pi^*\omega_Y^{-1})=K=\bigoplus_{x\in Z}K_x$$ is surjective.
\item[\rm(ii)] If $H^3(Y; T^0_Y) = 0$ and there exists a Cartier divisor  $H$ on $Y$ with $\omega_Y^{-1}=\scrO_Y(H)$ such that 
\begin{enumerate}
\item $H\cap Y_{\text{\rm{sing}}} = \emptyset$;
\item $H^{n-3}(H; \Omega^1_H) =0$,
\end{enumerate}    
then the map 
$H^2(\hY; \Omega^{n-1}_{\hY}\otimes \pi^*\omega_Y^{-1})\to
 H^0(R^2\pi_*(\Omega^{n-1}_{\hY}\otimes \pi^*\omega_Y^{-1}))$ is injective and hence the conclusions of {\rm(i)}  hold  for $Y$.  
 \item[\rm(iii)] If $H$ is as in {\rm{(ii)}}  and is   $1$-Du Bois in the sense of Remark~\ref{altDBdef}, then $H^{n-3}(H; \Omega^1_H) =0$ $\iff$ $H^{n-3}(\hY; \Omega^1_{\hY}(\log E)) =0$ 
 $\iff$ the map $\bigoplus _i H^{n-3}(E_i; \scrO_{E_i}) \to H^{n-2}(\hY; \Omega^1_{\hY})$ is injective and $\bigoplus _i H^{n-4}(E_i; \scrO_{E_i}) \to H^{n-3}(\hY; \Omega^1_{\hY})$ is surjective. 
\end{enumerate} 
\end{theorem}
\begin{proof} Proof of (i):  By Lemma~\ref{imageH0}, 
 $$R^0\pi_*(\Omega^{n-1}_{\hY}\otimes \pi^*\omega_Y^{-1}) = R^0\pi_*(T_{\hY}\otimes \scrO_{\hY}(D)) =T^0_Y,$$
 since $\pi$ is equivariant.  The Leray spectral sequence with $E_2$ page
 $$E_2^{p,q} = H^p(Y; R^q\pi_*(\Omega^{n-1}_{\hY}\otimes \pi^*\omega_Y^{-1}))  \implies H^{p+q}(\hY; \Omega^{n-1}_{\hY}\otimes \pi^*\omega_Y^{-1})$$
 satisfies $E_2^{p,q} \neq 0$ only for $p=0$ or $q=0$. Thus there is a  long exact sequence:
\begin{gather*}
H^0(R^1\pi_*(\Omega^{n-1}_{\hY}\otimes \pi^*\omega_Y^{-1})) \to H^2(R^0\pi_*(\Omega^{n-1}_{\hY}\otimes \pi^*\omega_Y^{-1}))\to H^2(\hY; \Omega^{n-1}_{\hY}\otimes \pi^*\omega_Y^{-1})\to \\
\to H^0(R^2\pi_*(\Omega^{n-1}_{\hY}\otimes \pi^*\omega_Y^{-1})) \to H^3(R^0\pi_*(\Omega^{n-1}_{\hY}\otimes \pi^*\omega_Y^{-1})).
\end{gather*}
  By assumption, the map $H^2(\hY; \Omega^{n-1}_{\hY}\otimes \pi^*\omega_Y^{-1})\to
 H^0(Y;R^2\pi_*(\Omega^{n-1}_{\hY}\otimes \pi^*\omega_Y^{-1}))$ is injective.
Thus the map 
$$H^0(Y; R^1\pi_*(\Omega^{n-1}_{\hY}\otimes \pi^*\omega_Y^{-1})) \to H^2(Y;R^0\pi_*(\Omega^{n-1}_{\hY}\otimes \pi^*\omega_Y^{-1}))= H^2(Y;T^0_Y)$$ is surjective. By Lemma~\ref{Fano1}, the map $H^0(Y;T^1_Y) \to H^2(Y;T^0_Y)$ is also surjective. Hence $\mathbb{T}^2_Y =0$. Moreover, given an element $\xi$ of $H^0(Y;T^1_Y)$, we can modify $\xi$ by an element $\theta$ in the image of the map 
$$H^0(Y;R^1\pi_*(\Omega^{n-1}_{\hY}\otimes \pi^*\omega_Y^{-1}))\to H^0(Y;T^1_Y)$$ so that $\xi+ \theta$ maps to $0$ in $H^2(Y;T^0_Y)$, and hence is in the image of the map $\mathbb{T}^1_Y \to H^0(Y;T^1_Y)$. It follows that  $\mathbb{T}^1_Y \to H^0(Y;T^1_Y)/\im H^0(Y;R^1\pi_*(\Omega^{n-1}_{\hY}\otimes \pi^*\omega_Y^{-1}))$ is surjective. 

\smallskip
 \noindent Proof of (ii):   By  assumption,  $H^3(Y; T^0_Y) =0$, and so  $H^3(Y;R^0\pi_*(\Omega^{n-1}_{\hY}\otimes \pi^*\omega_Y^{-1}))= 0$ as well.  The  long exact sequence  in the proof of Part (i) shows that 
 the map $$H^2(\hY; \Omega^{n-1}_{\hY}\otimes \pi^*\omega_Y^{-1})\to
 H^0(Y;R^2\pi_*(\Omega^{n-1}_{\hY}\otimes \pi^*\omega_Y^{-1}))$$ 
 is surjective. To see that it is injective, it suffices to prove that 
 $$\dim H^2(\hY; \Omega^{n-1}_{\hY}\otimes \pi^*\omega_Y^{-1})\le 
\dim  H^0(Y;R^2\pi_*(\Omega^{n-1}_{\hY}\otimes \pi^*\omega_Y^{-1})).$$

We have  $R^2\pi_*(\Omega^{n-1}_{\hY}\otimes \pi^*\omega_Y^{-1}) \cong R^2\pi_* \Omega^{n-1}_{\hY}$. By Corollary~\ref{Hn-1E}, 
$$H^0(Y;R^2\pi_*(\Omega^{n-1}_{\hY}\otimes \pi^*\omega_Y^{-1}))  \cong \bigoplus_i H^2(E_i; \Omega^{n-1}_{E_i}).$$
Thus 
$$\dim H^0(Y;R^2\pi_*(\Omega^{n-1}_{\hY}\otimes \pi^*\omega_Y^{-1})) = \sum_ih^{n-1, 2}(E_i) = \sum_ih^{0,n-3}(E_i).$$
Hence, it will suffice to prove that $\dim H^2(\hY; \Omega^{n-1}_{\hY}\otimes \pi^*\omega_Y^{-1})\leq  \sum_ih^{0,n-3}(E_i)$.
   By duality, we must show that 
   $$\dim H^{n-2}(\hY; \Omega^1_{\hY}\otimes \pi^*\omega_Y)\leq \sum_ih^{0,n-3}(E_i).$$
   
  The assumption (a) that $H\cap Y_{\text{\rm{sing}}} = \emptyset$ means that we can identify $H$ with its preimage in $\hY$.   Using the conormal sequence
$$0 \to \scrO_H(-H) \to \Omega^1_{\hY} |H \to \Omega^1_H \to 0,$$
and the assumption (b) that $H^{n-3}(H; \Omega^1_H) = 0$,   $H^{n-3}(\hY; \Omega^1_{\hY} |H) =0$ and there is an exact sequence
$$0\to H^{n-2}(\hY; \Omega^1_{\hY} \otimes \scrO_{\hY}(-H)) \to H^{n-2}(\hY; \Omega^1_{\hY}) \to H^{n-2}(H; \Omega^1_{\hY} |H).$$
Thus $H^{n-2}(\hY; \Omega^1_{\hY}\otimes \pi^*\omega_Y) = H^{n-2}(\hY; \Omega^1_{\hY} \otimes \scrO_{\hY}(-H)) \cong \Ker\{H^{n-2}(\hY; \Omega^1_{\hY}) \to H^{n-2}(\hY; \Omega^1_{\hY} |H)\} $ and we must show that 
 $\dim \Ker\{H^{n-2}(\hY; \Omega^1_{\hY}) \to H^{n-2}(\hY; \Omega^1_{\hY} |H)\} \leq \sum_ih^{0,n-3}(E_i)$. 
 
 Note that $Y-H$ is affine as $H$ is an ample divisor. By the Goresky-MacPherson-Lefschetz theorem \cite{GoreskyMacPherson88}, $H^k(Y-H) =0$ for $k\ge n+1$. The exact sequence of the pair $(\hY-H, E)$ leads to an exact sequence of mixed Hodge structures (cf.\ \cite[Proposition 5.46]{PS})
 $$\cdots \to H^k(\hY-H, E) \to H^k(\hY - H) \to H^k(E) \to \cdots .$$
 Moreover $ H^k(\hY-H, E) \cong  H^k(Y-H, Z)$, where as usual $Z=Y_{\text{\rm{sing}}}$.  Since $Z$ is a finite set of points, $H^k(Y-H, Z) \cong H^k(Y-H)$ for $k > 1$. Thus, in this range, there is an exact sequence
$$\cdots \to H^k(Y-H) \to H^k(\hY - H) \to H^k(E) \to \cdots .$$
(One can also obtain this long exact sequence via the   Mayer-Vietoris sequence.) 
Hence, for $k \ge n+1$, there is an isomorphism of mixed Hodge structures $H^k(\hY - H) \cong H^k(E)$.  Likewise,  the exact sequence of the pair $(\hY, H)$ gives an exact sequence of mixed Hodge structures
$$\cdots \to H^{n-1}(\hY, H) \to H^{n-1}(\hY) \to H^{n-1}(H)\to \cdots $$
By strictness there is an exact sequence
$$\Gr^1_FH^{n-1}(\hY, H) \to \Gr^1_FH^{n-1}(\hY) \to \Gr^1_FH^{n-1}(H).$$
As $\Gr^1_FH^{n-1}(\hY) = H^{n-2}(\hY; \Omega^1_{\hY})$ and $\Gr^1_FH^{n-1}(H)= \mathbb{H}^{n-2}(H; \underline{\Omega}^1_H)$, where $\underline{\Omega}^1_H$ is the first graded piece of the filtered de Rham complex for $H$, this identifies the kernel of $H^{n-2}(\hY; \Omega^1_{\hY}) \to \mathbb{H}^{n-2}(H; \underline{\Omega}^1_H)$ with the image of $\Gr^1_FH^{n-1}(\hY, H)$. 
Since $\hY$ is smooth and compact, by \cite[Lemma B.21, Theorem B.24, \S5.5, \S6.3]{PS}, there are isomorphisms of mixed Hodge structures
$$H^{n-1}(\hY, H) \cong  H^{n-1}_c(\hY- H)\cong H_{n+1}(\hY-H)(-n) = H^{n+1}(\hY-H)\spcheck(-n) = H^{n+1}(E)\spcheck(-n).$$
Moreover $\Gr^1_FH^{n+1}(E)\spcheck(-n) = \Gr^{1-n}_FH^{n+1}(E)\spcheck = (\Gr^{n-1}_FH^{n+1}(E))\spcheck$. 
For the normal crossing divisor $E$, which has dimension $n-1$,
$$\Gr^{n-1}_FH^{n+1}(E) = H^2(E; \Omega^{n-1}_E/\tau^{n-1}_E) = \bigoplus_iH^2(E_i; \Omega^{n-1}_{E_i}).$$
Hence, by Serre duality on the $E_i$, 
$$(\Gr^{n-1}_FH^{n+1}(E))\spcheck\cong \left(\bigoplus_iH^2(E_i; \Omega^{n-1}_{E_i})\right)\spcheck \cong \bigoplus_iH^{n-3}(E_i; \scrO_{E_i}).$$
In particular, 
$$\dim \Ker\{H^{n-2}(\hY; \Omega^1_{\hY}) \to \mathbb{H}^{n-2}(H; \underline{\Omega}^1_H)\} \leq \sum_ih^{0,n-3}(E_i).$$
On the other hand, the homomorphism $H^{n-2}(\hY; \Omega^1_{\hY}) \to \mathbb{H}^{n-2}(H; \underline{\Omega}^1_H)$ factors as
$$H^{n-2}(\hY; \Omega^1_{\hY}) \to H^{n-2}(\hY; \Omega^1_{\hY} |H) \to H^{n-2}(H; \Omega^1_H) \to H^{n-2}(H; \Omega^1_{H_{\text{\rm{red}}}}) \to \mathbb{H}^{n-2}(H; \underline{\Omega}^1_H).$$
So finally
\begin{align*}
\dim \Ker\{H^{n-2}(\hY; \Omega^1_{\hY}) \to H^{n-2}(\hY; \Omega^1_{\hY} |H)\} &\leq \dim \Ker\{H^{n-2}(\hY; \Omega^1_{\hY}) \to \mathbb{H}^{n-2}(H; \underline{\Omega}^1_H)\}  \\
&\leq \sum_ih^{0,n-3}(E_i),
\end{align*}
 as claimed.
 
 Proof of (iii): Arguing as in (ii), we have an exact sequence of mixed Hodge structures
 $$\begin{CD}
 H^{n-2}(\hY, H) @>>>  H^{n-2}(\hY) @>>>  H^{n-2}(H)  @>>> H^{n-1}(\hY, H) @>>> H^{n-1}(\hY) \\
 @VV{\cong}V @. @.   @VV{\cong}V @.\\
 H^{n+2}(E)\spcheck(-n) @. @. @. H^{n+1}(E)\spcheck(-n)
 \end{CD}$$
 Thus there is a corresponding exact sequence
 $$\Gr^1_FH^{n+2}(E)\spcheck(-n) \to  \Gr^1_FH^{n-2}(\hY)  \to \Gr^1_FH^{n-2}(H) \to  \Gr^1_FH^{n+1}(E)\spcheck(-n)\to \Gr^1_FH^{n-1}(\hY).$$ 
 As before, $\Gr^1_FH^k(\hY)\cong H^{k-1}(\hY; \Omega^1_{\hY})$ and 
 $$\Gr^1_FH^k(E)\spcheck(-n) \cong \bigoplus_i H^{2n-k -2}(E_i; \scrO_{E_i}).$$ 
By the $1$-Du Bois assumption on $H$,  $\Gr^1_FH^{n-2}(H)= \mathbb{H}^{n-3}(H; \underline{\Omega}_H^1) = H^{n-3}(H; \Omega^1_H)$. Thus there is an exact sequence
$$\bigoplus_i H^{n-4}(E_i; \scrO_{E_i}) \to H^{n-3}(\hY; \Omega^1_{\hY})\to H^{n-3}(H; \Omega^1_H) \to \bigoplus_i H^{n-3}(E_i; \scrO_{E_i}) \to H^{n-2}(\hY; \Omega^1_{\hY}).$$
It is clear by construction that the maps $\bigoplus_i H^{k-1}(E_i; \scrO_{E_i}) \to H^k(\hY; \Omega^1_{\hY})$ are induced by the Gysin maps $H^{k-1}(E_i) \to H^{k+1}(\hY)$.  In particular, comparing the above exact sequence with the long exact Poincar\'e residue sequence  
$$
\bigoplus_i H^{n-4}(E_i; \scrO_{E_i}) \to  H^{n-3}(\hY; \Omega^1_{\hY})\to  H^{n-3}(\hY; \Omega^1_{\hY}(\log E)) \to \bigoplus_i H^{n-3}(E_i; \scrO_{E_i}) \to H^{n-2}(\hY; \Omega^1_{\hY}),
$$
we see that  $H^{n-3}(H; \Omega^1_H) =0$ $\iff$ the map $\bigoplus_i H^{n-4}(E_i; \scrO_{E_i}) \to  H^{n-3}(\hY; \Omega^1_{\hY})$ is surjective and the map $\bigoplus_i H^{n-3}(E_i; \scrO_{E_i}) \to H^{n-2}(\hY; \Omega^1_{\hY})$ is injective $\iff$ $H^{n-3}(\hY; \Omega^1_{\hY}(\log E)) =0$. 
\end{proof}

In dimension three, we have the following more precise result:

\begin{theorem}\label{Fanothree} Suppose that $Y$ is a generalized Fano threefold with local complete intersection singularities. Then $H^3(Y; T^0_Y) = 0$. Moreover, either there exists a smooth element $H\in |\omega_Y^{-1}|$, with $H^0(H; \Omega^1_H) =0$, or $Y$ is a complete intersection in a weighted projective space $\Pee(1,1,1,1,2,3)$ of degree $(2,6)$.
\end{theorem}
\begin{proof} 
By a well-known  Riemann-Roch calculation (cf.\ \cite[\S4.4]{Reid}),  
 $$h^0(Y;  \omega_Y^{-1}) =   -\frac{c_1(\omega_Y)^3}{2} + 3 > 0.$$ 
Hence $\omega_Y= \scrO_Y(-H)$ for some effective Cartier divisor $H$.

To show that that $H^3(Y; T_Y^0)  = 0$,  it suffixes by duality to show that $H^3(Y; T_Y^0)\spcheck = \Hom (T^0_Y, \omega_Y) = 0$. By Lemma~\ref{reflexive}, $\Omega^1_Y$ is reflexive and $\Omega^1_Y = \pi_*\Omega^1_{\hY}$.  Thus,  
 $$\Hom (T^0_Y, \omega_Y) = H^0(Y; (\Omega^1_Y)\spcheck{}\spcheck \otimes \omega_Y) = H^0(Y; \Omega^1_Y\otimes \scrO_Y(-H))$$ 
 for some effective Cartier divisor $H$.
 By  Lemma~\ref{Fano0},   
  $H^0(Y; \Omega^1_Y) =  H^0(Y; \pi_*\Omega^1_{\hY}) = H^0(\hY;  \Omega^1_{\hY}) =0$.  Using the inclusion $\Omega^1_Y\otimes \scrO_Y(-H) \subseteq \Omega^1_Y$ and the fact that $H^0(Y; \Omega^1_Y) =0$,   it follows that $H^0(Y; \Omega^1_Y\otimes \scrO_Y(-H)) =0$.  Hence   $\Hom (T^0_Y, \omega_Y)=0$ and therefore $H^3(Y; T_Y^0)\spcheck=0$  as well. 
  
If there exists a smooth    $H\in |\omega_Y^{-1}|$, then $H$  is a (smooth)  $K3$ surface  and hence $H^0(H; \Omega^1_H) =0$. In the general case, by a theorem of Shokurov and Reid \cite[Theorem 0.5]{Reid}, a general element of $|\omega_Y^{-1}|$ is a $K3$ surface, possibly with rational double points. Moreover, essentially by a remark due to Reid \cite[\S0.5]{Reid}, \cite{Shin}, \cite{Mella},  standard results on nef and big linear systems on $K3$ surfaces imply that one of  following holds:  (i)  The linear system $|\omega_Y^{-1}|$ is base point free; (ii) The base locus of $|\omega_Y^{-1}|$ is isomorphic to $\Pee^1$ and the general element of $|\omega_Y^{-1}|$ is smooth; (iii) The base locus of $|\omega_Y^{-1}|$ is a single point $y$, the general element $H$ of $|\omega_Y^{-1}|$ has an $A_1$ singularity at $y$, so that $Y$ has a compound $A_1$ singularity at $y$, and $H^3 =2$. In this case,   $Y$ is a complete intersection in a weighted projective space $\Pee(1,1,1,1,2,3)$ of degree $(2,6)$ \cite[Theorem 2.9]{Mella}, 
   \cite[Proposition 4.1]{JahnkeRadloff}.  Thus, if there does not exist a smooth    $H\in |\omega_Y^{-1}|$, then the second case of the last sentence in Theorem~\ref{Fanothree} holds. 
\end{proof} 

\begin{corollary}\label{Fanocorollary} Suppose that $Y$ is a generalized Fano variety with only isolated hypersurface singularities,  such that one of the following holds: 
\begin{enumerate}
\item[\rm(i)] $\dim Y = 3$; 
\item[\rm(ii)] The singularities of $Y$ are $1$-irrational, there exists an element $H\in |\omega_Y^{-1}|$ with $H\cap Y_{\text{\rm{sing}}} = \emptyset$, and   $H^3(Y; T^0_Y) =H^{n-3}(H; \Omega^1_H) =0$. 
\item[\rm(iii)] $\dim Y \geq 4$, the singularities of $Y$ are $1$-liminal, there exists an element $H\in |\omega_Y^{-1}|$ with $H\cap Y_{\text{\rm{sing}}} = \emptyset$, and   $H^{n-3}(H; \Omega^1_H) =0$.
\end{enumerate}
Then $Y$ is smoothable. Moreover, every small smoothing of $Y$ is a Fano variety.
\end{corollary}
\begin{proof} 
If $\dim Y =3$ or the singularities of $Y$ are $1$-liminal, then $H^3(Y; T^0_Y) = 0$ by Theorem~\ref{Fanothree} and Theorem~\ref{Fano1DB}.  In case $\dim Y =3$ and $Y$ is a complete intersection in $\Pee(1,1,1,1,2,3)$ of degree $(2,6)$, then $Y$ is clearly smoothable (and a general smoothing of $Y$ is a double cover of $\Pee^3$ branched along a smooth sextic). By Theorem~\ref{Fanofirst} and Theorem~\ref{Fano1DB}, in all remaining cases, $\mathbb{T}^2_Y =0$, i.e.\ the deformations of $Y$ are unobstructed and the map 
$$\mathbb{T}^1_Y \to H^0(Y; T^1_Y)/\im H^0(\hY; R^1\pi_*\Omega^{n-1}_{\hY}\otimes \pi^*\omega_Y^{-1})=K=\bigoplus_{x\in Z}K_x$$ is surjective. Moreover, in case $\dim Y =3$, the singularities of $Y$ are necessarily  $1$-irrational. 

 We now argue along the lines of Corollary~\ref{meaning}. For  every $x\in Z=Y_{\text{sing}}$, the image of 
$$H^0(Y; R^1\pi_*\Omega^{n-1}_{\hY}\otimes \pi^*\omega_Y^{-1})\to T^1_{Y,x}$$  is contained in $\im  H^0(Y;(R^1\pi_*\Omega^{n-1}_{\hY})_x)$.  Since the singularities of $Y$ are   $1$-irrational in all cases,  $\im  H^0(Y;(R^1\pi_*\Omega^{n-1}_{\hY})_x)\subseteq \mathfrak{m}_xT^1_{Y,x}$. 
For   $x\in Z$, we have the element $1_x\in K_x$ which is the image of $1\in T^1_{Y,x}$ (for some identification of $T^1_{Y,x}$ with a cyclic module $\scrO_{Y,x}/I$).
Given an element $\xi \in \mathbb{T}^1_Y$, let $\overline\xi\in H^0(Y; T^1_Y)$ be the image of $\xi$ and, for $x\in Z$,  let $\overline\xi_x$ be the corresponding component of  $\overline\xi$ in $T^1_{Y,x}$. Since $\mathbb{T}^1_Y \to \bigoplus_{x\in Z}K_x$ is surjective,  there exists a $\theta \in \mathbb{T}^1_Y$ whose image in $K_x$ is $1_x$ for every $x\in Z$. For every $x\in Z $, $\overline\theta_x \equiv 1 \bmod  \im H^0(Y;(R^1\pi_*\Omega^{n-1}_{\hY})_x)$, and hence $\overline\theta_x \equiv 1 \bmod \mathfrak{m}_xT^1_{Y,x}$. Thus $\theta$ is a first order smoothing of $Y$. Since the deformations of $Y$ are unobstructed, there exists a $1$-parameter deformation $\mathcal{Y}\to \Delta$ of $Y$ whose Kodaira-Spencer class is $\theta$.  By Lemma~\ref{defsmooth}, $\mathcal{Y}$ is smooth, and in particular is a smoothing of $Y$. The final statement is clear since $\omega_Y^{-1}$ remains ample for every small smoothing. 
\end{proof}

\begin{remark}\label{Fanoremark} (i) If $\dim Y =3$ and $Y$ has only ordinary double point singularities, then, by \cite[Theorem 4.2]{F} and \cite[Proposition 4]{NamikawaF}, $H^2(Y; T^0_Y) =0$ and hence the map $\mathbb{T}^1_Y \to H^0(Y; T^1_Y)$ is surjective. In particular, the deformations of $Y$ are versal for the deformations of the singular points.  On the other hand, Namikawa has constructed an example of a generalized Fano threefold $Y$ such that the map $\mathbb{T}^1_Y \to H^0(Y; T^1_Y)$ is not surjective \cite[Example 5]{NamikawaF}. Hence, our general strategy of working with the quotient $K$ of $H^0(Y; T^1_Y)$ seems to be required for this case.

\smallskip
\noindent (ii)  If $\dim Y =3$ and $Y$ is a $(2,6)$ complete intersection in $\Pee(1,1,1,1,2,3)$,   a direct argument shows that $\mathbb{T}^2_Y =0$:   By looking at the hyperplane sections, it is easy to check by hand that  $Y$ is contained in the smooth locus of $P=\Pee(1,1,1,1,2,3)$. Then $\mathbb{T}^2_Y =\mathbb{H}^2(\mathcal{C}^\bullet)$, where $\mathcal{C}^\bullet$ is the complex (in degrees $0$ and $1$)
$$T_P|Y \to N_{Y/P},$$
where $N_{Y/P}$ is the normal sheaf $(I_Y/I_Y^2)\spcheck = \scrO_Y(2) \oplus \scrO_Y(6)$. Thus there is an exact sequence 
$$H^1(Y; N_{Y/P}) \to \mathbb{T}^2_Y \to H^2(Y; T_P|Y).$$
The vanishing of $H^1(Y; N_{Y/P})$ and $ H^2(Y; T_P|Y)$ are standard, using the exact sequences 
\begin{align*}
0 \to T_P \otimes I_Y \to &T_P \to T_P|Y \to 0;
\\
0 \to \scrO_P(-8) \to  \scrO_P(-2) &\oplus \scrO_P(-6)\to I_Y \to 0.
\end{align*}
Hence $\mathbb{T}^2_Y =0$.

\smallskip
\noindent (iii) It is likely that the arguments of Theorem~\ref{Fanofirst} can be pushed to handle more general cases, without making the somewhat unnatural assumption that $H\cap Y_{\text{\rm{sing}}} = \emptyset$. Without making this assumption, let $H$ be a general element of  $|\omega_Y^{-1}|$ and let $\widehat{H} = \pi^*H$ be the total transform of $H$. In particular, $\widehat{H} \cong H$ $\iff$ $H\cap Y_{\text{\rm{sing}}} = \emptyset$. The general strategy would then be to compare the dimensions of $H^{n-3}(\widehat{H}; \Omega^1_{\widehat{H}})$ and $H^{n-3}(H; \Omega^1_H)$ and show that they differ by at most $\sum_i'h^{0, n-3}(E_i)$, where the ``$\sum'$" indicates they we only sum over the components $E_i$ of $E$ which map onto a point of $H\cap Y_{\text{\rm{sing}}}$.

\smallskip
\noindent (iv) Suppose that $\dim Y \ge 4$. Then Chmutov and Givental \cite{Varchenkobook}, \cite[Appendix to no.\ 75]{Hirzebruch} have constructed examples of projective hypersurfaces of degree $d$ in $\Pee^{n+1}$ with many ordinary double points. These lower bounds for the maximum number of nodes on a hypersurface  of degree $d$ in $\Pee^{n+1}$ are on the order of $Cd^{n+1}/\sqrt{n}$ (for $d$ fixed and $n\to \infty$).  It is easy to see that this procedure gives examples of generalized Fano varieties with only ordinary double point singularities for which the map $\mathbb{T}^1_Y \to H^0(Y; T^1_Y)$ is not surjective. For example, using sharp bounds for the maximum number of nodes on cubic hypersurfaces and strong lower bounds for the maximum number of nodes on quartic hypersurfaces due to Goryunov \cite{Goryunov},  such examples exist  for $\deg Y = 3$  and $\dim Y \ge 7$ (cubic hypersurfaces), as well as $\deg Y = 4$ and $\dim Y \ge 4$ (quartic hypersurfaces).

\smallskip
\noindent (v) Consider the hypersurface of degree $n+1$ in $\Pee^{n+1}$ defined by $z_0(\sum_{i=1}^{n+1}z_i^n) - (\sum_{i=1}^{n+1}z_i^{n+1})$. This hypersurface is a generalized Fano variety whose only singular point is at $x=(1, 0, \dots, 0)$ analytically isomorphic to the cone over the Fermat hypersurface of degree $n$ in $\Pee^n$. In particular, it is a rational but strongly $1$-irrational singularity. It is easy to check that, for $n \ge 4$,  the map $\mathbb{T}^1_Y \to H^0(Y; T^1_Y)$ is not surjective by looking at the positive weight space of the weighted homogeneous singularity at $x$. Cf.\ \cite[Example 5.5]{FL24} for related examples.
\end{remark}

\section{The Calabi-Yau case}\label{Section5}

\begin{definition}\label{defnCY} A \textsl{canonical Calabi-Yau variety $Y$} is a   compact analytic variety $Y$  with at worst canonical Gorenstein singularities, such that $\omega_Y\cong \scrO_Y$, and such that either $Y$ is a scheme or $Y$ has only isolated singularities and the $\partial\bar\partial$-lemma holds for some resolution of $Y$.
\end{definition}

\begin{remark}\label{ddbarremark} Assume that the singularities of $Y$ are isolated. Since we only consider resolutions of $Y$ which are an isomorphism away from the isolated singular points, and thus for which the exceptional divisors are Moishezon, it is easy to see that the $\partial\bar\partial$-lemma holds for one resolution of $Y$ $\iff$ it holds for every such. The reason that we need to assume the $\partial\bar\partial$-lemma in some form comes from Theorem~\ref{unob1} below as well as the need to know that various spectral sequences degenerate at $E_1$ (cf.\ the proof of Lemma~\ref{7.1}).  
\end{remark}

By \cite[Corollary 1.5]{FL22c}, we have:
 
\begin{theorem}\label{unob1} Let $Y$ be a canonical Calabi-Yau $n$-fold such that all singularities of $Y$ are $1$-Du Bois lci singularities, not necessarily isolated.  Then the deformations of $Y$ are unobstructed. \qed
\end{theorem}

From now on, we assume that $Y$ is a canonical Calabi-Yau variety of dimension $n\ge 3$ with isolated  singularities, such that $H^1(Y;\scrO_Y) =0$,  and let $Z$ be its singular locus.   Let $\pi\colon \hY \to Y$ be a good equivariant resolution, so that  $E =\pi^{-1}(Z)=\bigcup_{x\in Z}E_x$ is a (not necessarily connected) divisor with normal crossings.  Let $V = \hY -E= Y-Z$. For each $x\in Z$, let $X_x$ be a good Stein neighborhood of $x$. Define $\hX_x= \pi^{-1}(X_x)$, a   good equivariant resolution of $X_x$  with exceptional divisor $E_x$,
 and let $U_x  = X_x-\{x\}=\hX_x - E_x$.

 \begin{lemma}\label{Hodgezero} If $Y$ is a canonical Calabi-Yau variety of dimension $n$   and $H^1(Y;\scrO_Y) =0$, then    $H^1(\hY;\scrO_{\hY}) = H^{n-1}(\hY;\scrO_{\hY}) =0$, and hence $H^1(\hY;\Omega^n_{\hY}) = H^{n-1}(\hY;\Omega^n_{\hY}) =0$.
 \end{lemma}
\begin{proof} By Serre duality on $Y$, $H^{n-1}(Y;\scrO_Y) = H^{n-1}(Y;\omega_Y) =0$. Since $Y$ has rational singularities,  $H^1(\hY;\scrO_{\hY}) = H^{n-1}(\hY;\scrO_{\hY}) =0$. Then by Serre duality on $\hY$, $H^1(\hY;\Omega^n_{\hY}) = H^{n-1}(\hY;\Omega^n_{\hY}) =0$. 
\end{proof}

For the rest of this section, we assume that the singularities of $Y$ are isolated lci singularities. Recall from Lemma~\ref{bigCD} that we have a commutative diagram  with exact rows
\begin{equation}\label{bigCDeqn}
\begin{CD}
H^1_E(\hY; \Omega^{n-1}_{\hY}) @>>> H^0(R^1\pi_*\Omega^{n-1}_{\hY}) @>>> H^0(Y; T^1_Y) @>>> K @>>> 0 \\
@VVV @| @VVV @VVV @.\\
H^1(\hY; \Omega^{n-1}_{\hY}) @>>>  H^0(R^1\pi_* \Omega^{n-1}_{\hY}) @>>> H^2(Y;T^0_Y) @>>>H^2(\hY; \Omega^{n-1}_{\hY}) @.  {}
\end{CD} 
\end{equation}
Here $K = \Ker\{H^2_E(\hY; \Omega^{n-1}_{\hY}) \to H^0(R^2\pi_* \Omega^{n-1}_{\hY})\} = \bigoplus_{x\in Z}K_x$, where 
$$K_x =\Ker\{H^2_{E_x}(\hX_x; \Omega^{n-1}_{\hX_x}) \to H^2(\hX_x, \Omega^{n-1}_{\hX_x})\}$$
as defined in Theorem~\ref{maintheorem}(v). For each $x\in Z$, by  Theorem~\ref{maintheorem}(vi), there is an inclusion 
$$K'_x\oplus \Gr_F^{n-1}H^n(L_x) \hookrightarrow K_x,$$
 where $K'_x = \Ker\{H^2_{E_x}(\hX_x; \Omega^{n-1}_{\hX_x}) \to H^{n+1}_{E_x}(X_x)\}$. Let $K' =\bigoplus_{x\in Z}K'_x$. 

\begin{lemma}\label{7.1} For every $x\in Z$, the induced map $K'_x \to H^2(\hY; \Omega^{n-1}_{\hY})$ is zero.
\end{lemma}
\begin{proof} By Lemma~\ref{Hodgezero} and the running assumption that $H^1(Y;\scrO_Y) =0$, $H^1(\hY; \Omega^n_{\hY})   =0$. Hence there is an injection $H^2(\hY; \Omega^{n-1}_{\hY}) \to H^{n+1}(\hY)$. If 
$$H^2_{E_x}(\hX_x; \Omega^{n-1}_{\hX_x}) \cong H^2_{E_x}(\hY; \Omega^{n-1}_{\hY}) \to H^2(\hY; \Omega^{n-1}_{\hY}) \to H^{n+1}(\hY)$$
is the composition, then we have  a commutative diagram
$$\begin{CD}
H^2_{E_x}(\hX_x; \Omega^{n-1}_{\hX_x}) @>>> H^{n+1}_{E_x}(\hX_x) =  H^{n+1}_{E_x}(\hY)\\
@VVV @VVV \\
H^2(\hY; \Omega^{n-1}_{\hY}) @>>> H^{n+1}(\hY).
\end{CD}$$
Then $K_x'$, which is the kernel of the upper horizontal arrow, must map to $0$ in $H^2(\hY; \Omega^{n-1}_{\hY})$ since $H^2(\hY; \Omega^{n-1}_{\hY}) \to H^{n+1}(\hY)$ is injective.
\end{proof}  

For each $x\in Z$, we also have the homomorphism $\Gr_F^{n-1}H^n(L_x) \to H^2(\hY; \Omega^{n-1}_{\hY})$ arising from  the composition   $H^n(L_x) \to H^{n+1}_E(\hY) \to H^{n+1}(\hY)$.  By Equations (\ref{1.11}) and (\ref{1.2}), this homomorphism is given by the composition
$$H^1(\hY; \Omega^{n-1}_{\hY}(\log E)|E) \to H^1(\hY; \Omega^{n-1}_{\hY}(\log E)/\Omega^{n-1}_{\hY}) \xrightarrow{\partial}  H^2(\hY; \Omega^{n-1}_{\hY}).$$
Let $T$ be the kernel of the homomorphism
$$\bigoplus_{x\in Z}\Gr_F^{n-1}H^n(L_x) \to H^2(\hY; \Omega^{n-1}_{\hY}).$$
Thus $K'\oplus T \subseteq K$ and the induced homomorphism $K'\oplus T \to H^2(\hY; \Omega^{n-1}_{\hY})$ is zero.

 \begin{lemma}\label{liftable} Let $K_0$ be the kernel of the map $K  \to H^2(\hY; \Omega^{n-1}_{\hY})$.
 \begin{enumerate}
 \item[\rm(i)] $K_0$ contains $K'\oplus T$.
 \item[\rm(ii)] $K_0 = \Ker\{\bigoplus_{x\in Z}H^2_{E_x}(\hX_x; \Omega^{n-1}_{\hX_x}) \to H^2(\hY; \Omega^{n-1}_{\hY})\}$.
  \item[\rm(iii)] The image of the kernel of $H^0(Y;T^1_Y) \to H^2(Y; T^0_Y)$ in $K$ contains $K_0$. Hence, the image of the composition $\mathbb{T}^1_Y \to H^0(Y;T^1_Y) \to K$ contains $K_0$.
 \end{enumerate}
\end{lemma}
\begin{proof} (i) is clear from Lemma~\ref{7.1} and the above remarks. For (ii), by excision, $H^2_{E_x}(\hX_x; \Omega^{n-1}_{\hX_x}) =  H^2_{E_x}(\hY; \Omega^{n-1}_{\hY})$, and  the map $H^2_{E_x}(\hX_x; \Omega^{n-1}_{\hX_x}) \to H^2(\hX_x; \Omega^{n-1}_{\hX_x})$ factors through the natural map
 $$H^2_{E_x}(\hX_x; \Omega^{n-1}_{\hX_x}) = H^2_{E_x}(\hY; \Omega^{n-1}_{\hY}) \to   H^2(\hY; \Omega^{n-1}_{\hY})\to H^2(\hX_x; \Omega^{n-1}_{\hX_x}).$$
Hence the kernel of $\bigoplus_{x\in Z}H^2_{E_x}(\hX_x; \Omega^{n-1}_{\hX_x}) \to H^2(\hY; \Omega^{n-1}_{\hY})$ is contained in $K$ and   by definition is equal to $K_0$.

 Then  (iii) follows from a diagram chase involving  (\ref{bigCDeqn}). Alternatively, to see the last statement of (iii), we have  the commutative diagram
$$\begin{CD}
\mathbb{T}^1_Y = H^1(V, \Omega^{n-1}_V) @>>> H^2_E(\hY; \Omega^{n-1}_{\hY}) @>>> H^2(\hY; \Omega^{n-1}_{\hY})\\
@VVV @| @VVV \\
H^0(T^1_Y) @>>> \bigoplus_{x\in Z}H^2_{E_x}(\hX_x; \Omega^{n-1}_{\hX_x}) @>>> \bigoplus_{x\in Z}H^2(\hX_x; \Omega^{n-1}_{\hX_x})
\end{CD}$$
using excision to get the middle equality. Then by definition $K_0 = \Ker \Big\{K  \to H^2(\hY; \Omega^{n-1}_{\hY})\Big\}$ is contained in $\Ker \Big\{H^2_E(\hY; \Omega^{n-1}_{\hY}) \to H^2(\hY; \Omega^{n-1}_{\hY})\Big\}$ and hence is in the image of $\mathbb{T}^1_Y$. 
\end{proof} 
 
 \begin{definition} Suppose that all singularities of $Y$ are isolated lci singularities. We say that $Y$ has a \textsl{good configuration of singularities} if every singularity of $Y$ is either strongly $1$-irrational or $1$-liminal. If $x\in Z$ is a $1$-liminal hypersurface singularity, then by Proposition~\ref{limdimone},  $\dim \Gr_F^{n-1}H^n(L_x) = 1$, say $\Gr_F^{n-1}H^n(L_x) =\Cee \cdot \varepsilon_x$. Let $Z = Z' \cup Z''$, where $Z'$ is the subset of strongly $1$-irrational singularities and $Z''$ is the subset of $1$-liminal  singularities. Finally, let   $L' =\bigcup_{y\in Z'}L_y$ and $E' = \bigcup_{y\in Z'}E_y$, and similarly for $L''$ and $E''$. Note that $L'$ and $L''$ are   subspaces of $Y-Z$. 
 
If $\dim Y =3$, the condition that $Y$ has a good configuration of singularities is always satisfied by Corollary~\ref{gooddim3cor}.  In case all of the singularities of $Y$ are weighted homogeneous hypersurface singularities, $Y$ has a good configuration of singularities $\iff$ all singularities are    $1$-irrational.
 \end{definition}
 
 \begin{theorem}\label{smoothCY} Suppose that $H^1(Y;\scrO_Y) =0$,  that  all singularities of $Y$ are isolated hypersurface  singularities and that $Y$ has a good configuration of singularities in the notation of the preceding definition.  Let $Y' \to Y$ be a good equivariant resolution over the points of $Z''$, so that the singular locus of $Y'$ is $Z'$. Let $\varphi \colon   H^n(L'') \to H_{n-1}(Y')$ be the composition
 $$ H^n(L'') \cong  H_{n-1}(L'') \to H_{n-1}(Y'),$$ 
 where the second map comes from the inclusion $L'' \subseteq Y-Z'' \cong Y'- E''\subseteq Y'$. 
  Finally, suppose  that, for every $x\in Z''$, there exists a nonzero $a_x\in \Cee$   such that 
\begin{equation}\label{localobs}
\sum_{x\in Z''}a_x\varphi(\varepsilon _x) =0\in H_{n-1}(Y').
\end{equation}

 Then there exists a first order smoothing $\theta \in \mathbb{T}^1_Y$.  If the deformations of $Y$ are unobstructed, then there is a smoothing $\mathcal{Y}\to \Delta$ with smooth total space $\mathcal{Y}$. Moreover all small smoothings of $Y$ have trivial canonical bundle. 
 \end{theorem}
 \begin{proof} It is tempting  to try to apply Lemma~\ref{liftable}(iii) directly. However, the hypotheses of Theorem~\ref{smoothCY} are somewhat weaker: we assume that there is a  relation $\sum_{x\in Z''}a_x\varphi(\varepsilon _x) =0\in H_{n-1}(Y')$, not that $\sum_{x\in Z''}a_x \varepsilon _x \in T$, i.e.\ that the image of  $\sum_{x\in Z''}a_x \varepsilon _x$ is $0$ in $H^2(\hY; \Omega^{n-1}_{\hY})$. Thus we proceed slightly differently.
 
 Let $y\in Z'$ be a strongly $1$-irrational singular point of $Y$. Then by Theorem~\ref{maintheorem}(vi), $K_y \to H^2_{E_y}( \hX_y; \Omega^{n-1}_{\hX_y}(\log E_y))$ induces an isomorphism $K'_y\cong H^2_{E_y}( \hX_y; \Omega^{n-1}_{\hX_y}(\log E_y))$. Moreover, the kernel of $T^1_{Y,y} \to H^2_{E_y}( \hX_y; \Omega^{n-1}_{\hX_y}(\log E_y))$ is contained in $\mathfrak{m}_yT^1_{Y,y}$.  For each $y\in Z'$, we can then choose $1'_y\in K'_y$ mapping to the image   $1_y$ of $1\in T^1_{Y,y}$  in $H^2_{E_y}( \hX_y; \Omega^{n-1}_{\hX_y}(\log E_y))$ (for some identification of $T^1_{Y,x}$ with a cyclic module $\scrO_{Y,x}/I$). 
  By Lemma~\ref{liftable}(iii), there exists a $\theta_1\in \mathbb{T}^1_Y$ whose image in $T^1_{Y,y}$ maps to $1'_y\in K'_y$. Then, for every $y\in Z'$,  the image $(\bar\theta_1)_y$ of $\theta_1$ in $T^1_{Y,y}$  satisfies: $(\bar\theta_1)_y \equiv 1 \bmod \mathfrak{m}_yT^1_{Y,y}$. 
 
 If $x\in Z''$ is a $1$-liminal singular point of $Y$, then $K_x =\Gr_F^{n-1}H^n(L)$ has dimension $1$ and $\varepsilon_x$ is a generator. Since the kernel of $T^1_{Y,x} \to K_x$ is then $\mathfrak{m}_xT^1_{Y,x}$, every lift of $\varepsilon_x$ to $T^1_{Y,x}$ is a generator of the cyclic module  $T^1_{Y,z}$. We claim that there exists a $\theta_2\in \mathbb{T}^1_Y$ whose image in $T^1_{Y,x}$ is $a_x\varepsilon_x$ for every $x\in Z''$. Assuming this, a general linear combination $t_1\theta_1 + t_2\theta_2$ is an element of $\mathbb{T}^1_Y$ whose projection to $\bigoplus_{z\in Z}T^1_{Y,z}$ is a generator of $T^1_{Y,z}$ for every $z\in Z$ as claimed.
 
 To see the claim, as in   Lemma~\ref{liftable} we have an exact sequence 
 $$\mathbb{T}^1_Y = H^1(V; \Omega^{n-1}_V)\to H^2_{E'}(\hY; \Omega^{n-1}_{\hY}) \oplus H^2_{E''}(\hY; \Omega^{n-1}_{\hY})\to H^2(\hY; \Omega^{n-1}_{\hY}).$$
 It suffices  to find an element in $\bigoplus_{x\in Z}H^2_{E_x}(\hX_x; \Omega^{n-1}_{\hX_x}) = H^2_{E'}(\hY; \Omega^{n-1}_{\hY}) \oplus H^2_{E''}(\hY; \Omega^{n-1}_{\hY})$ whose second component is $\sum_{x\in Z''}a_x\varphi(\varepsilon _x)$ and which maps to $0$ in $H^2(\hY; \Omega^{n-1}_{\hY})$.   We have  $\sum_{x\in Z''}a_x\varepsilon _x \in \Gr_F^{n-1}H^n(L'')$ and its image in $H^{n+1}(\hY)$ is contained in $H^2(\hY; \Omega^{n-1}_{\hY}) = \Gr_F^{n-1}H^{n+1}(\hY)$.  By Theorem~\ref{maintheorem}(vi),  the image of $H^2_{E'}(\hY; \Omega^{n-1}_{\hY}) =  H^2_{E'}(\hX; \Omega^{n-1}_{\hX})$ in $H^{n+1}_{E'}(\hY)$ contains $\Gr_F^{n-1}H^{n+1}_{E'}(\hY)$. Thus, by strictness, if we can show that, assuming the hypothesis of Theorem~\ref{smoothCY}, the image of $\sum_{x\in Z''}a_x\varepsilon _x$ in $H^{n+1}(\hY)$ is contained in the image of $H^{n+1}_{E'}(\hY)$, we will be done. 
 
 Via Poincar\'e duality, the map $H^n(L'') \to H^{n+1}(\hY)$ is identified with the map $H_{n-1}(L'') \to H_{n-1}(\hY)$, and the map $H^{n+1}_{E'}(\hY) \to  H^{n+1}(\hY)$ is identified with the map $H_{n-1}(E') \to H_{n-1}(\hY)$. So we must show that the image of $\sum_{x\in Z''}a_x\varepsilon _x$ in $H_{n-1}(\hY)$ is contained in the image of $H_{n-1}(E')$.  Taking the exact sequence of the pair $(\hY, E')$ gives an exact sequence
 $$H_{n-1}(E') \to H_{n-1}(\hY) \to H_{n-1}(\hY, E').$$
 Also, $H_{n-1}(\hY, E') \cong H_{n-1}(Y', Z')$, and $H_{n-1}(Y', Z') \cong H_{n-1}(Y')$ since $n \ge 3$.  In particular, there is a natural injection  $H_{n-1}(\hY)/\im H_{n-1}(E') \to H_{n-1}(Y')$, and $\varphi$ is the composition 
 $$ H^n(L'') \cong  H_{n-1}(L'') \to H_{n-1}(\hY) \to H_{n-1}(\hY)/\im H_{n-1}(E') \to H_{n-1}(Y').$$ 
 Thus, the hypothesis $\sum_{x\in Z''}a_x\varphi(\varepsilon _x) =0\in H_{n-1}(Y')$ implies that the image of $\sum_{x\in Z''}a_x\varepsilon _x$   is contained in the image of $H_{n-1}(E')$, as claimed. 
 
 The fact that a small deformation $Y_t$ of $Y$ has trivial canonical bundle then follows easily from the assumption that $ h^1(Y; \scrO_Y) = 0$.  
 \end{proof}
 
 \begin{remark} It is easy to check by unwinding the above argument that (\ref{localobs}) is a necessary and sufficient condition for the existence of first order smoothings.
 \end{remark}

 \begin{corollary}\label{3CYsmoothing} Let $Y$ be a a canonical Calabi-Yau variety of dimension $3$ with isolated hypersurface singularities and such that $H^1(Y;\scrO_Y) =0$. 
 \begin{enumerate} \item[\rm(i)]   $Y$ can be deformed to a canonical Calabi-Yau threefold whose only singularities are ordinary double points. More precisely, there exist small deformations of $Y$ which are locally trivial at all of the ordinary double points of $Y$ and smooth the remaining singularities, and which have trivial dualizing sheaves. 
 \item[\rm(ii)] Let $\overline{Y}' \to Y$ be a small resolution of  $Y$ at the ordinary double points and an isomorphism elsewhere, and, for each $x\in Z''$, let $C_x$ be the corresponding exceptional curve in $\overline{Y}'$. If there exist nonzero $a_x\in \Cee$ such that $\sum_{x\in Z''}a_x[C_x] = 0$ in $H_2(\overline{Y}')$, then $Y$ is smoothable, and all small smoothings of $Y$ have trivial canonical bundle. 
 \end{enumerate}
 \end{corollary}

  \begin{remark} (i) is a somewhat more precise version of \cite[Theorem 2.4]{NS}. (ii) is a strengthening of \cite[Theorem 2.5]{namstrata}, which only proves the result under the assumption that $Y$ has terminal singularities. (However, in \cite{namstrata} there is an if and only if statement.)
 \end{remark}

\begin{proof}[Proof of Corollary~\ref{3CYsmoothing}] By \cite{namtop}, the deformations of $Y$ are unobstructed.

\smallskip
\noindent (i) As in the statement of (ii), let $\overline{Y}' \to Y$ be a small resolution of  $Y$ at the ordinary double points. Then $\overline{Y}'$ is also a canonical Calabi-Yau variety of dimension $3$ with isolated singularities, and all singularities of $\overline{Y}'$ are strongly $1$-irrational.  Hence, by Theorem~\ref{smoothCY}, $\overline{Y}'$ is smoothable. Let $\overline{\mathcal{Y}}'\to \Delta$ be a smoothing, with general fiber $\overline{Y}'_t$. In the  smoothing $\overline{\mathcal{Y}}'$, the exceptional curves $C_x$ are rigid for small $t$. Thus we may contract them to obtain a deformation of $Y$ which smooths all of the singularities which are not ordinary double points. Since $\overline{Y}'_t$ has trivial canonical bundle, the dualizing sheaf of the contraction is trivial as well. 

\smallskip
\noindent (ii) Let $\overline{Y}'$ be as above, and let $Y' \to \overline{Y}'$ be  the blowup of $Y'$ along the curves $C_x$.  A standard argument identifies the condition that $\sum_{x\in Z''}a_x\varphi(\varepsilon _x) =0\in H_{n-1}(Y')$ with the condition that $\sum_{x\in Z''}a_x[C_x] = 0$ in $H_2(\overline{Y}')$. Hence $Y$ is smoothable by Theorem~\ref{smoothCY}.  
\end{proof}

A similar argument, using Theorem~\ref{unob1}, shows the following:

 \begin{corollary}\label{nCYsmoothing} Let $Y$ be a a canonical Calabi-Yau variety of dimension $n$ with isolated hypersurface singularities and such that $H^1(Y;\scrO_Y) =0$, and suppose that all singularities of $Y$ are $1$-liminal. Let $\hY \to Y$ be a good resolution of  $Y$ at the points of $Z = Z''$. With $\iota \colon   H^n(L) \to H^{n+1}(\hY)$ the natural map, suppose  that, for every $x\in Z$, there exists a nonzero $a_x\in \Cee$   such that 
 $$\sum_{x\in Z}a_x\iota(\varepsilon _x) =0\in H^{n+1}(\hY).$$ Then $Y$ is smoothable, and all small smoothings of $Y$ have trivial canonical bundle. \qed
 \end{corollary} 

 \begin{remark} In case $\dim Y = 3$ and all singularities are ordinary double points, then by \cite{Fddbar}, a nonempty open subset of smoothings of $Y$ will satisfy the $\partial\bar\partial$-lemma. It is natural to expect this to be true in greater generality.
 \end{remark}

\bibliography{cyref}
\end{document}